\documentclass[english]{article}
\usepackage{custom_tex}
\usepackage{graphicx}
\usepackage{grffile}
\usepackage{paralist}
\usepackage{booktabs}       
\usepackage{multicol}
\usepackage{multirow}
\usepackage{makecell, rotating, diagbox}
\usepackage{slashbox}
\usepackage{array}
\usepackage[export]{adjustbox}
\usepackage{subcaption}
\usepackage{tikz}
\usetikzlibrary{shapes,positioning}

\providecommand{\keywords}[1]
{
  \small	
  \textbf{\textit{Keywords---}} #1
}

\allowdisplaybreaks

\title{How to reduce dimension with PCA and random projections?}

\author{
Fan Yang,\footnote{Wharton Statistics Department, University of Pennsylvania. E-mail: \texttt{fyang75@wharton.upenn.edu}.}
\,
Sifan Liu,\footnote{Department of Statistics,
Stanford University. E-mail: \texttt{sfliu@stanford.edu}.}
\,
Edgar Dobriban,\footnote{Wharton Statistics Department, University of Pennsylvania. E-mail: \texttt{dobriban@wharton.upenn.edu}.}
\,
and David P. Woodruff\footnote{Department of Computer Science, Carnegie Mellon University. E-mail: \texttt{dwoodruf@cs.cmu.edu}.}
}

\begin{document}

\maketitle

\begin{abstract}
In our ``big data" age, the size and complexity of data is steadily increasing. Methods for dimension reduction are ever more popular and useful. Two distinct types of dimension reduction are ``data-oblivious" methods such as random projections and sketching, and ``data-aware" methods such as principal component analysis (PCA). Both have their strengths, such as speed for random projections, and data-adaptivity for PCA. In this work, we study how to combine them to get the best of both. We study ``sketch and solve" methods that take a random projection (or sketch) first, and compute PCA after. We compute the performance of several popular sketching methods (random iid projections, random sampling, subsampled Hadamard transform, count sketch, etc) in a general ``signal-plus-noise" (or spiked) data model. Compared to well-known works, our results (1) give asymptotically exact results, and (2) apply when the signal components are only slightly above the noise, but the projection dimension is non-negligible. We also study stronger signals allowing more general covariance structures. We find that (a) signal strength decreases under projection in a delicate way depending on the structure of the data and the sketching method, (b) orthogonal projections are more accurate, (c) randomization does not hurt too much, due to concentration of measure, (d) count sketch can be improved by a normalization method. Our results have implications for statistical learning and data analysis. We also illustrate that the results are highly accurate in simulations and in analyzing empirical data.
\end{abstract}

\section{Introduction}

In our ``big data" age, the size and complexity of data is steadily increasing. Methods for data reduction are used ever more commonly. Among these, dimension reduction methods are used to summarize many features into a small set \citep[see e.g., the reviews][and references therein]{jolliffe2002principal,burges2010dimension,cunningham2015linear}.

Two prominent and very different classes of dimension reduction exist: \emph{data-oblivious} methods such as random projections and sketching \citep[e.g.,][etc]{vempala2005random,mahoney2011randomized,woodruff2014sketching,erichson2016randomized}, and \emph{data-aware} methods such as principal component analysis (PCA) \citep[see e.g., the textbooks and reviews][and references therein]{anderson1958introduction,jolliffe2002principal,fan2014principal,johnstone2018pca}. This is of course just one way to classify the different methods, as there are also linear and nonlinear approaches, etc.

Both data-oblivious and data-aware methods have their strengths. Data-oblivious methods can be very fast and convenient to implement. 
Data-aware methods on the other hand can better exploit the structure of the data; e.g., PCA can be statistically optimal under certain conditions \citep[e.g.,][]{anderson1958introduction}.

{
\begin{figure}[!h]
\centering
\begin{tikzpicture}[
            > = stealth, 
            shorten > = 1pt, 
            auto,
            thick 
        ]

\tikzstyle{ann} = [draw,fill=none, minimum size=3em]
\begin{scope}[node distance = 0.5cm and 2cm]
        \node[ann,rectangle] (a) {Data matrix $Y$};
        \node[ann,rectangle] (b) [right=of a] 
        {\begin{tabular}{c} 
        	Sketched data $\wt Y = SY$
        	\end{tabular}};
        \node[ann,rectangle] (c) [right=of b]
        {\begin{tabular}{c} 
         PCA $\wt Y = VDU^\top$
        	\end{tabular}};           
        \path[->] (a) edge node {} (b);
        \path[->] (b) edge node {} (c);
\end{scope}
    \end{tikzpicture}
    \caption{Flowchart of the algorithm we analyze: We sketch or project the data $Y$ into $SY$, then perform PCA.
}
\label{fig-flow}
\end{figure}
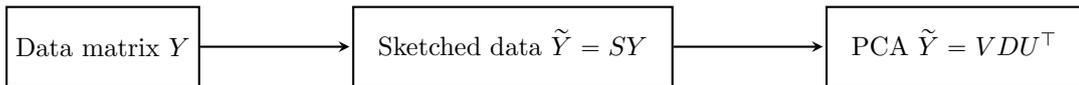
}

In this work, we study how to \emph{combine} them to get the best of both. We study ``sketch and solve" methods that take a random projection (or sketch) first, and compute PCA after (see Figure \ref{fig-flow}). Various versions of such algorithms have been proposed  \citep{rokhlin2009randomized,halko2011algorithm,halko2011finding} (see the related work section). In applied areas, such methods are starting to be used in economics \citep{ng2017opportunities}, forecasting \citep{schneider2016forecasting} and genomics \citep{galinsky2016fast}. 
In particular, such algorithms are ``state of the art" for dealing with extremely large genomics datasets, where the number of samples (people) is in the thousands to hundreds of thousands, and the number of features (genetic variants, basepairs) is on the order of millions to billions \citep{galinsky2016fast}.

However, it is not well understood how they perform in all regimes. How can we choose the dimension of sketches? What sketching method---e.g., subsampling or random Gaussian projections---to use? How does their performance depend on the characteristics of the data? Increasing the dimension always increases the accuracy. But that comes with an increased computational and memory cost. While the existing works do provide theoretical guarantees, they leave some regimes unstudied \citep{rokhlin2009randomized,halko2011algorithm,halko2011finding}. They typically focus on the regime where the signal components are ``dominant", and the sketching dimension is very small. 
As we will see, in a natural model of low-rank data we can get precise results even when the signal components are barely above the noise level, if the sketching dimension is sufficiently large. 

In our work we take a systematic approach to this problem.
We compute the performance of the most popular sketching methods, such as uniform orthogonal projections, random projections with iid entries, random sampling of the datapoints, subsampled Hadamard transform \citep{sarlos2006improved,ailon2006approximate}, and count sketch \citep{charikar2002finding,clarkson2017low}. We work in a general ``low-rank signal plus noise" model, sometimes called the ``spiked model", which has been widely used to study PCA \citep[e.g.,][etc]{spikedmodel,couillet2011random,yao2015large,johnstone2018pca,gavish2014optimal,nadakuditi2014optshrink,gavish-donoho-2017,dobriban2017sharp,dobriban2017factor,dobriban2019deterministic}. {In the spiked model, the entries of the noise matrix are independent random variables, and the signal is an arbitrary low rank matrix that is independent of the noise.} This is also a special type of a linear factor model.

Compared to well-known classical works \citep{rokhlin2009randomized,halko2011algorithm,halko2011finding}, our results (1) give asymptotically exact results under more specific assumptions, and (2) apply when the signal components are arbitrarily close to the noise level, provided the sketching dimension is sufficiently large.  We are able to accomplish that by building on the analysis of recent works in random matrix theory, such as \citep{yang2019spiked,ding2019spiked}. We also study stronger signals allowing more general covariance structures. In addition, we illustrate that the results are accurate in simulations and in analyzing empirical data. The computer code reproducing the numerical results in the paper is available from \url{https://github.com/liusf15/sketching-svd}.

\subsection{Related work}
\label{relw}
In this section we review some related work. Due to space limitations, we can only consider the most closely related work. For overviews of sketching and random projection methods, we refer the reader to \cite{vempala2005random,halko2011finding,mahoney2011randomized,woodruff2014sketching,drineas2017lectures}.  A cornerstone result is the Johnson-Lindenstrauss lemma. This states that norms, and thus also relative distances between points, are \emph{approximately} preserved after sketching i.e., $(1-\delta)\|x_i\|^2\leq\|Sx_i\|^2\leq(1+\delta)\|x_i\|^2$ for $x_1,\ldots,x_n\in\mathbb{R}^p$. {This is further extended to the {\it subspace embedding property}, that is, for all $x$ in a subspace of relatively small dimension, the norm of $x$ is preserved up to a $\delta$ factor.} Each projection studied in this paper has the embedding property, and this can be used to derive bounds for the accuracy of PCA. However, our results are much more refined, because they quantify the precise value of the error in an asymptotic setting, while the bounds above are inequalities up to the constant $\delta$.

Compared to well-known classical works \citep{rokhlin2009randomized,halko2011algorithm,halko2011finding} on random projection + PCA, our results are in a different data model. Our results give asymptotically exact results when the sample size $n$ and dimension $p$ increase to infinity at the same rate in a spiked model, while the previous results are bounds up to constants. Our results are accurate in simulations and are sharp even when the signal components are only slightly above the noise. 
{  In additional related work, \cite{homrighausen2016nystrom} study the Nyström and column-sampling methods for approximate PCA.}

For instance, a typical result, Theorem 1.1 in \cite{halko2011finding} states that if $Q$ is the $r\times p$ orthogonal matrix projecting into the row space of $SX$, then 
$$\E \|X (I-Q^\top Q)\|\le \left [1+ 4 \frac{\sqrt r}{r-k-1} \sqrt{\min{(n,p)}}\right]\sigma_{k+1}.$$ Here $\sigma_{k+1}$ is the $(k+1)$-st singular value of $X$. These bounds are sharp if the $(k+1)$-st singular value is small, and $r$ is slightly larger than $k$. In contrast, our bounds are also applicable to the setting where the $(k+1)$-st singular value of $X$ is only slightly smaller than the $k$-th one, but we take $r$ to be much larger than $k$. Thus, our results cover a different regime. {  In more detail, we consider the regime where $n,p,r$ are all large and proportional to each other, while $\sigma_{k+1}$ is lower bounded by a constant. In this case, the above bound is also lower bound by a constant, and hence has limited implications.}

Another comparison to prior work is that worst-case bounds for CountSketch are significantly weaker \citep{clarkson2017low}, whereas here we get much tighter bounds. For instance, we can effectively show that count-sketch reduces the signal strength by a factor of approximately $\zeta_n (1-\exp(-\zeta_n))$, where $\zeta_n=r/n$ is the ratio of sketched and original sample size. Our bounds more accurately model what is observed in experiments.

However, our results only concern one-step ``sketch-and-solve" methods, while there are also other more sophisticated methods. For instance, \cite{frieze2004fast} proposed randomized SVD using non-uniform row and column sampling. {  After early work by \cite{sarlos2006improved,liberty2007randomized,halko2011algorithm} introducing methods
based on random projections, \cite{halko2011finding} developed a unified framework, including iterative algorithms.}  \cite{woolfe2008fast} improved the speed via fast matrix multiplications on structured matrices. {Musco and Musco proposed a Randomized Block Krylov Iteration methods for fast SVD \cite{Musco2015}.}  \cite{tropp2017practical} studied the scenario where we can only access the $A$ via a linear map $SA$. In future work, it will be interesting to extend our approach to those algorithms. \cite{7008533} studies how to recover a sparse matrix $X$ from observations $AXB$.

Random projection based approaches have been studied for other problems too, including linear regression \citep{sarlos2006improved,drineas2011faster,raskutti2014statistical,dobriban2018new}, ridge regression \citep{lu2013faster,chen2015fast,wang2017sketched,liu2019ridge}, two sample testing \citep{lopes2011more,srivastava2016raptt}, classification \citep{cannings2017random}, convex optimization \citep{pilanci2015randomized,pilanci2016iterative,pilanci2017newton}, etc, see \cite{woodruff2014sketching} for a more comprehensive list of applications.

{  Compared with prior theoretical work on one-step sketching in linear regression \citep{dobriban2018new}, our perspective is similar, in that we want to develop a unified framework to analyze and compare the statistical performance of various sketching methods. In addition, some of the conclusions are also consistent: orthogonal sketches are more accurate, and subsampled randomized Hadamard transform is the best overall method. However, the similarity stops there. First, this paper is about PCA, a different problem from prior work on linear regression in \citep{dobriban2018new}. Second, this paper considers a theoretical approach leveraging recent results on local laws in random matrix theory \citep{yang2019spiked,ding2019spiked}, while the prior work \citep{dobriban2018new} uses tools such as properties of Wishart matrixes, the generalized Lindeberg principle \citep{chatterjee2006generalization}, and liberating sequences from free probability \citep{anderson2014asymptotically}. Thus, the tools are quite different. Finally, an additional difference is that this paper also discusses CountSketch, which has not appeared in \citep{dobriban2018new}.}

\section{Sketching in PCA}\label{sec sketchPCA}

In this section we explain our main results. We have an $n\times p$ data matrix $Y$ containing $p$ features of $n$ data points, such as $p$ different measurements on $n$ sensors. We want to perform an approximate PCA of the data. To study the performance of dimension reduction methods, we assume that the data follows ``signal-plus-noise" or ``spiked covariance" matrix model \citep[e.g.,][etc]{spikedmodel,couillet2011random,paul2014random,yao2015large,johnstone2018pca}:
$$ Y =WDU^{\top} +  X = \sum_{i=1}^k d_i w_iu_i ^{\top}+ X.$$
Here $WDU^{\top}=\sum_{i=1}^k d_i w_iu_i ^{\top}$ is the signal component, $\{d_i\}_{1\le i \le k}$ are the signal strengths (also known as population spikes), and $\{w_i\}_{1\le i \le k}$ and $\{u_i\}_{1\le i \le k} $ are the left and right singular vectors of the signals, respectively. They are arranged into the left and right matrices of eigenvectors $W$ and $U$ ($n\times k$ and $p\times k$), and the $k\times k$ diagonal matrix $D$ of population spikes. The matrices $U,W$ are orthogonal: $U^{\top}U = W^{\top} W=I_k$.
On the other hand, $X$ is the noise component, where the entries $x_{ij}$, $1 \leq i \leq n$, $1 \leq j \leq p$, are real independent random variables with zero mean and variance
$\mathbb{E} \vert x_{ij} \vert^2  = n^{-1}$.
We assume that any randomness in the signal is independent of the noise matrix $X$. {Other than that, the signal strengths and the singular vectors $w_i$ and $u_i$ can be completely arbitrary.} Such signal plus noise or spiked models have been widely studied. When $w_i$ have iid entries, this model can be viewed as a specific \emph{factor model}, and thus has a long history see e.g., \cite{spearman1904general,thurstone1947multiple,anderson1958introduction}. 

We consider a setting with large sample size $n$ and dimension $p$. We place ourselves in a setting where doing a full PCA on $Y$ is too expensive. As an alternative, we are instead interested in PCA on the \emph{sketched} data matrix 
\beqs\wt Y=  S Y \eeqs
where $S$ is an $r\times n$ ($r<n$) random \emph{sketching matrix} that is independent of both the signal and the noise. This can be written as 
\beqs 
  \wt Y =  VDU^{\top} + S X  = \sum_{i=1}^k d_i v_iu_i ^{\top} + \wt X, \quad \text{where } \ V:= SW, \ v_i:= Sw_i, \ \wt X:= SX.
\eeqs 
A similar spiked separable model has been studied in \cite{ding2019spiked}, although the setting there is somewhat different, because the spikes are added to the population covariance matrices instead of to the data. However, we will build on their analysis in our work.

\subsection{Heuristics}
\label{heur}

Here we explain heuristically what the expected behavior of sketching should be. For simplicity, we consider a one-spiked case, and write the data as $Y = d \cdot w u^{\top}+ X$. Let $S$ be an $r \times n$ partial orthogonal matrix such that $SS^\top=I_r$. Then, we have 
$SY = d\cdot Sw u^{\top} + SX$.
Suppose $X$ has iid Gaussian entries with mean zero and variance $n^{-1}$. Then $\wt X= SX$ also has iid Gaussian entries. After projection, the distribution of the noise is unchanged. The low-rank signal changes from $d \cdot w u^{\top}$ to $\widetilde  d \cdot v u^{\top}=  d \cdot Sw u^{\top}$, where $\widetilde  d := d\cdot \|Sw\|$, $v: = Sw/\| Sw\|$. 

Given the orthogonal invariance of the noise, only the singular values---and not the singular vectors---of the signal govern the behavior of the SVD of the data (i.e., signal-plus-noise). Thus, sketching effectively changes $n\to r$, $d\to d\cdot \|Sw\|$. These fully describe the effect of the projection matrix (which in this case was deterministic). Since $\|Sw\|\le \|S\|\|w\|=1$, both the sample size and the signal strength get reduced. 

However, since we do not know $w$ or $\|Sw\|$, we cannot use the above results to quantify or get insight into the reduction in signal strength. Taking a random $S$ allows us to characterize average behavior, and thus to get useful predictions about the behavior of the algorithm. Suppose $S$ is an $r \times n$ random partial orthogonal matrix, i.e., $S$ is uniformly random over the set of matrices such that $SS^\top=I_r$. Then we expect that the norm of $v = S w$ is $\|v\| \approx (r/n)^{1/2} \|w\|$. This is because we can construct $v$ by randomly rotating $w$ and choosing its first $r$ coordinates. A random rotation makes all coordinates exchangeable, and thus choosing the first $r$ will approximately capture about $r/n$ of the squared norm of $w$.

Let us write $\xi_n = r/n$ for the reduction in sample size due to sketching. The matrix $\xi_n^{-1/2} \wt X$ has iid entries of variance $r^{-1}$. Then the projected matrix $SY$ should be \emph{equivalent to a spiked model with the same spike strength but in a reduced dimension $r$}:
\be\label{Heuristicf}\xi_n^{-1/2} \wt Y  = d \cdot v u^{\top}     +   \xi_n^{-1/2}\wt X.\ee
Heuristically, after projection into $r$-dimensional space, both the sample size and the signal strength go down by a factor of $\xi_n  = r/n$. We will later show rigorously that this is indeed true.  For higher dimensional signals, the sketched signal no longer has orthonormal columns, and so the singular values of the signal slightly change. However, since we are dealing with the 1-dimensional case in this section, we do not need to worry about this. This shows how taking a random $S$ can simplify the results. More generally, without Gaussian noise and for other sketching methods, the randomness in $S$ becomes even more crucial to get interpretable results.

\subsection{Key takeaways}

\begin{table}[]
\renewcommand{\arraystretch}{1}
\centering
\caption{Informal summary of some of our results. For simplicity, suppose we have a {  single-spike} model $Y = d \cdot wu^\top + X$ of size $n \times p$, where $d \cdot wu^\top$ is the signal and $X$ is the noise. We do PCA after sketching on data $SY$, where $S$ is an $r\times n$ sketching matrix. We show the effective decrease of the signal strength 
due to sketching. The assumptions needed on $X$ and $w$ depend on the sketching method. The results for iid random $S$ are involved and only presented in the text.   Finally, for a more general multi-spike model $Y =  \sum_{i=1}^k d_i w_iu_i ^{\top} + X$, we have similar results for the eigenvalues corresponding to each single spike $d_i w_i u_i^\top$.\nc}
\label{summaryresults}
\begin{tabular}{>{\centering\arraybackslash}m{1.75cm}>{\centering\arraybackslash}m{2cm}>{\centering\arraybackslash}m{2cm}>{\centering\arraybackslash}m{5cm}>{\centering\arraybackslash}m{2cm}}
\toprule[1pt]
Assumption on $X$ & Gaussian & independent entries & independent entries & independent entries \\
\midrule[1pt] 
Assumption on $S$ & {  Partial orthonormal} & Haar/ Hadamard 
& 
\begin{tabular}{@{}c@{}}
Uniform sampling (US)/ 
\\
CountSketch (CS)
\end{tabular}
& iid random \\
\midrule[1pt]
Assumption on $w$ & Fixed & Fixed & Delocalized & Fixed \\
\midrule[1pt]
Effect on signal 
& $d\to d \cdot \|Sw\|$    
& $d\to d \cdot \sqrt{r/n}$ 
& 
\begin{tabular}{@{}c@{}}
US: $d\to d \cdot \sqrt{r/n}$ 
\\
CS: $d\to d \sqrt{r/n(1-\exp(-n/r))}$ 
\end{tabular}
& see Theorem \ref{sketchthm4}  
\\
\bottomrule[1pt]
\end{tabular}
\end{table}
We summarize our key takeaways as follows. Clearly, the signal strength goes down under projection, and the amount of decrease depends on the type of projection. Moreover:

\benum

\item {\bf Separations between sketching methods}:  Our analysis reveals precise separations: subsampled randomized Hadamard transform (SRHT) and subsampling are more accurate than CountSketch, which is typically more accurate than projections with Gaussian or iid entries. The superiority of orthonormal projections is consistent with previous observations in different contexts \citep{dobriban2018new,lacotte2020limiting}, but our work goes much beyond to include CountSketch and also considers a different problem.

\item {\bf Precise quantitative results}: Our results precisely quantify the locations of the sketched spikes. See Table \ref{summaryresults} for an informal summary of some of our results. We show the effective decrease of the signal strengths (i.e., spike strength $d$) for various sketching methods. However, we state our formal results in terms of the empirical eigenvalues and eigenvectors, because for some cases (especially for projections $S$ with iid entries), there seems to be no simple way to state them in terms of the decrease in signal strength.

For large signal strengths, we can handle general noise covariance structures, and get simpler results (cf. Section \ref{strong}).

\item {\bf Additional randomness does not hurt}: A key limitation and drawback of randomized algorithms is that they introduce additional variability in the data. This is an undesireable phenomenon, because the additional variability may lead to vastly different results every time the algorithm is run, and may reduce reproducibility. 

In our case, we see that the top eigenvalues and correlations between true and empirical eigenvectors are asymptotically concentrated around definite limits. This means that the additional randomness introduced by the sketching algorithm is relatively limited, for large data sets and for those particular functionals. However, we should still be cautious, in particular about interpreting results obtained from other functionals.

\item {\bf Implications for learning}: Our results have implications for statistical learning and signal processing. In particular, by following the methods from \citep{donoho2018,donoho2018optimal}, they can be used to derive optimal eigenvalue shrinkage estimators for the covariance matrix. Recall that the optimal shrinkage operations depend on the overlap between true and empirical eigenvectors. We find those formulas for various sketching methods, and so it becomes possible to use the shrinkage formulas. 

Very briefly, \cite{donoho2018} estimate the covariance matrix $\Sigma$ optimally using eigenvalue shrinkage estimators of the sample covariance matrix. We can replace this with the covariance matrix $\widehat \Sigma_s = r^{-1} \widetilde  Y^\top \widetilde  Y$ of the sketched data. Let $\widehat \Sigma_s = \sum_{i=1}^{\min(r,p)}\widetilde \lambda_i \wt{\bxi}_i\wt{\bxi}_i^\top$ be the spectral decomposition of $\widehat \Sigma_s$, with $\widetilde \lambda_i$ sorted in non-increasing order. Then we consider eigenvalue estimators $\widehat \Sigma_\eta = \sum_{i=1}^{\min(r,p)}\eta(\widetilde \lambda_i) \wt{\bxi}_i\wt{\bxi}_i^\top$ for some fixed shrinker $\eta:\R\to\R$. We evaluate the estimator based on a loss $L(\widehat \Sigma_\eta,\Sigma)$, where $\Sigma = I_p + \sum_{i=1}^k d_i^2 u_iu_i ^{\top}$ is the covariance matrix of the original data. For instance, we can have $L(A,B)=\|A-B\|_{\op}$ be the operator norm loss, or $L(A,B)=\|A-B\|_{\Fr}$ be the Frobenius norm loss. 

Based on the theory from \cite{donoho2018}, we can deduce that there is an asymptotically optimal shrinker $\eta$ for these losses (and a number of others). For instance, for uniform orthogonal random projection, uniform sampling and subsampled randomized Hadamard transform,  the optimal shrinkers for operator and Frobenius losses are, respectively,
\beqs
\eta_{\op}(x)=\lambda^{-1}(x^2, r/n),\,\,\quad \eta_{\Fr}(x)=\lambda^{-1}(x^2, r/n)\cdot c^2(\lambda^{-1}(x^2, r/n), r/n)+ s^2(\lambda^{-1}(x^2, r/n), r/n).
\eeqs
Here $\lambda$ is the functional inverse of the spike forward map from equation \eqref{outlierevalue}, and its expression can be found in \cite{donoho2018}, as well as implemented in software in \cite{Dobriban2015}. Also, $c^2$ is the cosine forward map from \eqref{outlierevector}, and $s^2$ is the squared sine, defined as $s^2 = 1-c^2$.



\eenum

{  \subsection{How to use our results?}

In this section, we give some additional illustration and guidance on how to use our results. Suppose we are interested to compute the SVD or PCA of a massive dataset. Suppose that we are in a setting where we need to use a single machine (possibly after dividing up the data into smaller pieces and distributing them onto different machines---our results can be used at various steps of a broader processing pipeline). Then, a natural approach may be to subsample the $n \times p$ data matrix to $r \le n$ samples. However, this has a chance to miss some datapoints with large entries. Fortunately, there are sketching methods that take linear combinations of each data point, and are thus more likely to pick up these large entries. Which sketching method to use and what projected dimension do we need to get a desired accuracy? What is the appropriate notion of accuracy?

Using our results, we can give some insight into this. First, we suggest that we can use a statistical notion of accuracy. Suppose the data is noisy, and we believe that the empirical principal components (PCs) are only estimators of ``true" unobserved PCs that one could recover from much more data. Then it makes sense to consider how much sketching reduces the ``signal strength" of the PCs in the data at hand. Intuitively, by subsampling $r$ out of $n$, the signal strength should decrease by a factor of $r/n$. It turns out this intuition is correct, but not at all trivial: it only holds when the ``true" principal components are suitably ``non-sparse", and requires a somewhat delicate argument. Thus, confirming our intuition, subsampling is only guaranteed to work in a suitable non-sparse setting. However, we show that other orthogonal sketches enjoy the same signal strength reduction, while also working under sparsity. Moreover, some orthogonal sketches, such as randomized Hadamard/Fourier sketches, can be applied in nearly the same time as subsampling. In addition, popular non-orthogonal sketches such as Gaussian projections have strictly worse signal preservation properties than orthogonal ones. Thus, fast orthogonal sketches emerge as the best choice. While this may not be extremely surprising based on prior work, we do believe that it is not commonly discussed in the literature; and in fact we are not aware of a specific work that makes this point for sketched PCA. {  Furthermore, in Section \ref{sec count}, we propose a normalized version of CountSketch, which modifies the original CountSketch. Our results suggest that this normalized version has slightly better signal preservation properties than the un-normalized version. Hence our work can be a guide as to when to use the newly proposed normalized CountSketch.}

Our results can be also used as a guide to choose the projection dimension. First, suppose we decide that we can tolerate at most a certain factor $f<1$ (say $f=1/2$) of decrease in the signal strength. Then, one should use projection dimension $r$ such that $r  = f n$ (say $r=n/2$). To compute the estimation error for estimating the true PCs, we can simply use the well-known formulas from spiked covariance models, see Section \ref{Gaussian}.  This illustrates how we may use our results. In addition, we believe that we can use our results as a tool to develop and analyze more complicated data analysis methods. However, this is beyond the scope of our current work.
}

\subsection{Details}

Our results require a few more technical assumptions, which are stated in detail in Section \ref{sec_defspike}. We use the notion of empirical spectral distribution (ESD) of a matrix $M$, which is the empirical distribution function of the eigenvalues of $M$.

In the end, we obtain the following steps for finding the values of the spikes of the sketched matrix $\wt  Y = SY$ (recall $S$ is $r \times n$, $r\le n$):
\begin{enumerate}
\item For any $x$, we find a fundamental quantity, the pair of weighted \emph{Stieltjes transforms} $(m_{1c}(x),m_{2c}(x))$, as the solution of a certain system of \emph{self-consistent equations} \eqref{separa_m12}. Recall that for a distribution $F$, its Stieltjes transform is defined for any $z\in \mathbb{C}$ away from the support as $m_F(z) = \E_{X\sim F} (X-z)^{-1}$ \citep[e.g.,][]{bai2009spectral,couillet2011random,yao2015large}. In our case, $(m_{1c},m_{2c})$ are the classical limits of certain weighted Stieltjes transforms $(m_1,m_2)$ of the ESDs of $\wt X^{\top}\wt X$ and  $\wt X\wt X^{\top}$ for $\wt X= SX$ (see Section \ref{resll}); and their importance, described below, is in how to use them. 

Let us denote by $\gamma_n=p/n$ the aspect ratio, by $\xi_n=r/n$ the sample size reduction factor ($<1$), and by $\pi_{B}:= \frac{1}{r} \sum_{i=1}^r \delta_{s_i}$ the empirical spectral distribution of $B=SS^{\top}$. The self-consistent equation shows that for any $z\in \mathbb{C_+}$ (complex numbers with positive imaginary parts), $(m_{1c}, m_{2c})$ are determined by the following pair of equations:
\beq\label{sc}
\begin{split}
& {m_{1c}(z)} =\gamma_n\frac{1}{-z\left[1+m_{2c}(z) \right]},\\
& {m_{2c}(z)} = \xi_n \int\frac{x}{-z\left[1+xm_{1c}(z) \right]} \pi_B (\dd x) .
\end{split}
\eeq
This is a \emph{general Marchenko-Pastur} or Silverstein equation; and can also be expressed as a fixed point equation for $m_{1c}$. It can be solved explicitly in certain special cases. There are also fast numerical solvers available, based on fixed-point methods and ODE solvers see e.g., \cite{couillet2011deterministic,dobriban2015efficient,cordero2018numerical}. {  In general, this is one of the two mathematically challenging parts of the algorithm (i.e., of using these steps presented here to figure out the behavior of the spikes after sketching).}
\item We combine the above quantities into the $2k\times 2k$ \emph{master matrix} $M(x)$
\be\label{defnMx}
M(x) = 
\left( {\begin{array}{*{20}c}
   -x^{-1/2}(1+m_{2c}(x))^{-1}I_k & {D^{-1}}   \\
   {D^{-1}} & - x^{-1/2} W^{\top}S^{\top} (1+m_{1c}(x)SS^{\top})^{-1}SW  \\
   \end{array}} \right).
\ee

\item We solve for the values $x$ for which this matrix is singular, i.e., 
\be\label{eme}
\det M(x)=0.
\ee
We call this the \emph{eigenvalue master equation}.

Such a condition has appeared in many cases in the literature \citep[e.g.,][and references therein]{couillet2011random,yao2015large}. In general we expect at most $k$ solutions $x$. The theory guarantees that these are all possible candidates for the empirical spikes of the sketched data $\wt  Y^{\top}  \wt  Y $. This step turns out to become feasible in several applications due to the randomness in either the sketching matrix $S$ or the signal matrix $W$. This randomness causes the lower right block to become diagonal, and hence, after rearrangement, the matrix $M$ can be studied as a block matrix with $2\times 2$ blocks.

\item To get the angles between the eigenvectors corresponding to an eigenvalue $\wt\lambda_i$ of $\wt  Y^{\top}  \wt  Y $,  again we follow an approach related to many works in the literature \citep[e.g.,][and references therein]{couillet2011random,yao2015large}. We consider a small contour $\Gamma_i$ which encloses $\wt\lambda_i$ (or its classical limit $\theta_i$, as explained below) but no other eigenvalues of $\wt  Y^{\top}  \wt  Y $. The overlap of the corresponding right singular vector $\wt{\bxi}_i$ with any spike eigenvector $u_j$ of the original data matrix $Y$ is given by the \emph{angle master equation}:
\begin{align}\label{ame}
\begin{split}
|\langle u_j, \wt{\bxi}_i\rangle|^2 & = \frac{1}{2\pi \ii (\wt\lambda_i)^{1/2}} \left( \oint_{\Gamma_i}e_j^{\top}\cal D^{-1}M(z)^{-1} \cal D^{-1}e_j  \dd z\right) .
\end{split}
\end{align}
Again, it turns out that in certain cases we can explicitly calculate these integrals.
\end{enumerate}
This finishes the general description of the procedure for finding the sketched spikes. See Section \ref{sec_defspike} for details. Next we will go over various sketching methods in detail, and show how to use this general procedure.

\section{Types of random projections}\label{sec5types}

In this section, we go over the various types of random projections, and explain the behavior of the sketched eigenvalues and eigenvectors. Fix any signal  strengths $d_1 > d_2 > \cdots > d_k >0$. Without projections, when $S=I_{n}$, it is well-known that a signal of strength $d_i$ leads to an \emph{outlier} if and only if
$d_i^2 > \sqrt{ {\gamma_n} }$ \citep{baik2005phase,baik2006eigenvalues}.
Here outliers are the eigenvalues of the sample covariance matrix separated and above the ``bulk" of the noise eigenvalues, which is described by a standard Marchenko-Pastur distribution \citep{marchenko1967distribution,bai2009spectral}. Moreover, the $i$-th spiked sample eigenvalue converges to its ``classical value" 
\be\label{know111} \wt\lambda_i \to
\lambda(d_i^2,\gamma) :=  \left( 1+d_i^2\right) \left(\frac{\gamma}{d_i^2} + 1 \right) \quad \text{in probability},\ee
if $\gamma_n\to \gamma$, see e.g., \citep{baik2005phase,baik2006eigenvalues}. The map $\ell \to \lambda(\ell, \gamma)$ between the population and sample spikes is sometimes referred to as the spiked forward map.  The overlaps between population and sample eigenvectors converge to \cite[e.g., ][]{Benaych2012}
\be\label{know222}  |\langle u_j, \wt{\bxi}_i\rangle|^2 \to 
\delta_{ij}c^2(d_i^2,\gamma)
:=\delta_{ij} \frac{1 - \frac{\gamma}{d_i^4}}{1+ \frac{\gamma}{d_i^2}} \quad \text{in probability},\ee
{where $\delta_{ij}=1$ if $i=j$, and $\delta_{ij}=0$ otherwise.}
The expression $c^2(d^2,\gamma)$ may be referred to as the squared cosine forward map, giving the asymptotic squared cosines between the population and sample eigenvectors.
Now we consider several choices of $S$, and compare the corresponding results to the above. We restrict to a certain high probability event $\Omega$ (given formally in \eqref{aniso_outstrong}), where the so-called ``local law" holds, and certain empirical quantities are close to their population verions.  So ``with high probability" means with high probability on $\Omega$.

\subsection{Uniform orthogonal random projections}\label{sec unifproj}

\label{sec: orthogonal}

We take $S$ to be $r\times n$ partial orthonormal, such that $S S^{\top}=I_r$. 

\subsubsection{Results known from prior work, Gaussian data}
\label{Gaussian}

There are a few results that can be readily deduced from known work. They are not our main point (as they are limited to Gaussian data); and our main results can handle much more general data distributions and sketching distributions in a unified framework.  However, as they are not available in prior work, we present them here for the reader's convenience. 

When the noise $X$ is has iid Gaussian entries, we have seen in Section \ref{heur} that for fixed dimensions $n$ and $p$, the data $ Y = X + WDU^{\top}$ transforms into $ SY = SX + (SW)DU^{\top}$, with the distribution of $SX$ still being Gaussian. The signals are transformed into $(SW)DU^{\top}$, and we let its SVD be $\wt W\wt  D\wt U^{\top}$. As is well known from the classical theory of spiked models, \citep{baik2005phase}, the singular values of the signal control the behavior the data SVD. This shows that we are in a new spiked model with new signal strengths $\wt D$, which can be checked to have been reduced compared to $D$.

If in addition we assume that $S$ is distributed uniformly over the Stiefel manifold of partial orthonormal matrices, then it is not hard to check that $\widetilde  D \approx (r/n)^{1/2} D$, so the signal reduces by a factor of $r/n$. In addition, the sample size is also reduced. To quantify the change of the outlier eigenvalues and eigenvectors, we recall that after scaling by $\xi_n^{-1/2}$, $ SY$ is equivalent to the model in \eqref{Heuristicf}, which has the same spike strength but in a reduced dimension $r$. In this model, we have that the aspect ratio changes as $\gamma_n\to \frac{n}{r}=\frac{\gamma_n}{\xi_n}$. Thus by \eqref{know111} and \eqref{know222}, if $\gamma_n\to \gamma$ and $\xi_n\to \xi$, then we have
\be\label{know333}\wt\lambda_i \to \theta_i:= \xi \left( 1+d_i^2\right) \left(\frac{\gamma/\xi}{d_i^2} + 1 \right)=  \left( 1+d_i^2\right) \left(\frac{\gamma}{d_i^2} + \xi \right), \quad \text{in probability},\ee
and
\be\label{know444}|\langle u_j, \wt{\bxi}_i\rangle|^2 \to \delta_{ij} \frac{1 - \frac{\gamma/\xi}{d_i^4}}{1+ \frac{\gamma/\xi}{d_i^2}}=\delta_{ij}\frac{\xi - \frac{\gamma}{d_i^4}}{\xi + \frac{\gamma}{d_i^2}}, \quad \text{in probability},\ee
for the eigenvalues and eigenvectors of $SY$. One can compare them to the results in \eqref{outlierevalue} and \eqref{outlierevector}, and see how the location of the spikes decreases.

The same logic applies to all distributions of partial orthonormal sketching matrices $S$ and all signal matrices $W$ for which $(SW)^\top SW \approx r/n \cdot W^\top W = r/n \cdot I_k$. We will discuss this for each case separately. 

\subsubsection{New results, general data}

When the distribution of the data is general, the above direct argument cannot be used. We will instead use our general framework. We assume $S$ is distributed uniformly over the Stiefel manifold of partial orthonormal matrices.

Then the self-consistent equation for the Stieltjes transforms \eqref{sc} becomes
\begin{equation}\label{self_ortho}
\begin{split}
& {m_{1c}(z)} = \gamma_n \frac{1}{-z\left[1+m_{2c}(z) \right]} ,\quad {m_{2c}(z)} =  \xi_n  \frac{1}{-z\left[1+m_{1c}(z) \right]}  .
\end{split}
\end{equation}
We can obtain the following equation for $m_{2c}:=m_{2c}(z)$:
\begin{equation}\label{g_{2c}}
\begin{split}
& \frac{\xi_n}{m_{2c}(z)} =  -z+ \frac{\gamma_n }{1+m_{2c}(z)}   \Rightarrow zm_{2c}^2 + (z-\gamma_n +\xi_n)m_{2c} + \xi_n =0.
\end{split}
\end{equation}
This equation has a unique solution with non-negative imaginary part for $z\in \C_+$, that is, 
\be\label{solv m2c}
m_{2c}(z)=\frac{-(z-\gamma_n +\xi_n) + \sqrt{(z-\lambda_+)(z-\lambda_-)} }{2z}, \quad \lambda_{\pm}= (\sqrt{\gamma_n} \pm \sqrt{\xi_n})^2.
\ee
Moreover, $m_{2c}$ is injective on the right half complex plane $\{z: \re z > \lambda_+ \}$, and we denote its inverse function as $g_{2c}$, which takes the form
\be\label{solv g2c}
g_{2c}(m)= \frac{\gamma_n }{1+m} -  \frac{\xi_n}{m}.
\ee

Given $m_{1c}$ and $m_{2c}$, we now study the master matrix $M(z)$ in \eqref{defnMx}. We can write the matrix $W$ of eigenvectors as
$$W=\mathbb W \begin{pmatrix}I_k \\ 0 \end{pmatrix},$$ 
where $\mathbb W$ is an $n\times n$ orthogonal matrix. Now recall that $V= S W$. Since the distribution of $S$ is rotationally invariant, we have 
$$V^{\top} V= \begin{pmatrix}I_k , 0 \end{pmatrix} \wh S^{\top} \wh S \begin{pmatrix}I_k \\ 0 \end{pmatrix} ,$$
where $\wh S= S \mathbb W$ is, like $S$, also an $r\times n$ uniformly distributed partial orthogonal matrix. We claim that 
\be\label{claim unifV}
V^{\top} V = \xi_n I_k + \oo(1)\quad \text{in probability}.
\ee
This can be easily verified by a simple variance calculations using exchangeability of the rows or columns of $\wh S$; see the calculations in Appendix \ref{pf sketch1}. 
With \eqref{claim unifV}, the eigenvalue master equation \eqref{eme} 
becomes
\be\label{master_evalue1}
\det \left( {\begin{array}{*{20}c}
   {  -x^{-1/2}\left(1+m_{2c}(x)\right)^{-1} }I_k & {D^{-1}}   \\
   {D^{-1}} & {x^{1/2} m_{2c}(x)} I_k \\
   \end{array}} \right) = \oo(1)
\ee
in probability. Ignoring the small (random) error on the right-hand side, the above matrix equation holds if and only if one of the following equations holds: for $i=1,\ldots,k$, 
\be\label{getthetai}\det\begin{pmatrix} -x^{-1/2}\left(1+m_{2c}(x)\right)^{-1}  & d_i^{-1} \\  d_i^{-1} & x^{1/2} m_{2c}(x)\end{pmatrix} =0 \Leftrightarrow m_{2c}(x) = -\frac{1}{1+d_i^2}  \Leftrightarrow x= g_{2c}\left(-\frac{1}{1+d_i^2}\right) ,\ee
where in the second step we used that $g_{2c}$ is the inverse function of $m_{2c}$. 
This gives an equation for any potential outlier $x$. However, in order to have an outlier, we need to have that 
\be\label{cond outlier}-\frac{1}{1+d_i^2} > m_{2c}(\lambda_+) = -\frac{\sqrt{\xi_n}}{\sqrt{\xi_n} + \sqrt{\gamma_n}} \Leftrightarrow d_i^2 > \sqrt{\frac{\gamma_n}{\xi_n}} .\ee
This is because the Stieltjes transform $m_{2c}$ is increasing on $[\lambda_+,\infty)$ outside of the bulk of eigenvalues, and so having $\lambda_+<x$  is equivalent to $m_{2c}(\lambda_+)<m_{2c}(x)=-1/(1+d_i^2)$. Then we use the known formula for $m_{2c}(\lambda_+)$ as given in \eqref{solv m2c}.
Using \eqref{solv g2c} and \eqref{getthetai}, we obtain that the classical location for the outlier caused by $d_i$ is
$$ \theta_i = \left( 1+d_i^2\right) \left(\frac{\gamma_n}{d_i^2} + \xi_n \right). $$
This formula is very similar to the well known one for the location of the empirical spike in spiked models, presented above, and described in  \citep{baik2005phase,baik2006eigenvalues}. However, we have a different setting in this paper, and so the formula cannot be deduced from the classical one.

See Figure \ref{f_sp_or} for simulation results illustrating the accuracy of this formula. In our simulation, we follow the above model. We generate the data $Y= X + WDU^{\top} $ {  with i.i.d. entries $x_{ij}\sim \mathrm{Unif}(-\sqrt{3/n}, \sqrt{3/n})$ (so $X_{ij}$ has variance $1/n$)}, 
and take the rank $k=1$, with $W,U$ being independent uniformly distributed partial orthogonal random matrices. We choose some specific values for the spike $D$, and compute a random projection $\wt  Y=SY$ with a uniformly random partial orthogonal matrix $S$. We then compute its SVD and find its first eigenvalue and eigenvector. We compare them to the theoretical formulas above. See the caption to Figure \ref{f_sp_or} for more details. In Section \ref{sec: simulation k=5}, we will provide simulation results where the rank $k$ is greater than one.

Next we turn to finding the formula for the angle between projected and true spike. 
Using \eqref{ame}, we obtain that $|\langle u_j, \wt{\bxi}_i\rangle|^2 =\oo(1)$ if $j\ne i$. If $j=i$, we have that the inverse of $M(z)$ is also a block matrix with $2\times2$ blocks, and so its $(i,i)$-th entry can be recovered as the appropriate entry of the inverse of the $2\times2$ block it belongs to. Therefore, we have in probability,
\begin{align} 
|\langle u_i, \wt{\bxi}_i\rangle|^2 &= \frac{1}{2\pi \ii \sqrt{\theta_i}} \oint_{\Gamma_i}(0,d_i^{-1})\left( {\begin{array}{*{20}c}
   {  -z^{-1/2}\left(1+m_{2c}(z)\right)^{-1} } & {d_i^{-1}}   \\
   {d_i^{-1}} & {z^{1/2} m_{2c}(z)}  \\
   \end{array}} \right)^{-1}  \begin{pmatrix} 0  \\ d_i^{-1}\end{pmatrix} \dd z +\oo(1) \nonumber\\
 &=  \frac{1}{2\pi \ii \theta_i \left(1+d_i^2\right)} \oint_{\Gamma_i} \frac1{  {m_{2c}(z)}  + (1+d_i^{2})^{-1}}\dd z +\oo(1)= \frac{1}{\theta_i \left(1+d_i^2\right) m_{2c}'(\theta_i)}   +\oo(1)  \nonumber\\
 &= \frac{g_{2c}'(-(1+d_i^2)^{-1})}{\theta_i \left(1+d_i^2\right) } +\oo(1) = \frac{\xi_n - \frac{\gamma_n}{d_i^4}}{\xi_n + \frac{\gamma_n}{d_i^2}} +\oo(1).
\label{cos_outer}
\end{align}

\begin{figure}[htb!]
\centering
\begin{subfigure}{.45\textwidth}
\includegraphics[width=\textwidth]{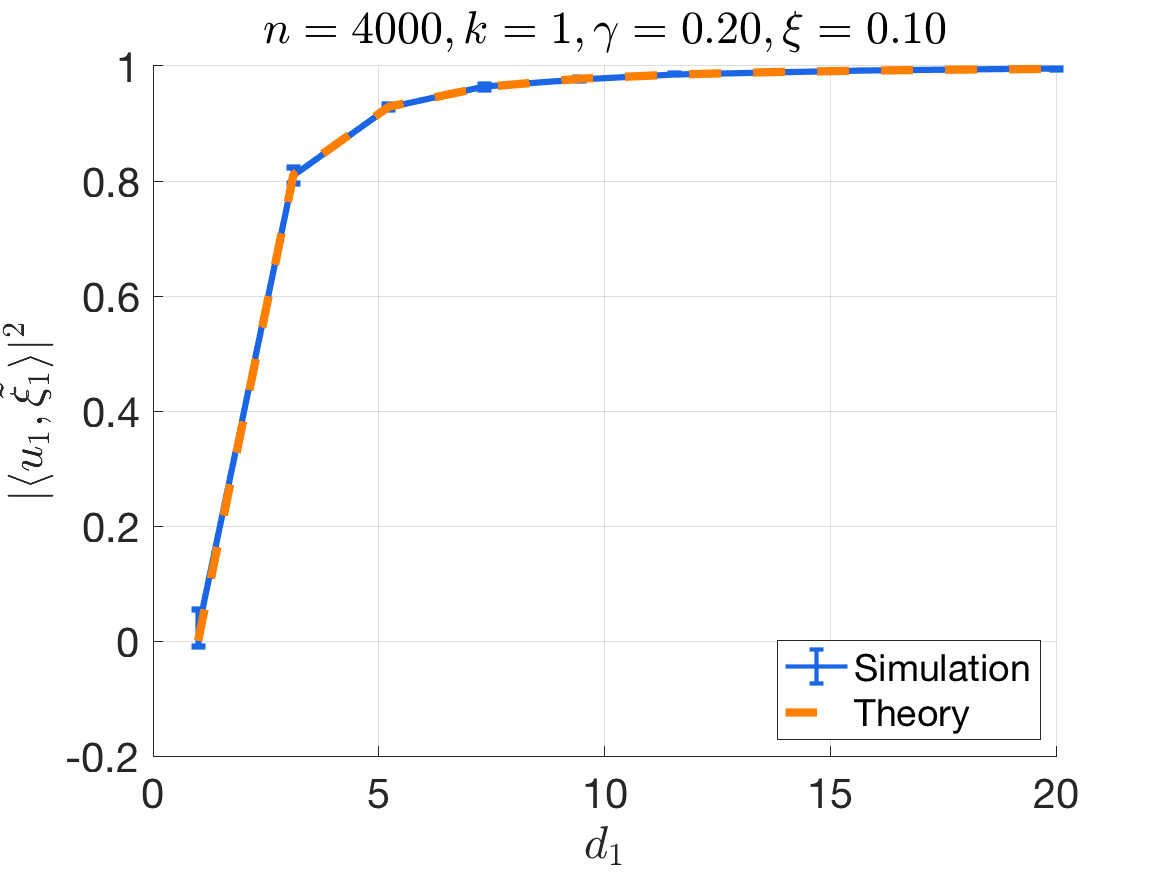}
\end{subfigure}
\begin{subfigure}{.45\textwidth}
\includegraphics[width=\textwidth]{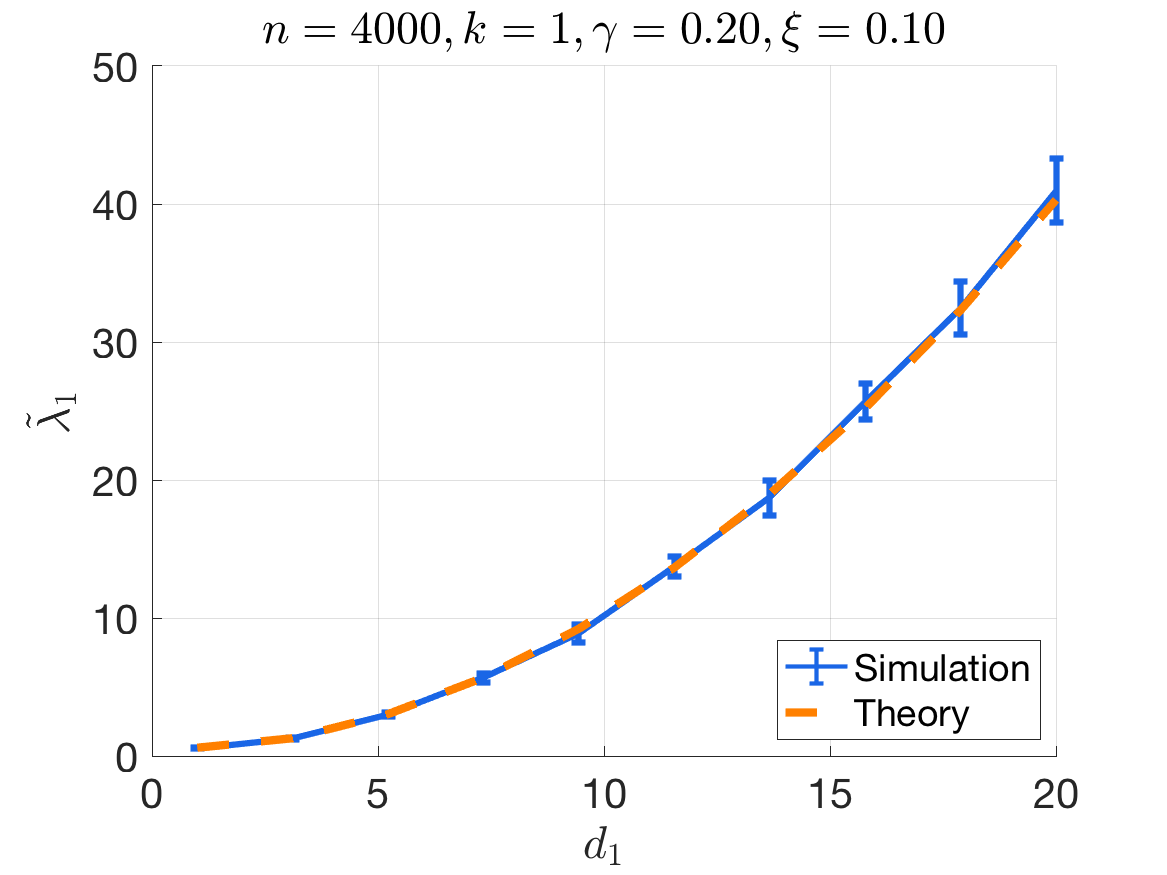}
\end{subfigure}
\caption{
Checking the accuracy of the spiked eigenvalue and eigenvector formulas for orthogonal projections. We show results  with $n = 4000$, $p = 800$, $r = 400$, where the rank is $k = 1$ and we vary the signal strength $d_1$ between 1 and 20. {  The entries $X_{ij}$ are i.i.d. random variables sampled from the  $\mathrm{Unif}(-\sqrt{3/n}, \sqrt{3/n})$ distribution.}
We show the mean and one standard deviation over 20 Monte Carlo simulations (note that the SD is very small). Left: $d_1$ against overlap between the spiked population eigenvector and the sample eigenvector after sketching.  Right: $d_1$ against $\widetilde  \lambda_1$, the first sample eigenvalue after sketching. The theoretical and empirical results agree well.}
\label{f_sp_or}
\end{figure}

Now we state the above results as the following theorem. We shall prove it rigorously in Appendix \ref{pf sketch1}. In the following statements, ``$\to_p$" means ``converges in probability".

\begin{theorem}[Uniform orthogonal random projection]\label{sketchthm1}
Consider the $r\times p$ sketched data matrix $\wt Y=  S Y$, where $S$ is $r\times n$ partial orthonormal, distributed uniformly over the Stiefel manifold of partial orthonormal matrices. Also, $X=(x_{ij})$ is an $n\times p$ random matrix where the entries $x_{ij}$ are real independent random variables with mean zero, variance $n^{-1}$, satisfying that their higher moments are bounded as in \eqref{eq_highmoment}. Also the number of signals $k$ is a finite fixed integer and the strengths $d_1 > d_2 > \cdots > d_k >0$ are fixed constants; $ \{u_i\}_{1\le i \le k}$ and  $\{w_i\}_{1\le i \le k}$ are deterministic sets of orthonormal unit vectors in $\R^p$ and $\R^n$, respectively. Let $\gamma_n:= p /n \to \gamma$ and $\xi_n:= r / n \to \xi$ as $n\to \infty$. Then for any $1\le i \le k$, if $d_i > \sqrt{{\gamma}/{\xi}} $, we have
\be\label{outlierevalue}
\wt\lambda_i \to_p \theta_i:=  \left( 1+d_i^2\right) \left(\frac{\gamma}{d_i^2} + \xi \right), 
\ee
and 
\be\label{outlierevector}
|\langle u_j, \wt{\bxi}_i\rangle|^2 \to_p \delta_{ij}\frac{\xi - \frac{\gamma}{d_i^4}}{\xi + \frac{\gamma}{d_i^2}} .
\ee
Otherwise, if $d_i \le \sqrt{{\gamma}/{\xi}} $, we have
\be\label{outlierevalue2}
\wt\lambda_i \to_p \lambda_+,
\ee
and  
\be\label{outlierevector2}
|\langle u, \wt{\bxi}_i\rangle|^2 \to_p 0
\ee
 for any deterministic unit vector $u$.
\end{theorem}

Recall our heuristic thoughts that this sketched spiked model should be \emph{equivalent to a spiked model \eqref{Heuristicf} with the same spike strengh but in a reduced dimension $r$}. The above theory is consistent with that. Moreover, the theory is also consistent with the results readily deduced from prior work for Gaussian data presented in \eqref{know333} and \eqref{know444}.

\subsection{Projections with i.i.d. entries} \label{sec Gauss}

Now we pick $S$ to be an $r\times n$ random matrix with i.i.d. entries of zero mean, variance $n^{-1}$, and with bounded moments, as in \eqref{eq_highmoment}. In particular, $S$ can be a random Gaussian projection if its entries are i.i.d.\;Gaussian. Then $B= SS^{\top}$ is a sample covariance matrix with identity population covariance.   Before giving the main result, Theorem \ref{sketchthm4}, we first introduce the notations that are used in its statement.  \nc

 First, we define two functions $m_{1c}^S$ and $m_{2c}^S$, which are the Stieltjes transforms of the well-known Marchenko-Pastur law ($m_{1c}^S$ is for $SS^{\top}$ and $m_{2c}^S$ is for $S^{\top}S$):
\beqs 
\begin{split}
& m_{1c}^S(z)=\frac{-(z-1 +\xi) + \sqrt{(z-\lambda_+^S)(z-\lambda_-^S)} }{2z\xi}, \quad m_{2c}^S(z)=\frac{-(z +1 - \xi) + \sqrt{(z-\lambda^S_+)(z-\lambda^S_-)} }{2z},
\end{split}
\eeqs
where $\lambda^S_{\pm}$ are the edges of the support of the MP law,
$\lambda^S_{\pm}= (1 \pm \sqrt{\xi})^2.$ Then $g_{1c}$ is defined as
\beqs 
g_{1c}(m)=- \frac{\gamma}{m}   + \frac{\xi}{m}\left( 1- \frac{1}{m}m^{S}_{1c}(-m^{-1}) \right).
\eeqs
In fact, $g_{1c}$ is the inverse function of $m_{1c}(z)$, which is the unique solution to the cubic equation
\beqs z^2m_{1c}^3 - z(1+\xi_n-2\gamma_n)m_{1c}^2 - \left[ z + (1-\gamma_n)(\gamma_n-\xi_n)\right] m_{1c} - \gamma_n =0, \eeqs
that satisfies $\im m_{1c}(z)>0$ for any $z$ with $\im z > 0$. It is possible to give an explicit expression of $m_{1c}(z)$ using the roots formulas for cubic equations, but we do not state it here. Taking $\im z$ down to zero, we get a continuous function $\rho_{1c}(x): = \lim_{\eta\downarrow 0} \pi^{-1}\Im\, m_{1c}(x+\ii \eta)$. $\xi^{-1}\rho_{1c}$ is a probability density function compactly supported on
 $\mathbb R_+:=[0,\infty)$, and we denote the rightmost edge of its support by $\lambda_+$ following the convention in random matrix theory. Then we define $\al_i\equiv \al(d_i) $ as
\beqs
\al_i\equiv\al(d_i):= -\frac{\gamma d_i^{-2}}{\left(1+\gamma d_i^{-2}\right)\left(\xi+\gamma d_i^{-2}\right)},
\eeqs
and $d_c>0$ is defined as the unique solution to equation 
$$\al(d_c)= m_{1c}(\lambda_+).$$ 

 We have the following result for sketching with i.i.d. projection, which will be proved rigorously in Appendix \ref{pf sketch1.5}. 

\begin{theorem}[Random projection with iid entries]\label{sketchthm4}
Suppose that the assumptions in Theorem \ref{sketchthm1} hold except that $S$ is an $r\times n$ i.i.d.\;random projection matrix whose entries are of zero mean, variance $n^{-1}$, and with bounded moments, as in \eqref{eq_highmoment}. Then for any $1\le i \le k$, if $d_i > d_c $, we have
\be\label{outlierevalueiid}
\wt\lambda_i \to_p \theta_i:= g_{1c}  \left( \al_i\right), 
\ee
and 
\be\label{outlierevectoriid}
|\langle u_j, \wt{\bxi}_i\rangle|^2 \to_p  \delta_{ij}\frac{\al_i^2}{d_i^{2}} \frac{g_{1c}'(\al_i)}{\left[(m_{2c}^S)' (-\al_i^{-1})\right]\al_i^{-2}-(1+\gamma d_i^{-2}) } .
\ee
Otherwise, if $d_i \le d_c $, then \eqref{outlierevalue2} and \eqref{outlierevector2} hold. 
\end{theorem}

We can get explicit expressions for the right-hand sides of \eqref{outlierevalueiid} and \eqref{outlierevectoriid} using the formulas for $m_{1c}^S$, $m_{2c}^S$ and $g_{1c}$. But they are very complicated, so we do not state them here. 


Algorithmically, given $\gamma_n,\xi_n,d_i^2$, we find the spike by calculating the value of \eqref{outlierevalueiid}  using the above formula for $g_{1c}$ (which involves $m_{1c}^S$.) See Figure \ref{f_sp_g} for simulations checking the accuracy of these results.


\begin{figure}[h]
\centering
\begin{subfigure}{.45\textwidth}
\includegraphics[width=\textwidth]{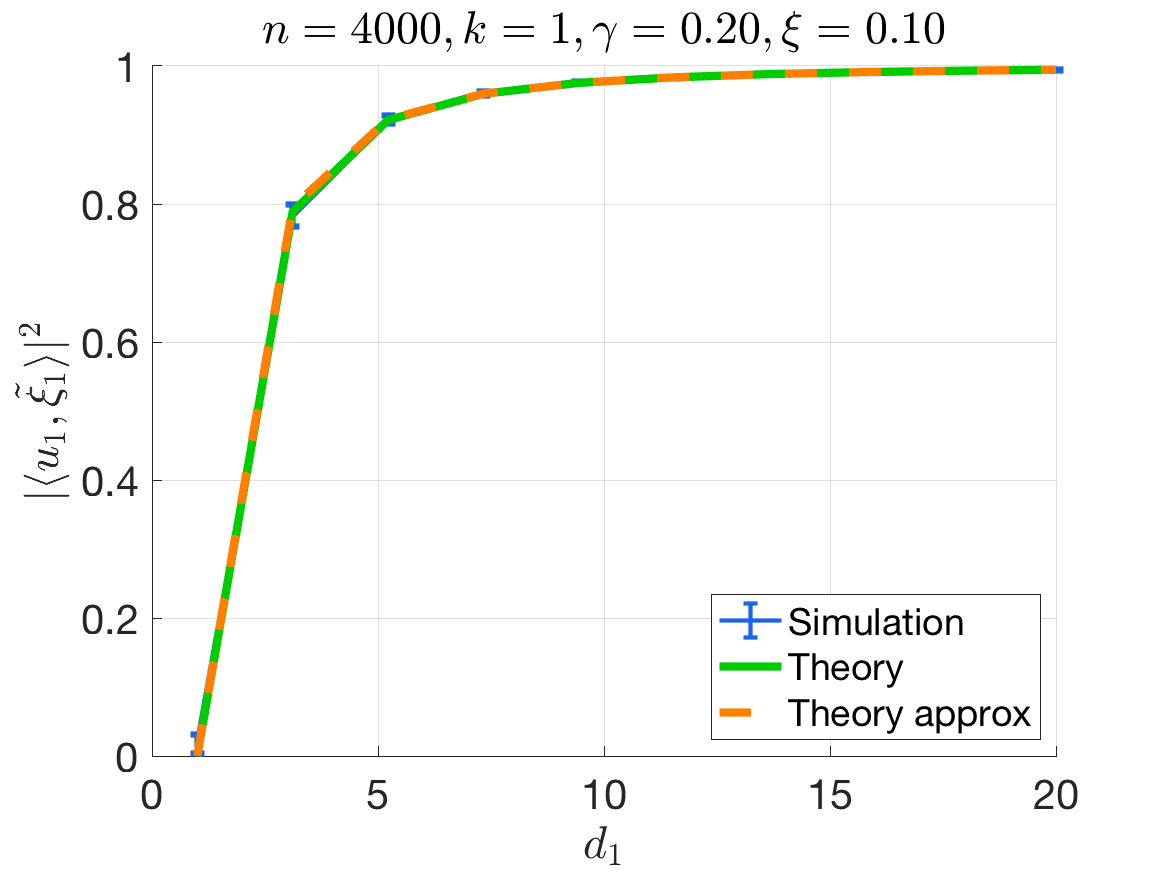}
\end{subfigure}
\begin{subfigure}{.45\textwidth}
\includegraphics[width=\textwidth]{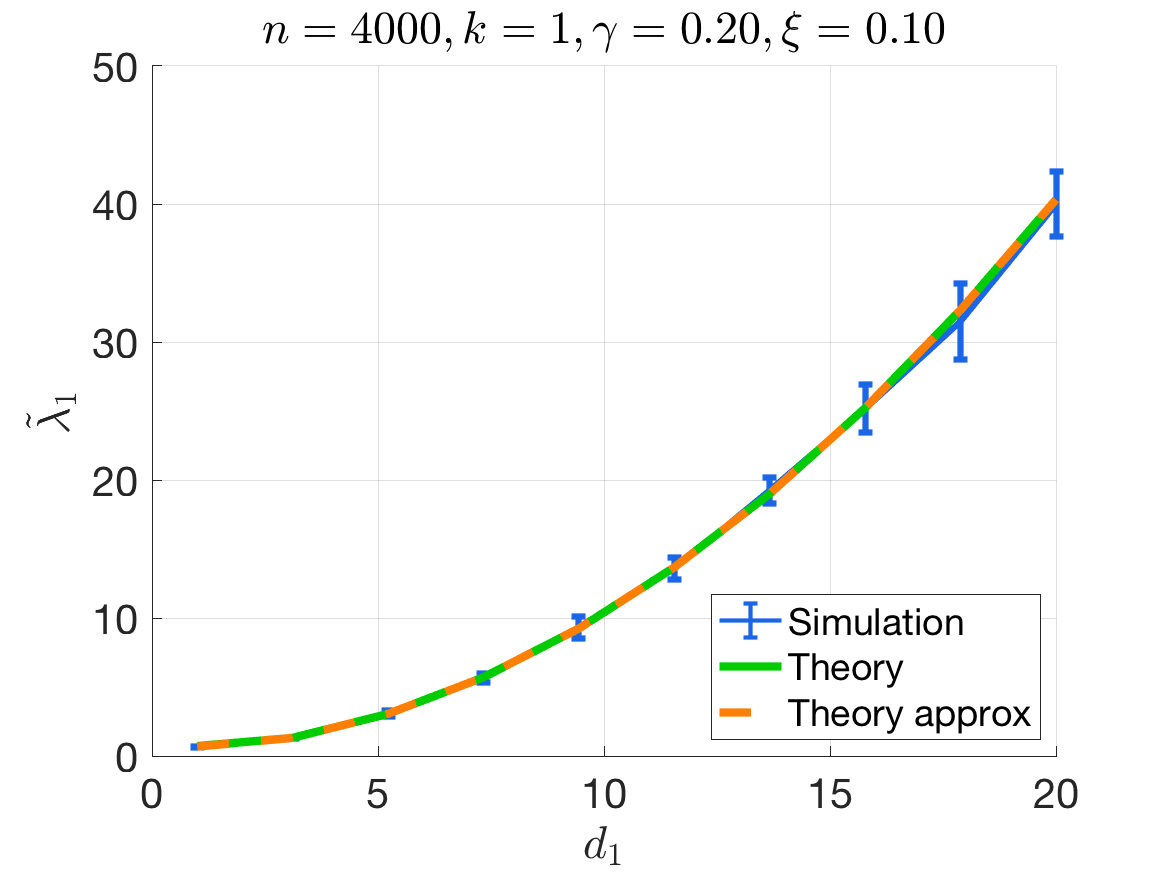}
\end{subfigure}
\caption{Checking the accuracy of the spiked eigenvalue and eigenvector formulas for Gaussian projections. We follow the protocol from the experiment in Figure \ref{f_sp_or}.}
\label{f_sp_g}
\end{figure}



Now we compare the i.i.d.\;projection with the uniform random projection in previous section. The explicit expressions for $\theta_i$ and $|\langle u_i, \wt{\bxi}_i\rangle|^2$ are pretty cumbersome. To simplify the expression, we consider the large signal case where $d_i$ is a large constant, and develop asymptotic expressions of $\theta_i$ and $|\langle u_i, \wt{\bxi}_i\rangle|^2$ in terms of $d_i^{-2}$. Through direct calculation, we get that 
$$ (m_{1c}^S)' (-\al_i^{-1}) = \al_i^2\left( 1 - 2\al_i + (3+3\xi) \al_i^2 + \OO(\al_i^3)\right), $$  
$$(m_{2c}^S)' (-\al_i^{-1}) = \al_i^2\left( 1 - 2\xi\al_i + (3+3\xi) \xi_n\al_i^2 + \OO(\al_i^3)\right),$$
and
$$ m_{1c}^S (-\al_i^{-1}) = \al_i\left(1-\al_i +(\xi+1)\al_i^2  + \OO(\al_i^3)\right), \quad g'_{1c}(\al_i)= \frac{\gamma}{\al_i^2} - \xi(\xi+1)+ \OO(\al_i) .$$
Plugging them into \eqref{outlierevalueiid} and \eqref{outlierevectoriid}, we obtain that
\beq  \label{evaluecorr} 
\theta_i=  \xi d_i^2+ (\xi \gamma + \gamma  +\xi )+(\gamma+\xi+1) \gamma d_i^{-2}  + \OO(d_i^{-4}), 
\eeq
and
\be \label{evectorcorr}
\begin{split}
|\langle u_i, \wt{\bxi}_i\rangle|^2 &\to_p \frac{\gamma - \xi(1+\xi)\al_i^2 +\OO(\al_i^3)}{d_i^{2}} \frac{1}{  - 2\xi\al_i + (3+3\xi) \xi\al_i^2 + \OO(\al_i^3)- \gamma d_i^{-2} } \\
&= \frac{\xi - \frac{(1+\xi)\gamma}{ d_i^4} +\OO(d_i^{-6})}{\xi + \frac{(1+\xi)\gamma}{ d_i^2} +\OO(d_i^{-4})} . 
\end{split}
\ee
Compared with \eqref{outlierevector}, one can see that, at least in the large signal regime, the correlation \eqref{evectorcorr} is smaller, and thus worse than  random uniform projection. Moreover, we have the simple relation
$$ \frac{(1- |\langle u_i, \wt{\bxi}_i\rangle|^2)_{\text{i.i.d. projection}} }{(1- |\langle u_i, \wt{\bxi}_i\rangle|^2)_{\text{uniform projection}} } = 1+\xi + \OO(d_i^{-2})$$
for these two cases, where the notations are self-explanatory. 

One can see Figure \ref{f_sp_g} for simulations checking the accuracy of the results \eqref{evaluecorr} and \eqref{evectorcorr}. Surprisingly, even for small $d_i$, they are already sufficiently precise.  



\subsection{Uniform random sampling}\label{sec unifsample}

Next, we take $S$ to be an $n\times n$ diagonal sampling matrix, where the entries $S_{ii}$ are i.i.d. with 
\be\label{defnSsampling}S_{ii} = \epsilon_i ,\ee
where $\e_i \sim Bernoulli(r/n)$. 
This is closely related to sampling $r$ out of $n$ datapoints uniformly at random, as for large $r$ and $n$ the number sampled concentrates around $r\pm \OO(\sqrt{r(1-r/n)})\approx r$. 
Then we find the following result.

\begin{theorem}[Uniform random sampling]\label{sketchthm2}
Suppose that the assumptions in Theorem \ref{sketchthm1} hold, but $S$ is a random sampling matrix as in \eqref{defnSsampling}. 
We also make the extra assumption that the vectors $w_i$ are delocalized in the following sense:
\be\label{delocal}
\max_{1\le i \le k} \|w_i\|_\infty \to 0 \quad \text{as} \quad n\to \infty. 
\ee
Then the results 
\eqref{outlierevalue}-\eqref{outlierevector2} hold. 
\end{theorem}

The proof of Theorem \ref{sketchthm2} is a minor modification of the one for Theorem \ref{sketchthm1} in Appendix \ref{pf sketch1}, and we highlight the differences in Section \ref{pfunif}.  See Figure \ref{f_sp_samp} for experimental results supporting these theoretical results.

We remark that the delocalization condition \eqref{delocal} is necessary for uniform random sampling, because it means that for any $w_i$, $1\le i \le k$, it does not contain an entry that is particularly significant. For example, consider the case where $w_i$ only contains one non-zero entry. Then uniform random sampling has a positive probability of missing this entry, so that we have $Sw_i=0$. In this case, the principal components of the sketched matrix $SY$ will deviate greatly from those of $Y$. 
\nc

{In fact, when $w_i$ is not delocalized, it is more natural to use some non-uniform sampling methods.
We refer the reader to Section \ref{sec_disc} for a discussion of some more advanced sampling methods.}

\begin{figure}[htb!]
\centering
\begin{subfigure}{.45\textwidth}
\includegraphics[width=\textwidth]{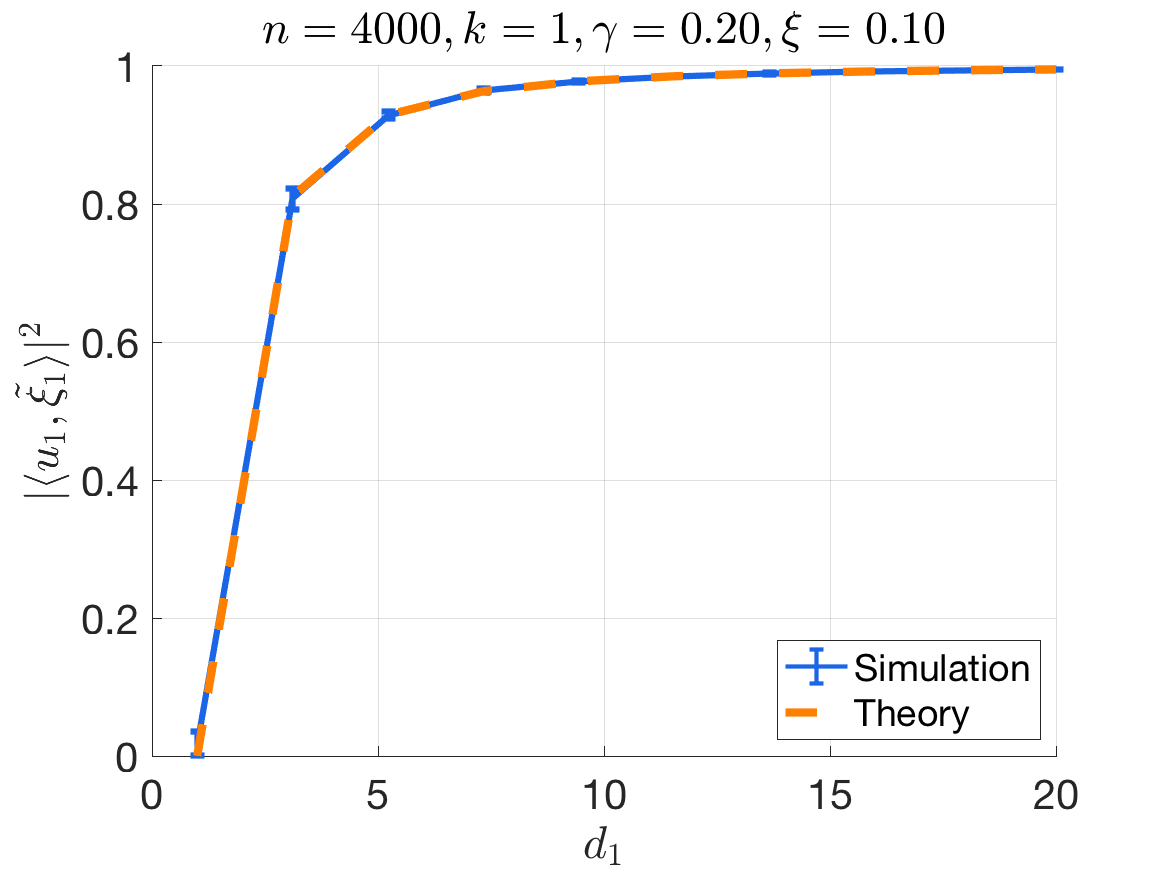}
\end{subfigure}
\begin{subfigure}{.45\textwidth}
\includegraphics[width=\textwidth]{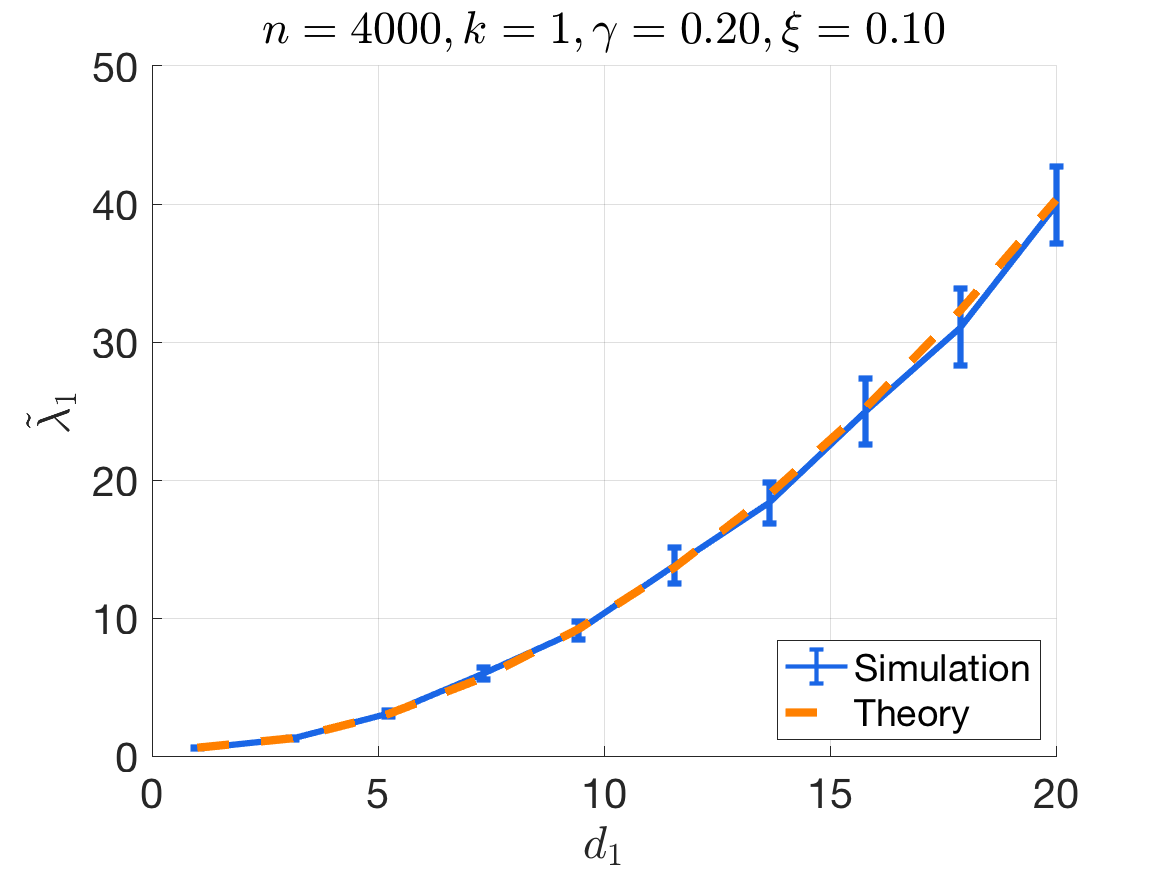}
\end{subfigure}
\caption{Checking the accuracy of the spiked eigenvalue and eigenvector  formulas for uniform random sampling. We follow the protocol from the experiment in Figure \ref{f_sp_or}.} 
\label{f_sp_samp}
\end{figure}

{\bf Gaussian data.} Recall from Section \ref{Gaussian} that we can readily get the same results for Gaussian noise from known spiked model results if we can show $W^\top S^\top SW \approx r/n \cdot I_k$. Now $W^\top S^\top SW = \sum_i \ep_i w[i]^\top w[i] $, where $w[i]$ are the rows of $W$. The claim follows by standard matrix concentration results \citep[e.g.,][etc]{vershynin2010introduction,tropp2012user}. Since this is not our main point, we will not elaborate it in more detail.

\subsection{Randomized Hadamard sampling}\label{sec hadamard}

We consider the subsampled randomized Hadamard transform. Define the $n\times n$ subsampled randomized  
Hadamard matrix as
\be\label{defnSHard} S=\frac1{\sqrt{n}}B_r HD,\ee
where $B_r $ is a diagonal sampling matrix with i.i.d. $Bernoulli(r/n)$ diagonal entries, $H$ is the \emph{Walsh-Hadamard matrix} and $D$ is a diagonal matrix of i.i.d. sign random variables, equal to $\pm 1$ with probability 1/2. 
Recall that the Walsh-Hadamard matrix is defined recursively by
\begin{align*}
H_n=\left(\begin{array}{cc}H_{n/2} & H_{n/2}\\
H_{n/2} & -H_{n/2}\end{array}
\right),
\end{align*}
with $H_1=(1)$. This requires $n$ to be a power of 2. For $n$ that is not a power of 2, one may take a random subsample of rows and columns of $H_{n'}$, for some $n'>n$ that is a power of 2.

For $S$ defined in \eqref{defnSHard}, we denote the action of the Walsh-Hadamard matrix $H$ and the signflip matrix $D$ on a vector $w_i$ as
$$z_i : =\frac1{\sqrt{n}} HD w_i, \quad 1\le i \le k.$$
Note that each entry $z_i(l)$ is of the form
$$z_i(l) = \sum_{j=1} a^{(l)}_j w_i(j),$$
where $a_j^{(l)}= \pm n^{-1/2}$ is chosen independently and uniformly. Then a Chernoff type bound gives that the $z$ vectors are delocalized, i.e., 
\be\label{chernoff}
\|z_i\|_\infty \le C\frac{\log n}{\sqrt n}
\ee
with high probability. Moreover, $\{z_i\}$ are orthonormal since $HD$ is orthogonal. Then the result for uniform random sampling can be applied here without the delocalization assumption in \eqref{delocal}, because \eqref{chernoff} already gives the desired delocalization for $z_i$-s after acting $HD$ on $w_i$-s.


The argument above clearly applies more broadly to general \emph{Hadamard matrices}.
An $n\times n$ possibly complex-valued matrix $H$ is called a Hadamard matrix if $H/\sqrt{n}$ is orthogonal and the absolute values of its entries are unity, $|H_{ij}|=1$ for $i,j=1,\ldots,n$. 
The Walsh-Hadamard matrix above clearly has these properties. Another construction is the discrete Fourier transform (DFT) matrix with the $(u,v)$-th entry equal to 
$H_{uv}= e^{-2\pi i(u-1)(v-1)/n}$.
Multiplying this matrix from the right by $X$ is equivalent to applying the discrete Fourier transform to each column of $X$, up to scaling. The time complexity for the matrix-matrix multiplication for both transforms is $\OO(np\log n)$ using the Fast Fourier Transform.

To summarize, we have the following theorem as a corollary of Theorem \ref{sketchthm2} and the delocalization property in \eqref{chernoff}.




\begin{theorem}[Randomized Hadamard sampling]\label{sketchthm3}
Suppose that the assumptions in Theorem \ref{sketchthm1} hold except that $S$ is now a random sampling matrix as in \eqref{defnSHard}, where $H$ is a general $n\times n$ Hadamard matrix. Then the results 
\eqref{outlierevalue}-\eqref{outlierevector2} hold. 
\end{theorem}

See Figure \ref{f_sp_sampH} for experimental results supporting this theorem.


\begin{figure}[htb!]
\centering
\begin{subfigure}{.45\textwidth}
\includegraphics[width=\textwidth]{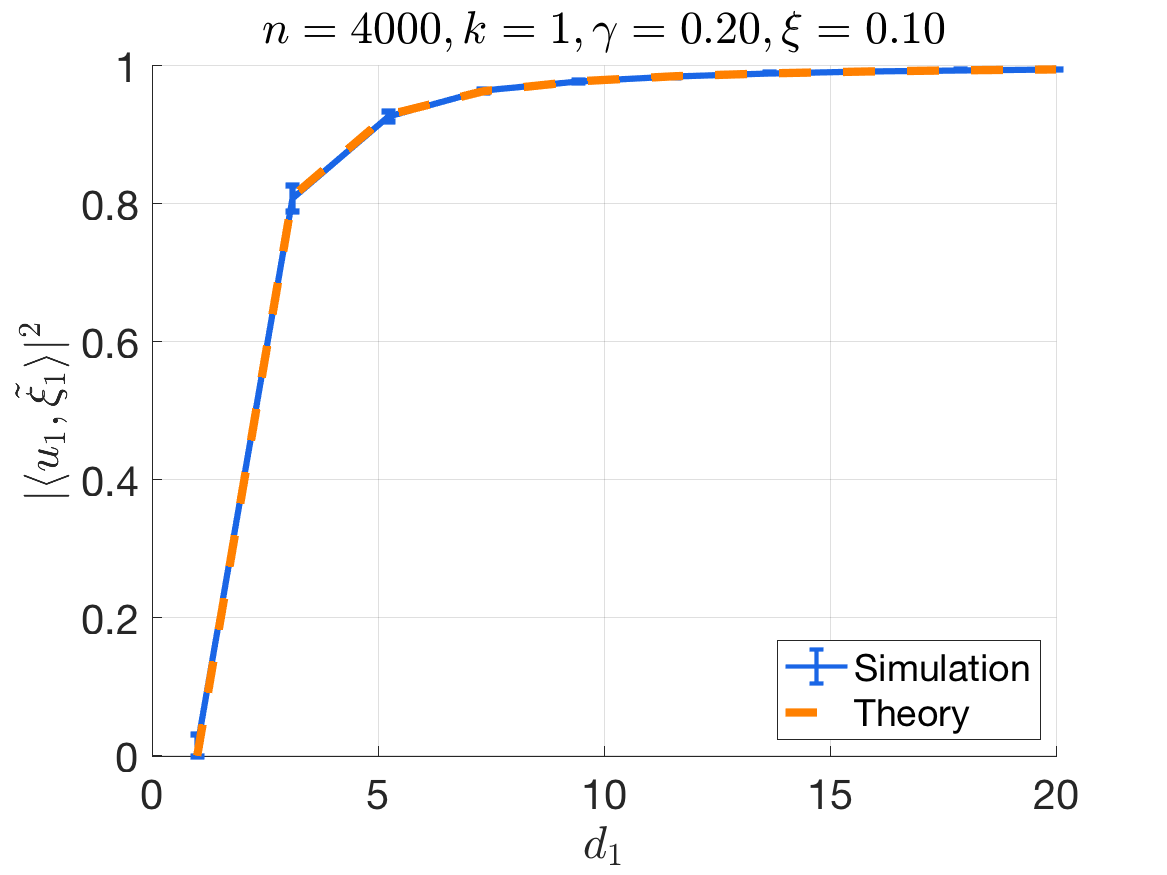}
\end{subfigure}
\begin{subfigure}{.45\textwidth}
\includegraphics[width=\textwidth]{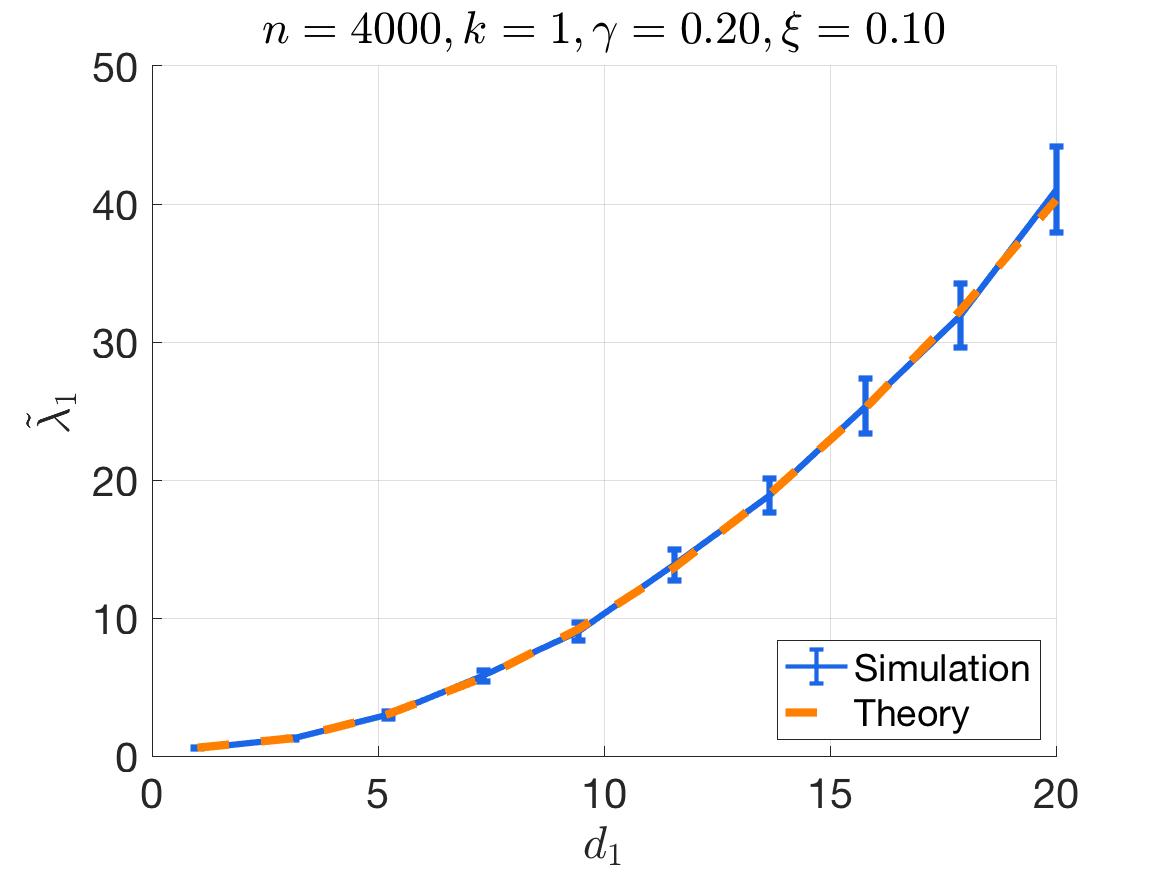}
\end{subfigure}
\caption{Checking the accuracy of the spiked eigenvalue and eigenvector  formulas for the subsampled randomized Hadamard transform. We follow the protocol from the experiment in Figure \ref{f_sp_or}.} 
\label{f_sp_sampH}
\end{figure}

{\bf Gaussian data.} From Section \ref{Gaussian}, we get the same results for Gaussian noise if we can show $W^\top S^\top SW \approx r/n \cdot I_k$. This follows from the same argument as for uniform sampling.

\subsection{CountSketch}\label{sec count}
Another popular sketching method is CountSketch \cite{charikar2002finding}, 	also known as Clarkson-Woodruff sketch \cite{clarkson2017low}.  Here $S$ is an $r\times n$ matrix that has a single randomly chosen non-zero entry $S_{h(j),j}$ in each column $j$, for a random mapping $h:\{1,\ldots,n\}\to \{1,\ldots, r\}$. Moreover, each $S_{h(j),j}$ is a Rademacher random variable, i.e., $S_{h(j),j} = \pm 1$ with probability 1/2. In other words, we have
\be\label{defnScount}
S_{ij} = \delta_{i h(j)} a_j,
\ee
where $a_j$ are i.i.d. Rademacher random variables that are independent of $h$. Intuitively, $S$ maps the vector $x$ to a random partition of its entries (mapping into random buckets), and takes randomly signed sums of the entries in each partition (or in each bucket).

	When applied to an $n\times p$ matrix $X$, $SX$ computes an $r\times p$ matrix, such that each row is a randomly signed sum of some rows of $X$. This is similar to random sampling. However, the advantage is that no rows of $X$ are ``left out", and thus we automatically get a type of adaptive leverage score sampling, see e.g., \cite{clarkson2017low}. The only constraint is that we need $r$ to be large enough so that we avoid collisions of rows with large leverage score.

In our case, it turns out it is advantageous to study a slightly modified ``normalized" CountSketch. To see, this, we denote 
	$$SS^{\top} = \diag(c_1,\ldots,c_r),$$
	that is, $c_i$ is the number of coordinates from $1,\ldots, n$ that map into the $i$-th bucket. Then $(c_1,\ldots,c_r)$ has the exact joint distribution
	$$(c_1,\ldots,c_r)\sim Multinomial(n; 1/r,\ldots,1/r).$$
	Each $c_i$ has a marginal distribution equal to $Binomial(n,1/r)$, with mean $n/r$, and variance $\frac{n}{r}(1-\frac1r)$. As $n,r\to\infty$, $r/n\to\xi>0$, this tends to a Poisson distribution with constant rate.  Thus we know that for any constant $C>0$,
	$$ \mathbb P(\|SS^{\top}\| \ge C)\ge c$$
	for some constant $c>0$ depending on $C$. Hence \emph{the operator norm of the sketching matrix is unbounded} (i.e., the first bound in \eqref{assm3} fails). This is a problem because, theoretically, the spikes may be ``covered up" by the noise eigenvalues. {  Moreover, the spectral distribution of $SS^\top$ is spread over many points, but we will see that it is better to have a more concentred spectral distribution (cf. Remark \ref{rem extra}).} Hence we propose a simple normalization, in which we divide each bucket by the square root of the number of entries mapped into it. Formally, we define $\wh S:=(SS^{\top})^{-1/2}S$, such that $\wh S\wh S^{\top}=I_r$. Then we shall use $\wh S$ as our sketching matrix.  With $\xi_n$ converging to a constant, there is a significant number of zeros among the counts. Hence $(SS^{\top})^{-1/2}$ should be understood as a pseudo-inverse. Alternatively, we can discard the buckets of size zero at the beginning.
	
Experiments show that the simple normalized version of CountSketch works similarly to uniform projection. As discussed above, we normalize $SX$ as $B^{-1/2}SX$, where $B=SS^{\top} $ is the matrix of counts mapped into each bucket. In the regime where $n/r$ is a constant, 
the probability of getting a zero count is approximately 
$$\mathbb P(Poisson(1/\xi_n)=0)=\exp(-1/\xi_n)= \exp(-n/r).$$ 
We discard those rows. {  From Figures \ref{f_sp_cs}, \ref{fig: count compare}, \ref{fig: check k=5 cos} and \ref{fig: compare all}, we find that the value of $|\langle u_i, \wt{\bxi}_i\rangle|^2$ for normalized CountSketch is larger than the one for CountSketch. This shows that normalized CountSketch is more accurate than the original CountSketch, in the sense that the principal components of the sketched matrix approximate the principal components of the signal matrix in a better way. The reason is that CountSketch has some large buckets, and the sum of the rows mapped into them can sometimes dominate the eigenvectors, leading to a loss of precision. In Remark \ref{rem extra} below, we will also give another heuristic explanation. (Note that in Figure \ref{f_sp_cs}, normalized CountSketch and the original one have similar accuracy. This is because $r/n=0.1$ is small and hence $SS^\top$ concentrates well around $\frac nr I_r$.) }

\begin{figure}[h!]
\centering
\begin{subfigure}{.45\textwidth}
\includegraphics[width=\textwidth]{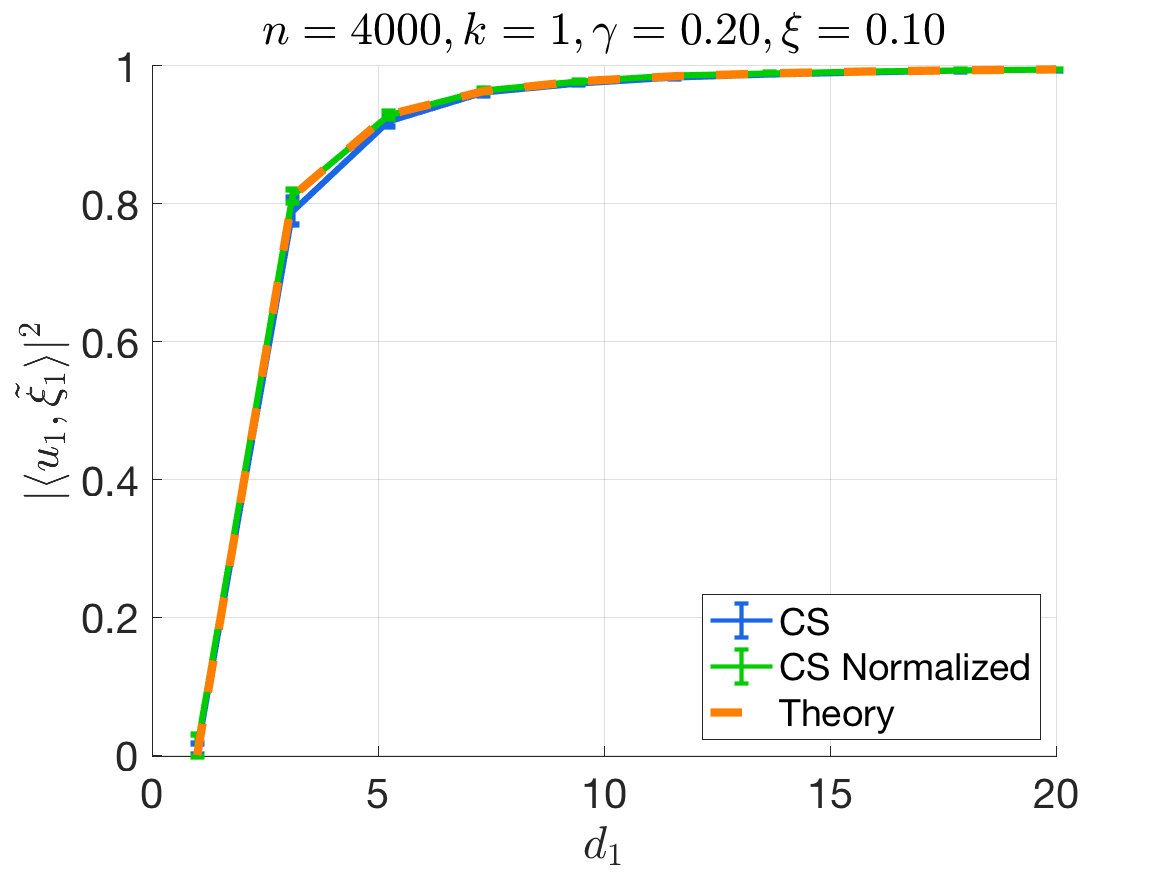}
\end{subfigure}
\begin{subfigure}{.45\textwidth}
\includegraphics[width=\textwidth]{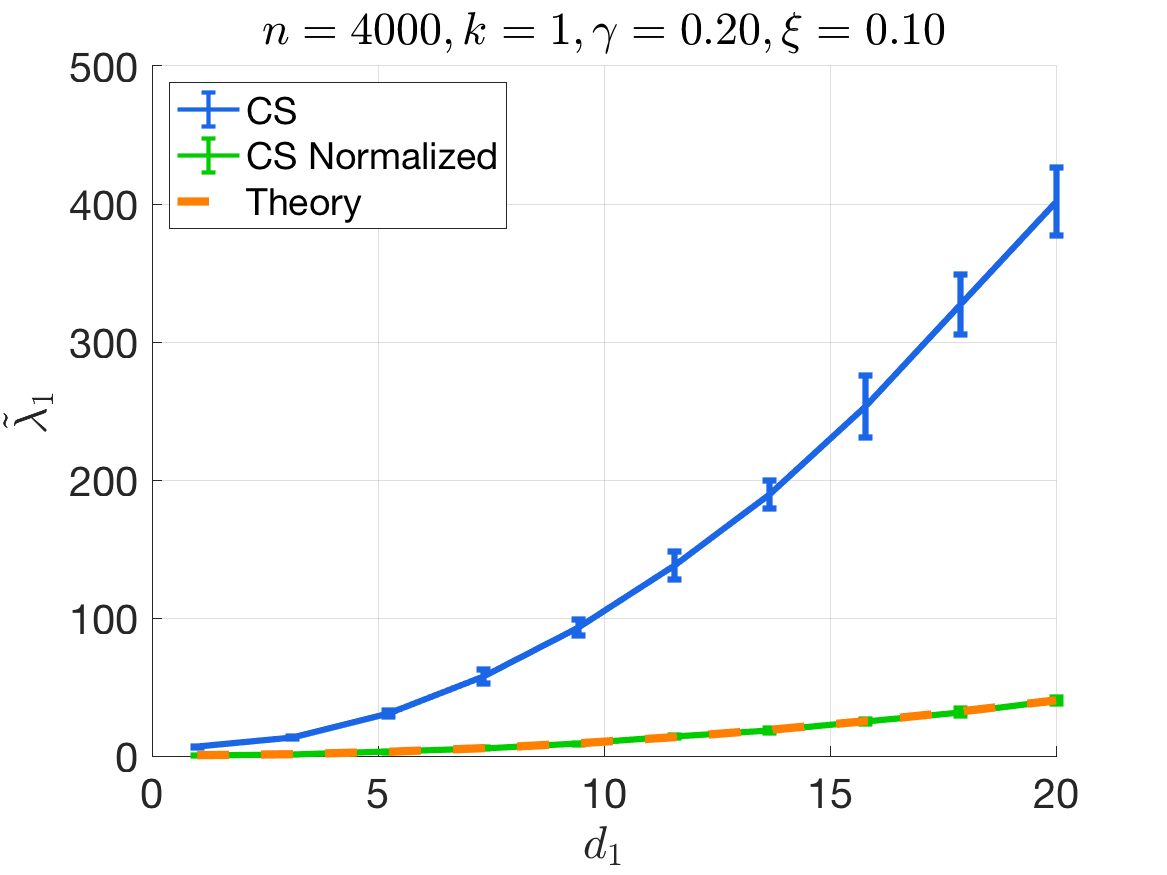}
\end{subfigure}
\caption{Checking the accuracy of the spiked eigenvalue and eigenvector formulas for CountSketch. We follow the protocol from the experiment in Figure \ref{f_sp_or}. 
}
\label{f_sp_cs}
\end{figure}

{  In Figure \ref{fig: count compare}, we compare the accuracy of CountSketch and normalized CountSketch where $p=500$, $n\in\{20, 50, 100, 500\}$, and $\xi=0.2$. We see that normalized CountSketch is more accurate than the unnormalized version, especially for small $n$ and large $p$. However, the standard errors overlap, so one must exercise some caution when reading these figures. This simulation also shows that our theoretical formula is accurate even when $n,\,p$, and $r$ are relatively small.
}
\begin{figure}[h!]
\centering
\begin{subfigure}{.45\textwidth}
\includegraphics[width=\textwidth]{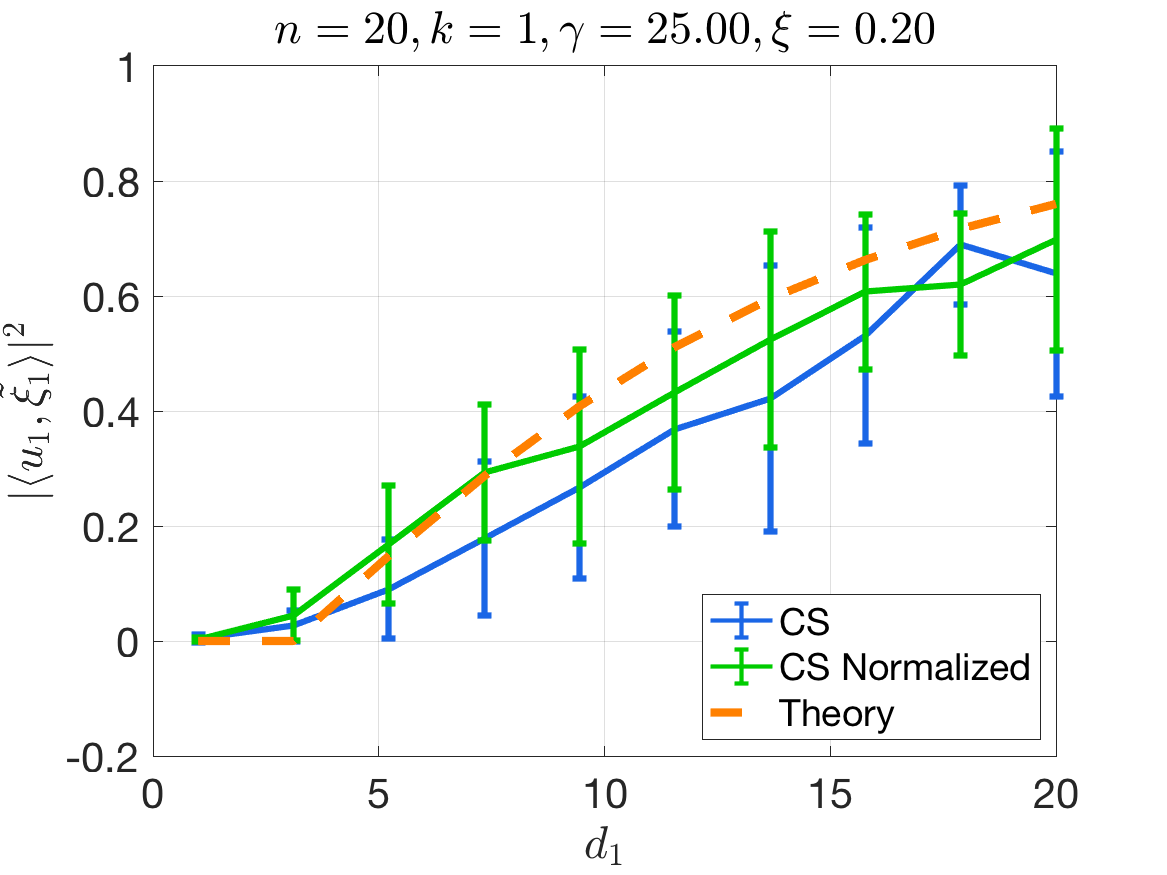}
\end{subfigure}
\begin{subfigure}{.45\textwidth}
\includegraphics[width=\textwidth]{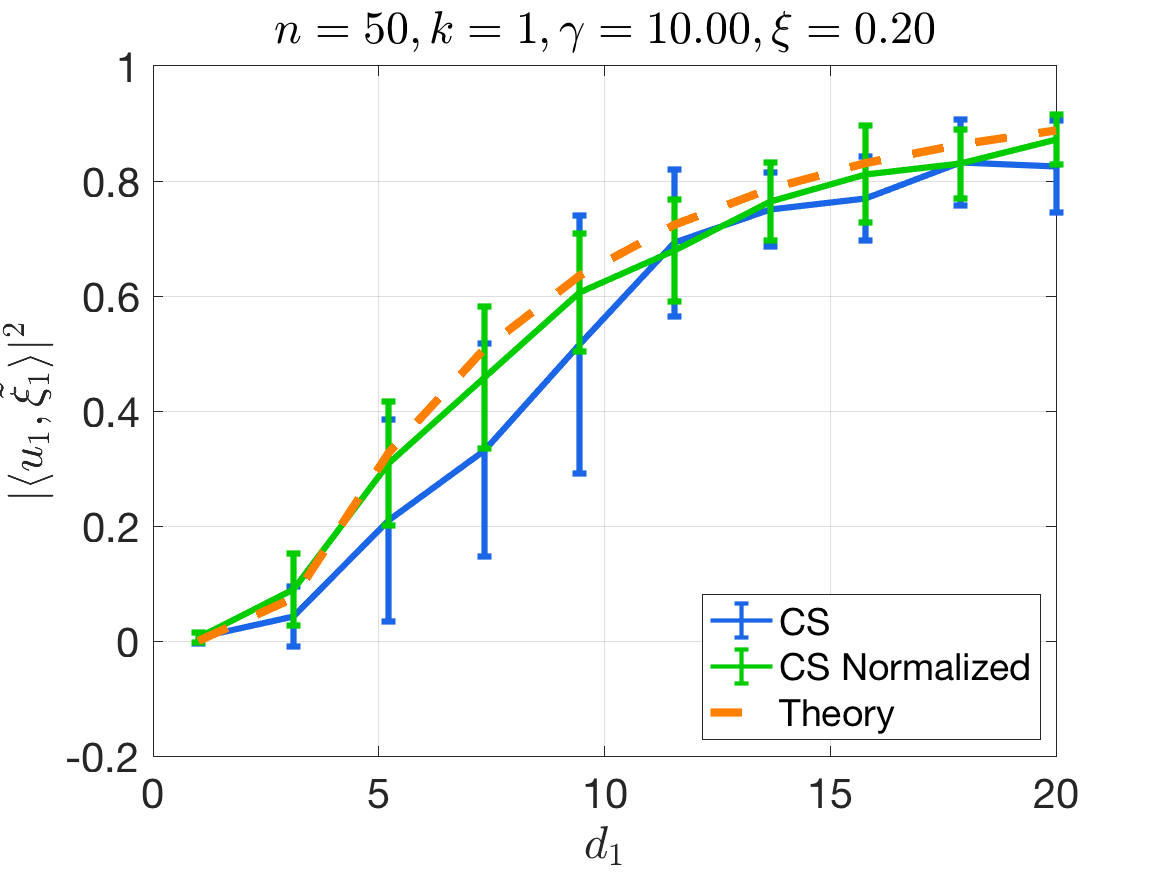}
\end{subfigure}
\begin{subfigure}{.45\textwidth}
\includegraphics[width=\textwidth]{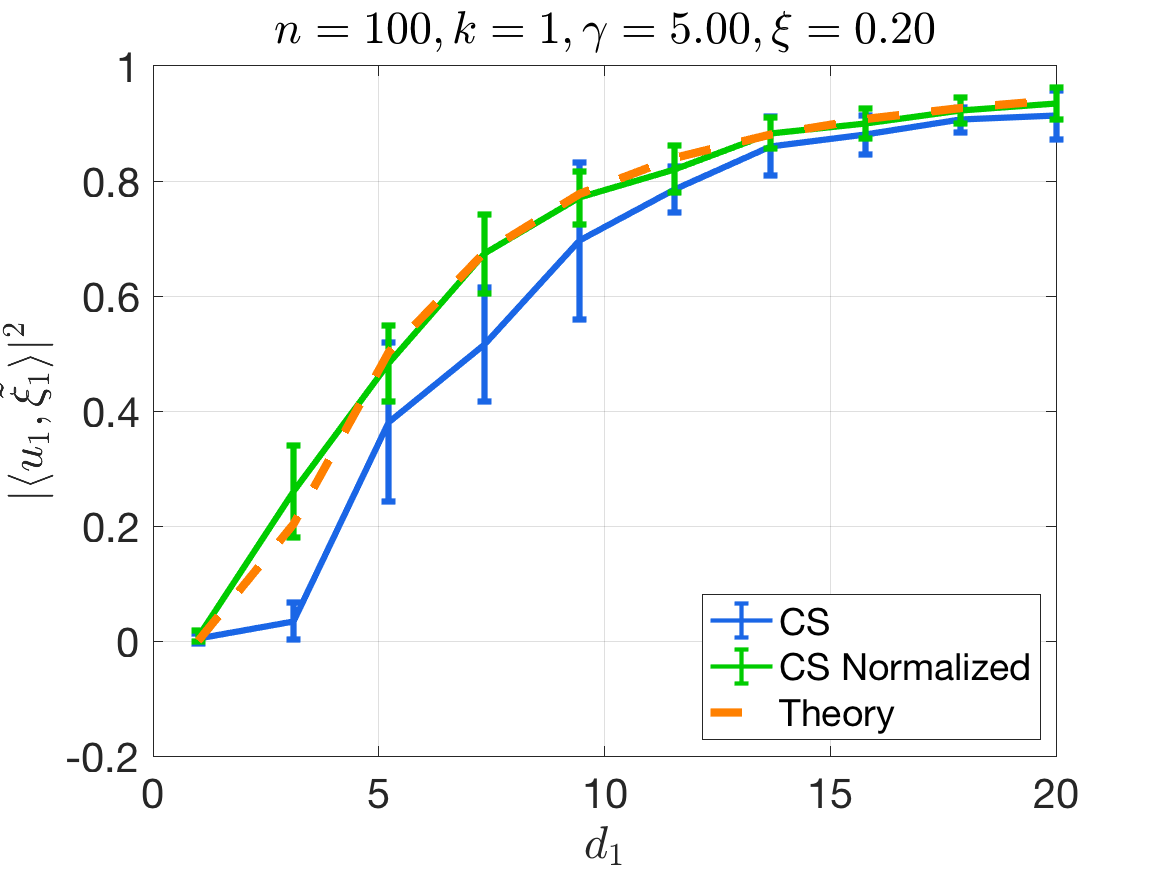}
\end{subfigure}
\begin{subfigure}{.45\textwidth}
\includegraphics[width=\textwidth]{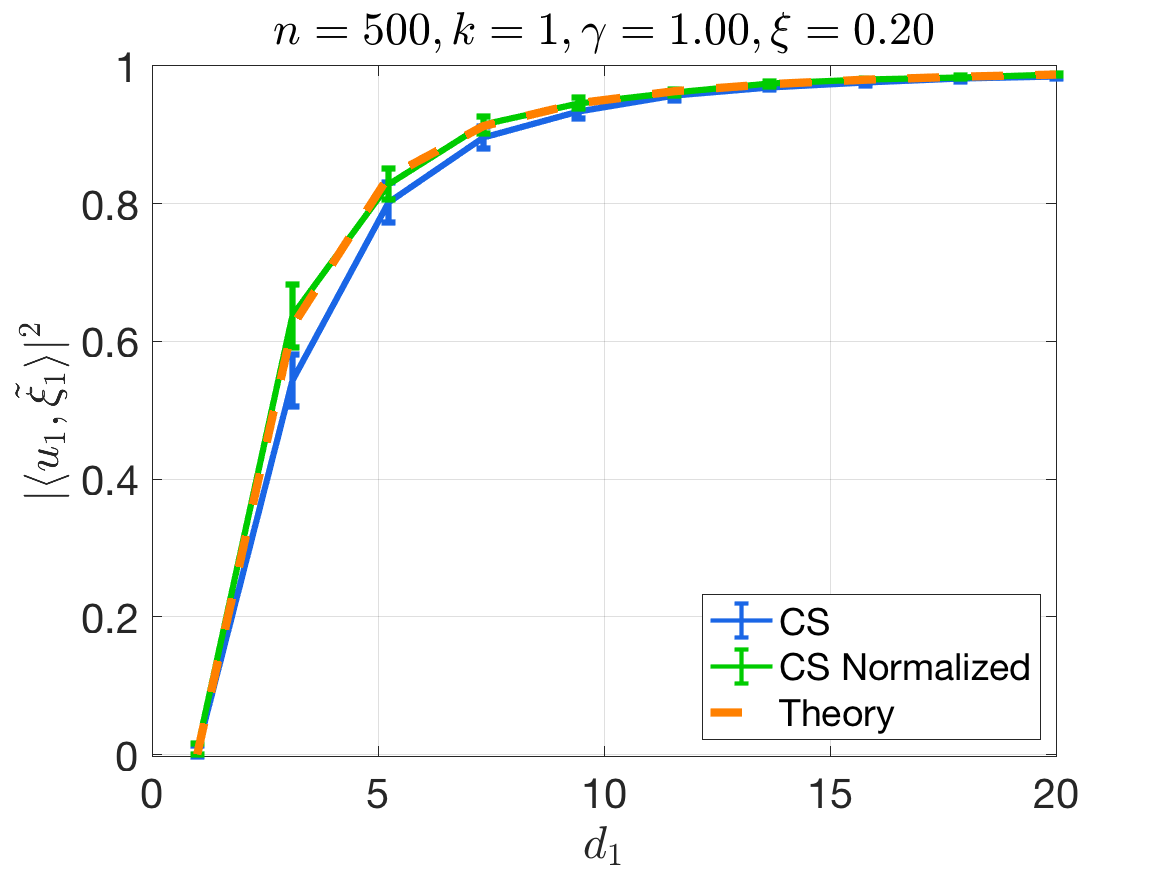}
\end{subfigure}
\caption{Comparing CountSketch and normalized CountSketch accuracy. We take $n\in\{20, 50, 100, 500\}$, $p=500$, and $r/n=0.2$. The setting is the same as in Figure \ref{f_sp_or}.}
\label{fig: count compare}
\end{figure}

	 If $n\gg r$, say $n\ge Cr\log n$ for some large constant $C>0$, then each $c_i$ concentrates around $n/r$. In much of the literature on sketching, this is a common assumption \cite{mahoney2011randomized,woodruff2014sketching}. In this case, we have
	 $$\wh S \approx \xi_n S,$$
	 and hence $\wh S$ is simply a rescaling of the CountSketch matrix $S$. 

We collect the above results for CountSketch below. The proof is presented in Section \ref{pfunif}. 

\begin{theorem}[CountSketch]\label{sketchthm5}
Suppose that the assumptions in Theorem \ref{sketchthm1} hold, except we assume the delocalization condition \eqref{delocal} and that $S$ is a random sampling matrix as in \eqref{defnScount}. Then \eqref{outlierevalue}-\eqref{outlierevector2} hold if we replace $\xi$ with $\wh \xi = \xi \left[1-\exp(-1/\xi)\right]$.
\end{theorem}

CountSketch can be regarded as an interpolation between the uniform random sampling and randomized Hadamard sampling. For the time complexity, we have
$$ \text{uniform random sampling} < \text{CountSketch} < \text{randomized Hadamard sampling}.$$
However, unfiorm random sampling and CountSketch are much closer in complexity (within a constant), while randomized Hadamard sampling has an additional logarithmic factor in the cost. On the other hand, uniform random sampling and the CountSketch requires the delocalization condition \eqref{delocal}, while randomized Hadamard sampling does not. 

{
One of the advantages of CountSketch is that it is extremely fast for sparse datasets. For example, we can consider the sparse sample covariance matrices which are Hadamard products of the form $\wt X=A\circ X$, where $X$ is a random matrix considered in this paper and $A$ is a random matrix with i.i.d. ${Beroulli}(p_n)$ entries. Then $0<p_n<1$ controls the sparsity of the sample covariance matrices. We expect that CountSketch will perform well in the sparse case with $p_n\ll 1$. Unfortunately, this case is beyond our current setting---the moment condition \eqref{eq_highmoment} will be violated if the entries of $\wt X$ are scaled to have variance $n^{-1}$. However, we expect that our results will still hold under the sparse setting, although we need to rebuild the whole theory in \cite{yang2019spiked} from scratch using the methods in \cite{ER1} for sparse Erd{\H o}s-R{\'e}nyi graphs. This is beyond the scope of the current paper, and we will explore this topic in future work.
}

We also remark that the delocalization condition \eqref{delocal} is needed for CountSketch because we are considering the setting where $r$ is of the same order as $n$. In the conventional setting where $n\gg r \gg p$, this condition is not needed \citep{clarkson2017low}. In our setting, we can recover this result. With a simple Chernoff estimate and a union bound, we know that if $n\ge Cr\log r$ for a large enough constant $C>0$, then with probability $1-\oo(1)$ all the $c_i$-s are concentrated around $n/r$. Moreover, as for Hadamard sampling in Section \ref{sec hadamard}, $\wh Sw_i$ will be delocalized, i.e.,
$$\|\wh Sw_i\|_\infty \to 0 \ \ \ \text{in probability}$$
as $n\to \infty$. This estimate holds for the same reason as \eqref{chernoff}, because we take random averages over roughly $n/r$ many entries of $w_i$. 

\subsection{Strong signals}
\label{strong}

Finally, in this subsection, we consider a more general spiked covariance matrix model
\be\label{generalSigma}Y = \sum_{i=1}^k d_i w_iu_i ^{\top}+ X \Sigma^{1/2},\ee
where the covariance matrix $\Sigma$ can be non-identity. In this case, the covariance matrix of the spiked model is of the form
\be\label{generalSigma2} \wt{\Sigma}=\Sigma+ \sum_{i=1}^k d_i^2 u_i u_i^{\top}. \ee
Since $u_i $-s are not necessarily the eigenvectors of $\Sigma$, they are also not the eigenvectors of $\wt {\Sigma}$ in general. However, if we assume the signal strengths to be sufficiently strong and well-separated, then we can regard $u_i$ as an approximate eigenvector of $\wt{\Sigma}$. This is the setting we shall consider in this subsection.


Suppose we focus on the $i$-th spiked eigenvalue. We assume that
$$ l_i:= d_i^2 \wedge \min_{j\ne i} |d_i^2 - d_j^2|$$
is sufficiently large compared with $\|\Sigma\|$. {  Here we use $a\wedge b$ as a shorthand notation for $\min(a,b)$.} Combining the arguments in the proof for Theorem \ref{sketchthm1} with standard perturbation theory for matrices, we can obtain the following theorem. In the statement, we shall use the notation $x=\OO(y)$ if
$|x| \le C|y|$ for some constant $C>0$ that does not depend on $n$ or $l_i$. The proof of this theorem will be given in Appendix \ref{sec_pflarge}. 

\begin{theorem}[Large signals]\label{sketchthmlarge}
Suppose that the assumptions in Theorem \ref{sketchthm1} hold, so we consider uniform orthogonal random projections. Moreover, assume that for some fixed $k_+\le k$,
\be\label{largegap} \max_{1\le i \le k_+}l_i \ge C_0\|\Sigma\|\ee for a sufficiently large constant $C_0>0$. 
Then for any $1\le i \le k_+$, we have
\beq\label{largethetai}
\wt\lambda_i \to_p \theta_i=  \xi (d_i^2 + E_{ii}) + \gamma \rho_1+ \OO(l_i^{-1}). 
\eeq
Here $E :=U^{\top} \Sigma U$, and  $\rho_{i}$ are the moments of the spectral distribution of $\Sigma$, 
\be\label{rhoisigma}\rho_i:= \int x^i \pi_{\Sigma}(x).\ee
Also,
\be\label{precision_rho2}
|\langle u_i, \wt{\bxi}_i\rangle|^2 \to_p  \frac{\xi -\frac{\gamma}{d_i^{4}}  \rho_2 }{\xi + \frac{\gamma}{d_i^2} \left[ \rho_1 + d_i^{-2} (\rho_2 - \rho_1 E_{ii}) \right]} + \OO (l_i^{-3}),
\ee
and for $j\ne i$,
\beqs
|\langle u_j, \wt{\bxi}_i\rangle|^2 \to_p     \frac{\xi -\frac{\gamma}{d_i^{4}}  \rho_2 }{\xi + \frac{\gamma}{d_i^2} \left[ \rho_1 + d_i^{-2} (\rho_2 - \rho_1 E_{ii}) \right]}  \left|  \frac{E_{ji}}{d^2_i - d^2_j} \right|^2+ \OO(l_i^{-3}) .
\eeqs

Similarly, if the assumptions in Theorem \ref{sketchthm2} (uniform random sampling) or Theorem \ref{sketchthm3} (Hadamard transform) hold, then the same results hold; if the assumptions in Theorem \ref{sketchthm5} (CountSketch) hold, then the same results hold if we replace $\xi$ with $\wh \xi$.
\end{theorem}

 We check the formulas in simulations. {  In the first example, we take $\Sigma=O^\top \Lambda O$, where $O$ is a $p\times p$ orthogonal matrix, $\Lambda$ is a diagonal matrix with $\Lambda_{11}=5$, $\Lambda_{ii}=2$ for $2\leq i\leq p/2$, $\Lambda_{ii}=1$ for $p/2<i\leq p$; see Figure \ref{fig: large signal 1}.}
\begin{figure}[!htb]
\centering
\begin{subfigure}{.45\textwidth}
\includegraphics[width=\textwidth]{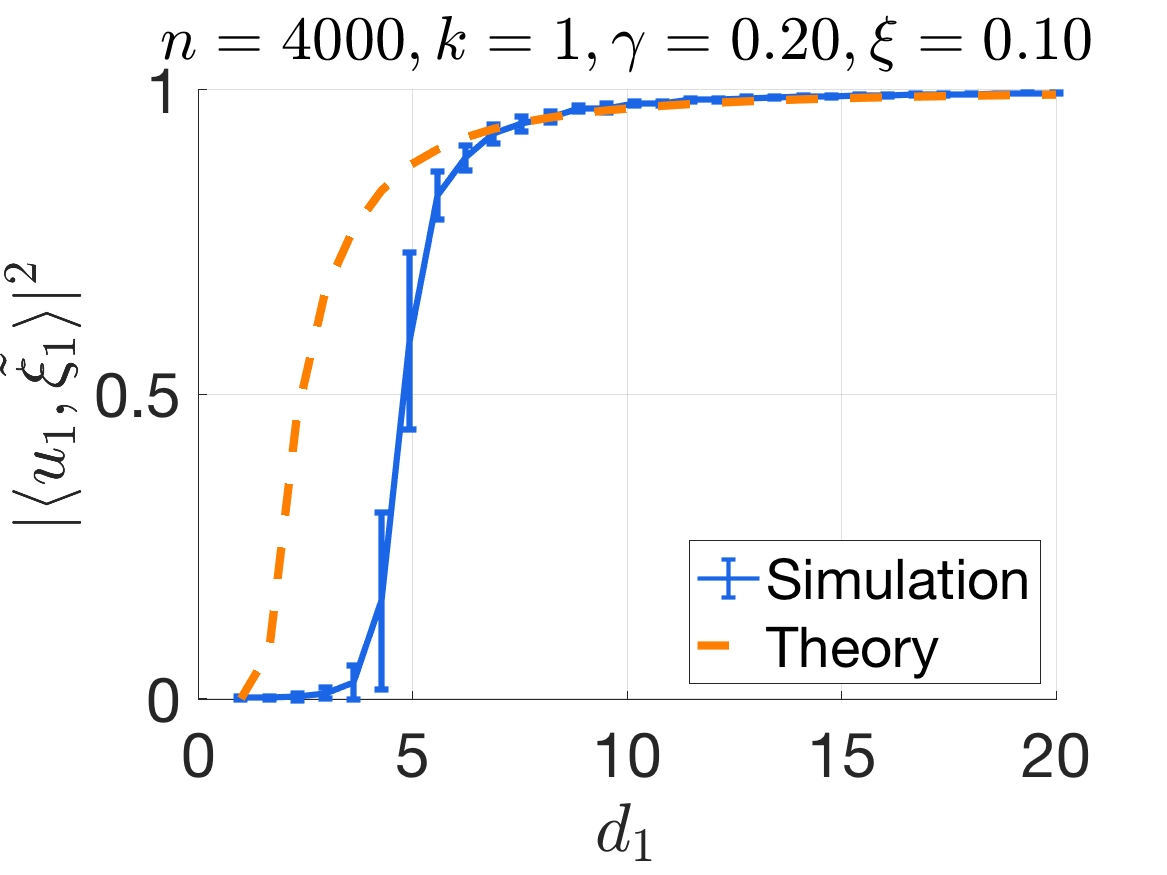}
\end{subfigure}
\begin{subfigure}{.45\textwidth}
\includegraphics[width=\textwidth]{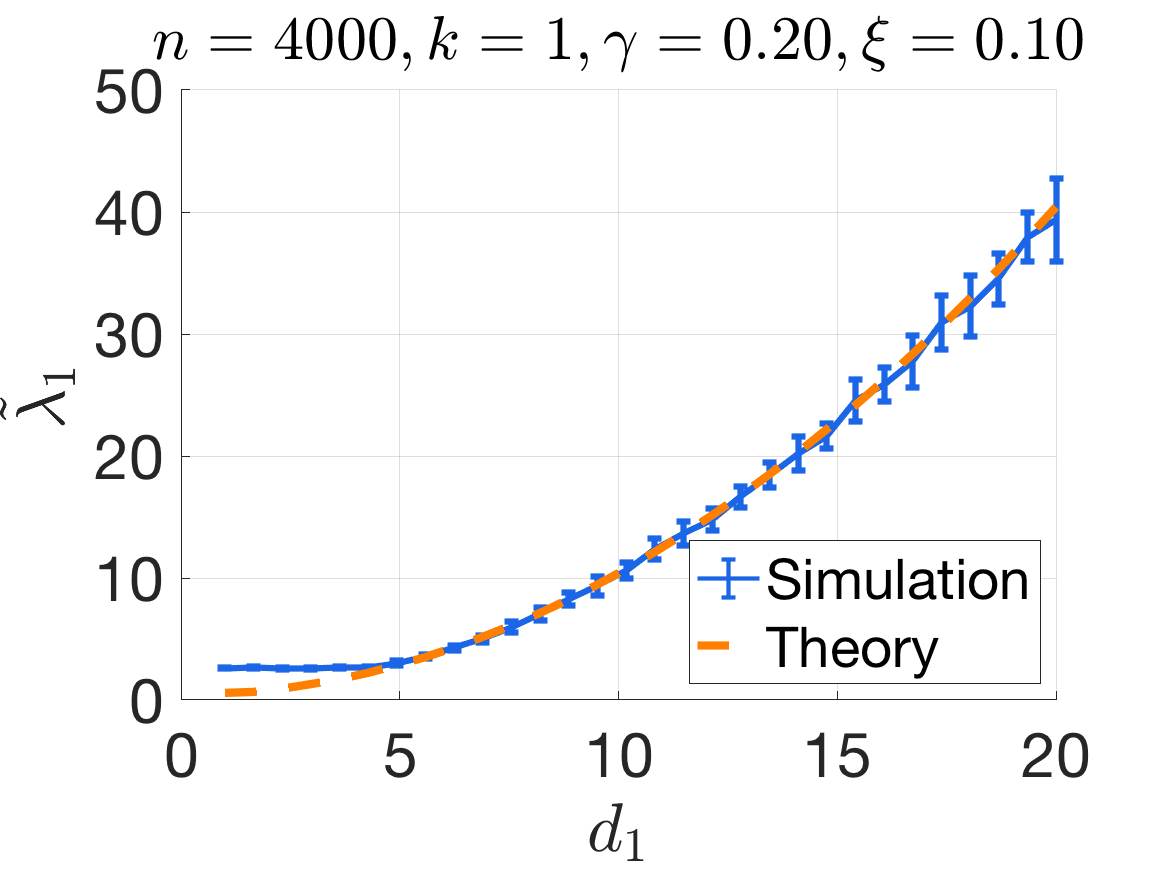}
\end{subfigure}
\caption{Checking the accuracy of the spiked eigenvalue and eigenvector formulas for large signals in a single-spiked model. We follow the protocol from the experiment in Figure \ref{f_sp_or}. Here $\Sigma=O^\top \Lambda O$, where $O$ is a $p\times p$ orthogonal matrix, $\Lambda$ is a diagonal matrix with $\Lambda_{11}=5$, $\Lambda_{ii}=2$ for $2\leq i\leq p/2$, $\Lambda_{ii}=1$ for $p/2<i\leq p$, and $d$ ranges from 1 to 20 with equal spaces. }
\label{fig: large signal 1}
\end{figure}
In the second example, we take $\Sigma$ to be the Toeplitz matrix with $\Sigma_{ij}=0.9^{|i-j|}$; see Figure \ref{fig: toep 0.9}. 
 {  In general, for a Toeplitz matrix whose $(i,j)$-th entry is $q^{|i-j|}$, we have $\rho_1=p^{-1}\tr(\Sigma)=1$ because the diagonal entries are all ones. For $\rho_2$, we have 
\begin{align*}
\rho_2=\frac{1}{p}\tr(\Sigma^2)=\frac{1}{p}\|\Sigma\|_{F}^2=\frac{1}{p} \sum_{i,j}\Sigma_{ij}^2.
\end{align*}
Among the $p^2$ entries of $\Sigma$, the $p$ diagonal entries are equal to 1; and for $i=1,\ldots,p-1$, there are $2(p-i)$ entries that are equal to $q^i$. Thus,
\begin{align*}
\frac{1}{p}\sum_{i,j}\Sigma_{ij}^2=1+\frac{1}{p}\sum_{i=1}^{p-1}2(p-i)q^{2i}=1+\frac{2}{p}\left[\frac{pq^2}{1-q^2}-\frac{q^2(1-q^{2p})}{(1-q^2)^2}\right],
\end{align*}
which converges to $\frac{1+q^2}{1-q^2}$ as $p$ goes to infinity.} We see a good match between the simulation and the theoretical result, as long as the signal strength $d_i$ is reasonably large. 
\begin{figure}[!htb]
\centering
\begin{subfigure}{.45\textwidth}
\includegraphics[width=\textwidth]{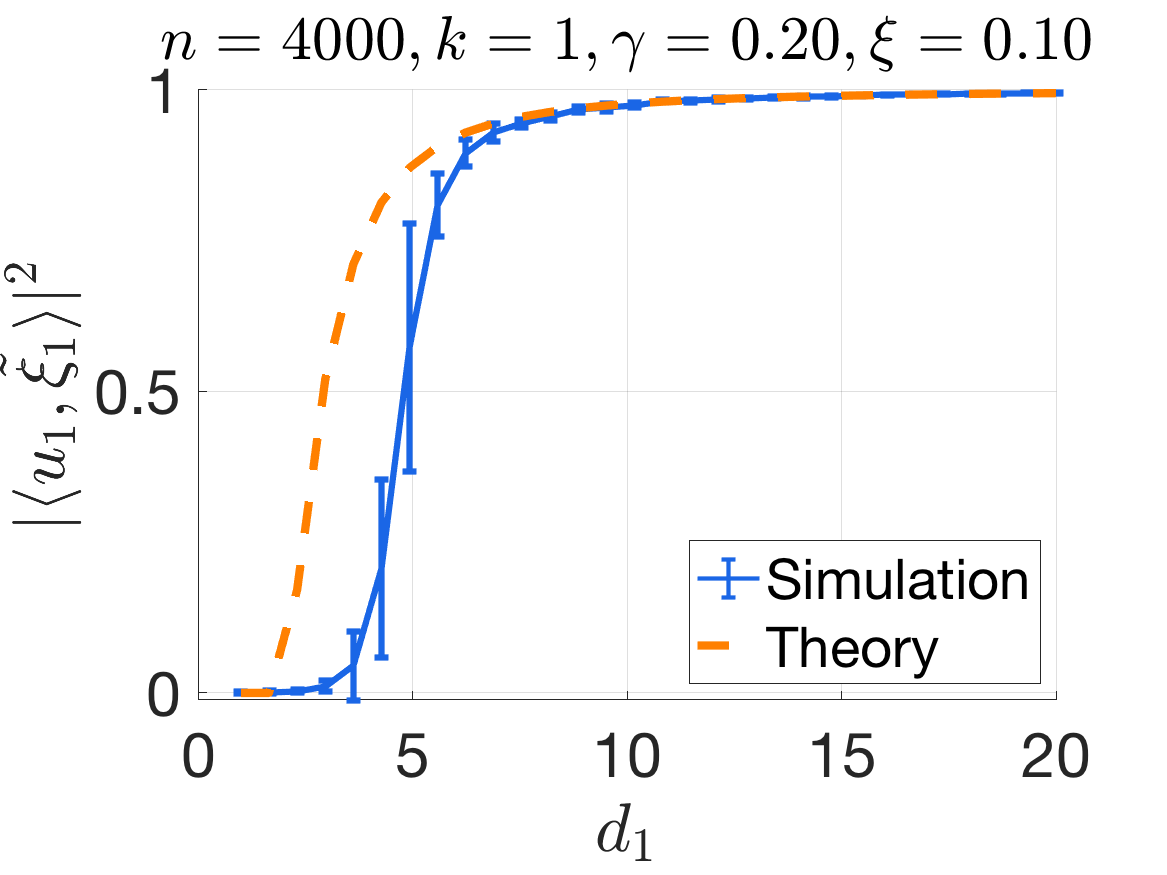}
\end{subfigure}
\begin{subfigure}{.45\textwidth}
\includegraphics[width=\textwidth]{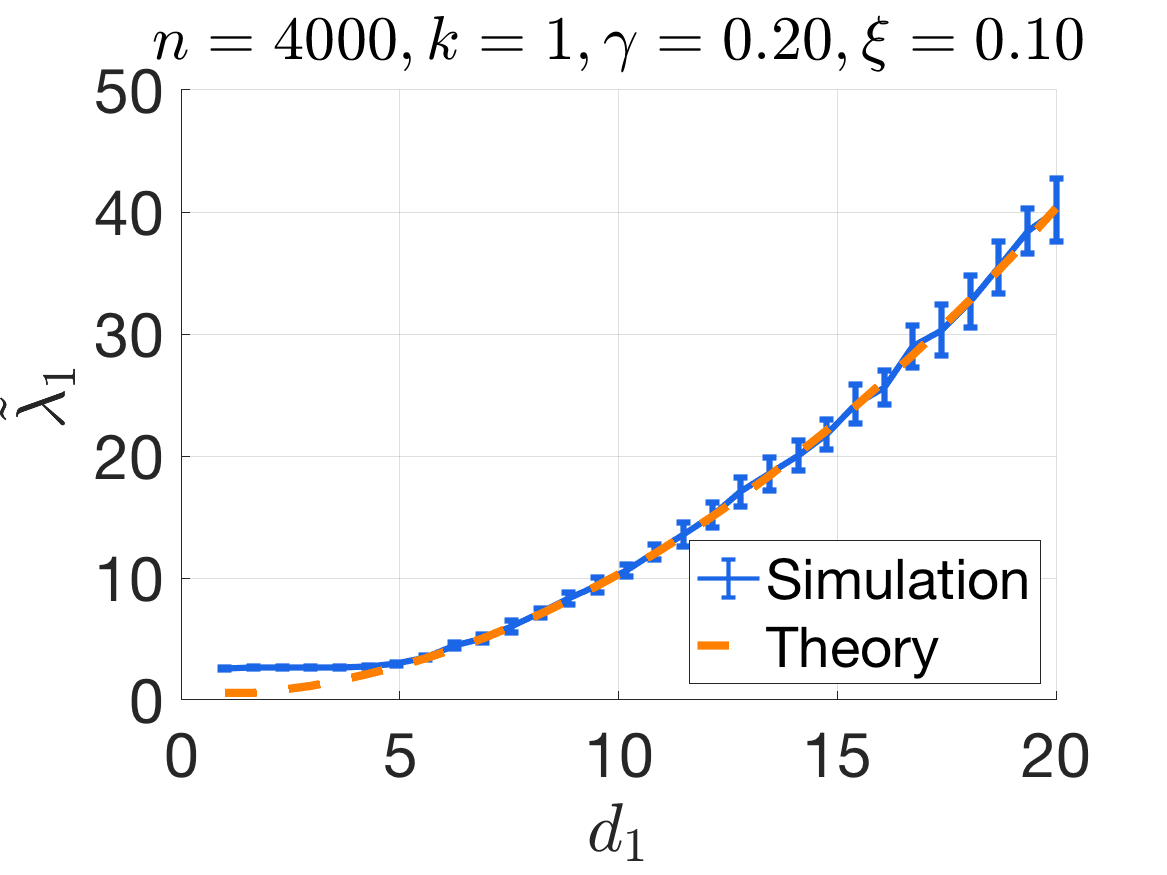}
\end{subfigure}
\caption{The protocol is the same as in Figure \ref{fig: large signal 1}, except that $\Sigma$ is the Toeplitz matrix with the $(i,j)$-th entry equal to $0.9^{|i-j|}$. }
\label{fig: toep 0.9}
\end{figure}


In fact, we can get more precise results by deriving higher order asymptotic expansions in terms of $l_i^{-1}$. The calculations are pretty straightforward, but tedious. We do not pursue this direction here. 

\begin{remark}\label{rem extra}
The two terms on the right-hand side of \eqref{largethetai} can be understood heuristically as follows. First, based on standard perturbation theory, the $i$-th largest eigenvalue of $\smash{\wt\Sigma}$ is about $ d_i^2 + u_i^{\top} \Sigma u_i= d_i^2 + E_{ii}$. As discussed in Section \ref{heur}, heuristically after projection into $r$-dimensional subspace, the signal strength should go down by a factor of $\xi$, which leads to the term $\xi( d_i^2 + E_{ii})$ in \eqref{largethetai}.

For the $\gamma\rho_1$ term, we consider the extreme case where $\xi\to 0$ and hence the signal strength goes down to zero. Without loss of generality, we assume that $S$ is random sampling. Then by concentration, one can see that  
$$(SX\Sigma^{1/2})(\Sigma^{1/2}X^{\top}S^{\top}) \approx \left(\frac{1}{n}\sum_{i=1}^p \Sigma_{ii}\right)I_{r\times r} = \gamma \rho_1 I_{r\times r}.$$
This leads to the  $\gamma\rho_1$ term that does not depend on $\xi$.

  Finally, for simplicity, suppose $U$ is a uniform partial orthonormal matrix. Then $E_{ii}$ is well-concentrated around $\rho_1$. Now we notice that for a fixed $\rho_1$, the right-hand side of \eqref{precision_rho2} becomes smaller as $\rho_2$ increases. In particular, the second moment of the spectral distribution of $\Sigma$ is minimized when it is degenerate (i.e. concentrates on one point). With a similar method, in the setting with $\Sigma=I_p$, a general sketching matrix $S$ and a uniform partial orthonormal $W$, we can derive that
$$|\langle u_i, \wt{\bxi}_i\rangle|^2 \to_p  \frac{ \xi- \frac{\rho_2}{\rho_1^2}\frac{ \gamma}{ d_i^{4}}}{\xi+\frac{\rho_2}{\rho_1^2} \frac{\gamma}{ d_i^{2}}}   + \OO (l_i^{-3}),$$
where $\rho_1$ and $\rho_2$ are the first and second moments of the spectral distribution of $ SS^\top $. We omit the details, since the derivation is similar to the one in Appendix \ref{sec_pflarge}. One can also compare it with \eqref{evectorcorr}. Hence it is better to use a sketched matrix with smaller $\rho_2/\rho_1^2$, which is minimized at $1$ when the spectral distribution of $ SS^\top $ is degenerate. This heuristically explains why projections with i.i.d. entries and CountSketch are slightly worse than other methods---they have less concentrated spectrum compared to other methods.
\end{remark}

To our knowledge, such a general model as in \eqref{generalSigma} has not been studied in the literature, even in the strong signal regime. In the classical setting, it is usually assumed that $\Sigma$ is identity or a finite rank perturbation of identity matrix; see \citep[e.g.,][etc]{spikedmodel, baik2005phase,baik2006eigenvalues, Benaych2012,ding2019distri,Ding2020}. This is also our setting in Sections \ref{sec unifproj}-\ref{sec count}. Another type of spiked covariance model has the spikes added to the population covariance matrices directly; see \cite[e.g.,][etc]{paul2007asymptotics,benaych2011eigenvalues,principal,ding2019spiked}. That model is $ X\wh\Sigma^{1/2}$, where $\smash{\wh\Sigma}$ is a spiked covariance matrix of the form \eqref{generalSigma2}. By diagonalizing the matrix $\smash{\wh\Sigma}$, one can assume that $u_i$-s are also eigenvectors of $\smash{\wh\Sigma}$, which is more restrictive than our model \eqref{generalSigma}. Thus we believe that the Theorem \ref{sketchthmlarge} and the methods used in its proof may be of independent theoretical interest.

\section{Empirical verification}\label{sec_emp}
\subsection{Proposed method}
We aim to verify our results empirically. In previous work for linear regression \citep{dobriban2018new} we developed formulas for the behavior of the OLS residuals under sketching. We predicted the behavior of the ratio of residuals, as a function of the known quantities $n,p,r$ only. This idea is similar to constructing a pivotal random variable, whose behavior does not depend on un-measured quantities. We tested these formulas empirically. Surprisingly, we found that the ratio of residuals can be close to the predicted value in empirical datasets.

Here we do not have the direct analogue of the residuals. However, we can work from first principles to derive a similar method. We know that the top eigenvalues in standard spiked models follow the spiked forward map $\ell \to \lambda(\ell, \gamma)$ from equation \eqref{know111}, see e.g., \cite{baik2005phase,baik2006eigenvalues}. A well known method to estimate the spike is to invert this. These methods have been implemented in the \verb|EigenEdge| package \citep{Dobriban2015}.

Thus, we propose to calculate the inverse 
both for the original and sketched data. In our model, both should be close to $\ell$.
Then we divide them, leading to the statistic

$$T=\frac{\lambda^{-1}\left(\sigma_1(X)^2,p/n\right)}{\lambda^{-1}\left(\sigma_1(SX)^2,p/r\right)}.$$
Our theoretical results predict that this should be close to unity, i.e., we should have $T\approx 1$.

\subsection{Datasets tested}

We consider three data sets to test our theoretical results: the Human Genome Diversity Project (HGDP) dataset \cite[e.g.,][]{cann2002human,li2008worldwide}, the Million Song Dataset (MSD, \cite{Bertin-Mahieux2011}) and New York Flight Dataset (\cite{nycflights13}). 
For each, we take uniform orthogonal random projections on the data with {  $r = \lfloor \xi n \rfloor$, with $\xi = 0.8, 0.5, 0.3$.} For HGDP, we repeat this while subsampling (1) every 20th column; (2) every 10th row and 20th column.

{  For some context, the purpose of collecting the HGDP dataset was to evaluate the diversity in the patterns of genetic variation across the globe. We use the CEPH panel, in which single nucleotide polymorphism (SNP) data was collected for 1043 samples representing 51 different populations from Africa, Europe, Asia, Oceania and the Americas. We obtained the data from \url{www.hagsc.org/hgdp/data/hgdp.zip}. We provide the data and processing pipeline on this paper's GitHub page.

The data has $n=1043$ samples, and we focus on the $p=9730$ SNPs on chromosome 22. Thus we have an $n\times p$ data matrix $X$, where $X_{ij}\in\{0,1,2\}$ is the number of copies
of the minor allele of SNP~$j$ in the genome of individual~$i$. We standardize the data SNP-wise, centering each SNP by its mean, and dividing by its standard error. For this step, we ignore missing values. Then,  we impute the missing values as zeroes, which are also equal to the mean of each SNP.}

First, for the HGDP dataset, we have seen in previous work, that it is not well modeled by a matrix with iid Gaussian entries \citep{dobriban2019deterministic}. In particular, there are correlations both between the columns as well as between the rows. Despite this model mismatch,  for $\xi=0.8$ we get values of $T$ between 1.2 and 1.4 on this dataset, which are quite close to the expected value of unity under correct model specification. This suggests that our theory may sometimes be applicable and relevant even when the data do not follow the theoretical model. 

We also consider the Million Song Dataset \citep{Bertin-Mahieux2011} and New York Flights Dataset \citep{nycflights13}. For $\xi=0.8$, we get $T=1.26$ and $T=1.24$, respectively. {  However, for both datasets, if we use $r = 0.5 n$, then $T$ is about 2; if $r = 0.3 n$, $T$ is about 3. This also suggests that, in these datasets where the assumptions does not hold, the theoretical results become somewhat less accurate as $r/n$ decreases.}

\subsection{Simulation for multi-spike model}
\label{sec: simulation k=5}

From Section \ref{sec: orthogonal} to Section \ref{strong}, the simulations all concern the single-spiked model with rank $k=1$.  Here we verify our theoretical results on a multi-spiked model with $k=5$. The results for $|\langle u_i,\widetilde \xi_i\rangle|^2$ and $\widetilde \lambda_i$ are shown in Figures \ref{fig: check k=5 cos} and \ref{fig: check k=5 lambda} respectively. We see that the formulas are very accurate for $\widetilde \lambda_i$ , but less accurate for $|\langle u_i,\widetilde \xi_i\rangle|^2$ for large signals. Heuristically, this is because the variance of  $|\langle u_i,\widetilde \xi_i\rangle|^2$ increases compared to the single-spike case due to the repulsion between different spikes, see e.g. \cite{paul2007asymptotics}.
\begin{figure}
\centering
\begin{subfigure}{.3\textwidth}
\includegraphics[width=\textwidth]{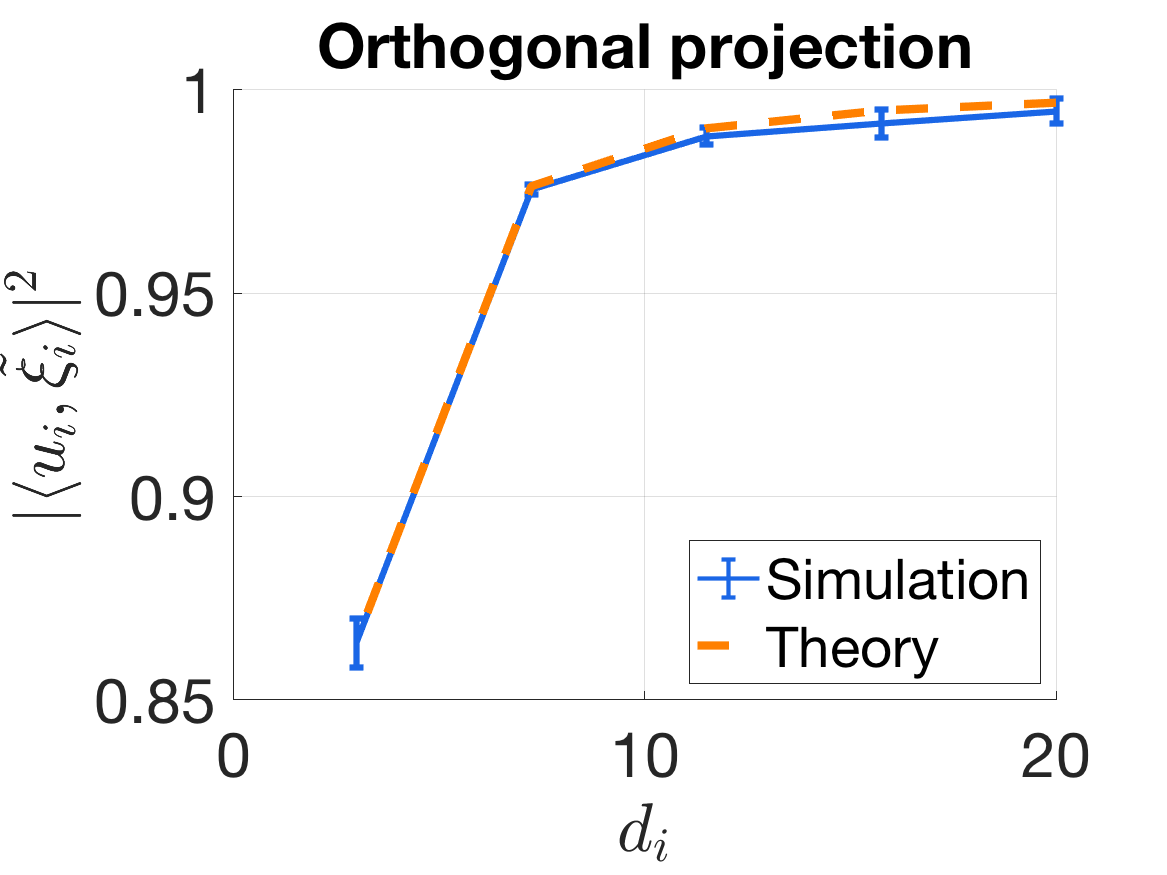}
\end{subfigure}
\begin{subfigure}{.3\textwidth}
\includegraphics[width=\textwidth]{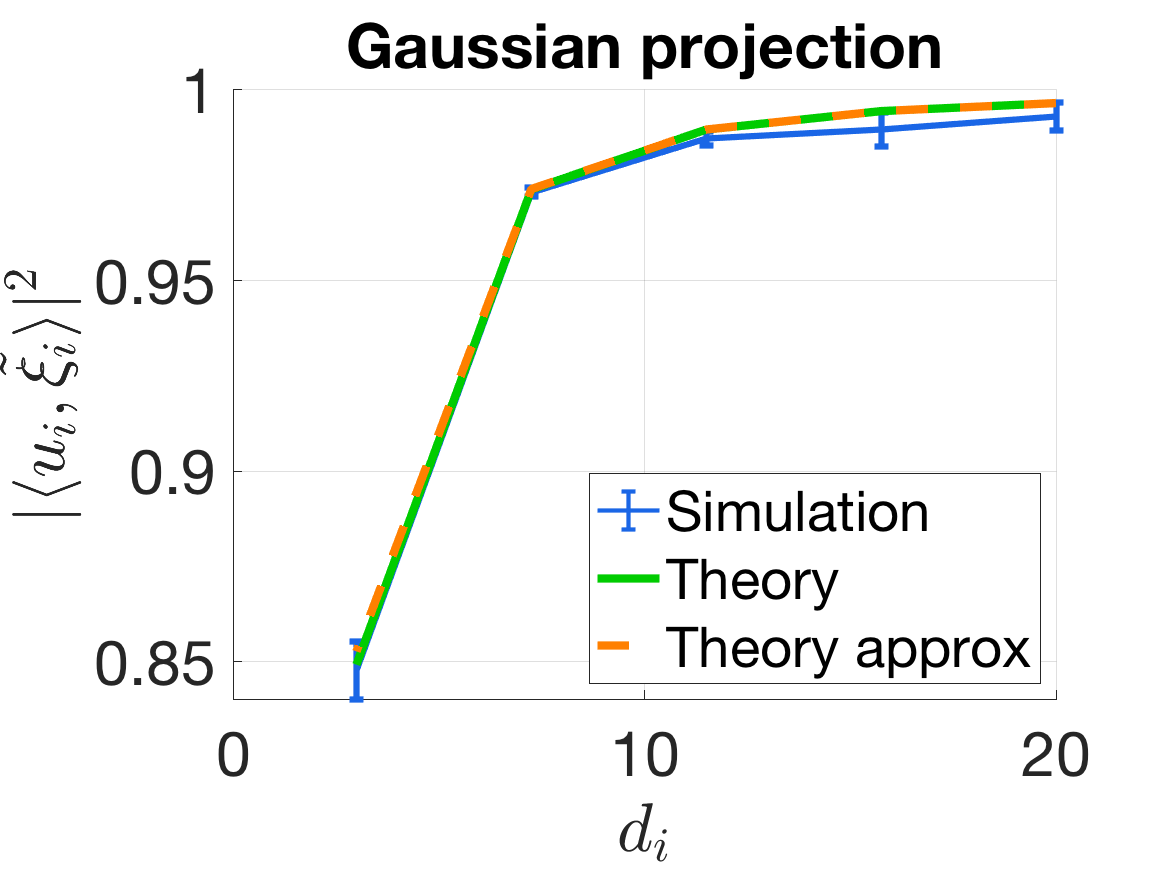}
\end{subfigure}
\begin{subfigure}{.3\textwidth}
\includegraphics[width=\textwidth]{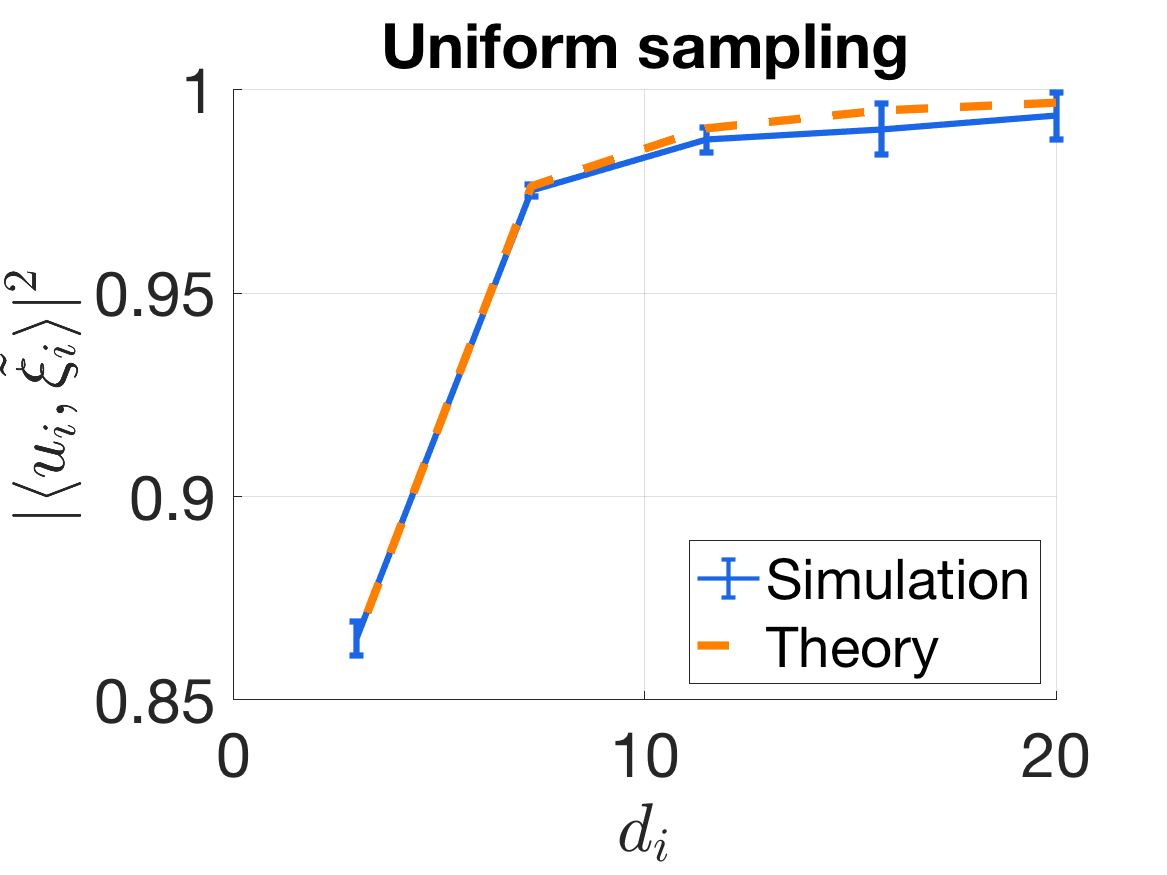}
\end{subfigure}
\begin{subfigure}{.3\textwidth}
\includegraphics[width=\textwidth]{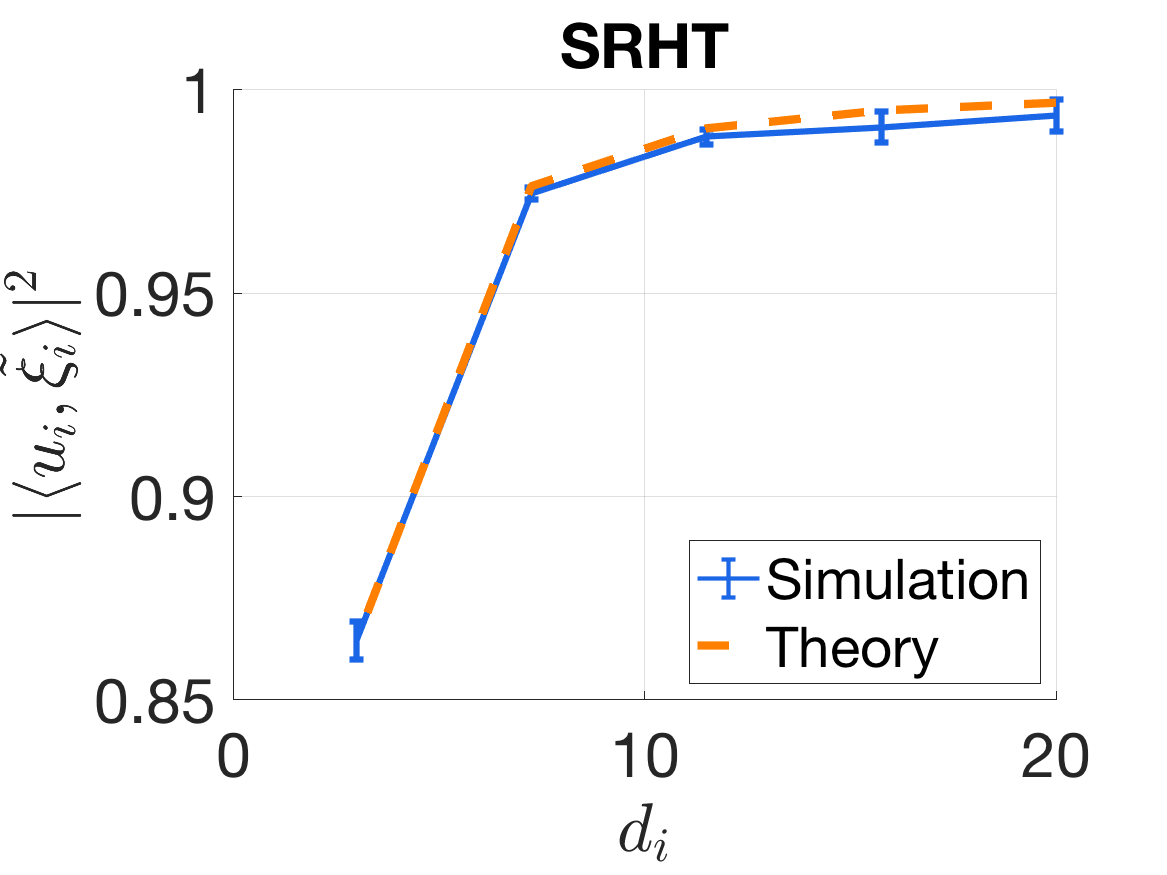}
\end{subfigure}
\begin{subfigure}{.3\textwidth}
\includegraphics[width=\textwidth]{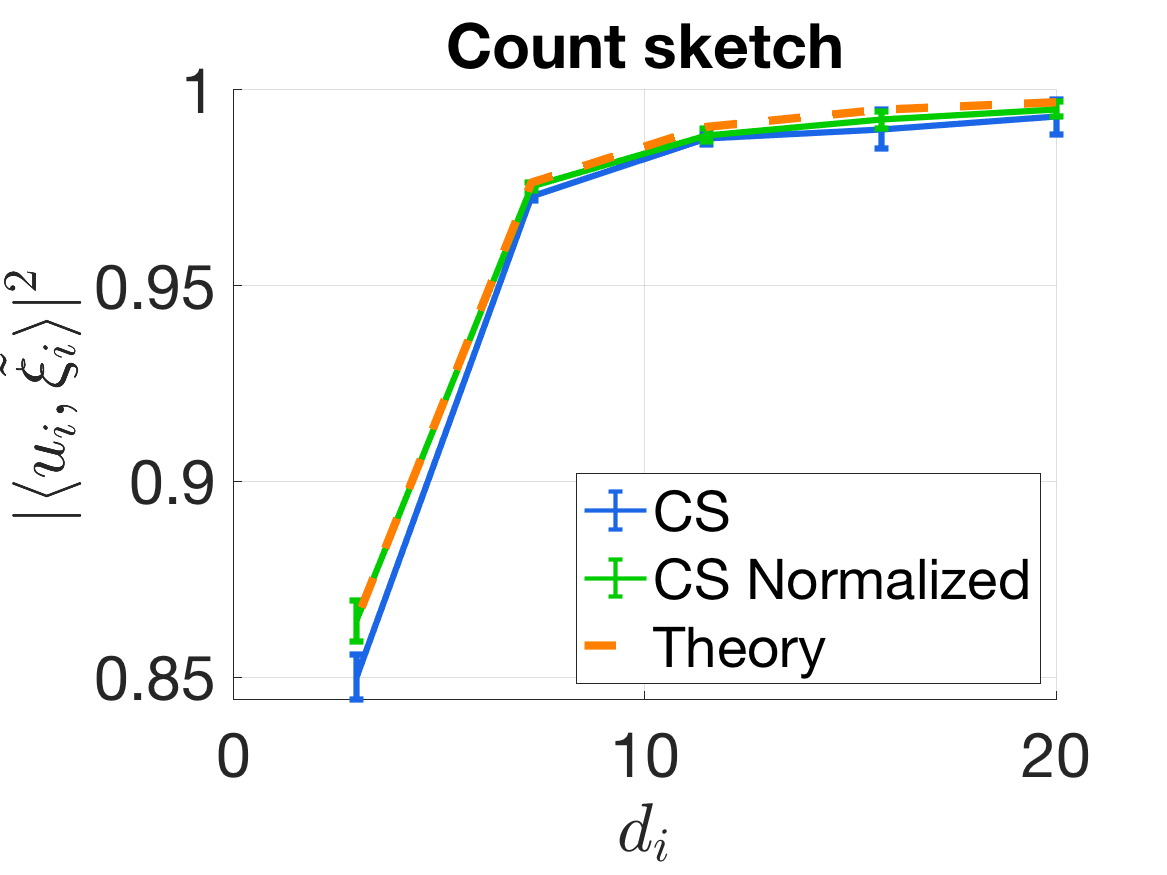}
\end{subfigure}
\caption{Checking the accuracy of the formulas for $|\langle u_i,\widetilde \xi_i\rangle|^2 $ for different sketching methods. We take $n=20000$, $p=2500$, $r=2000$, $k=5$. We plot the mean and standard error over 20 repeated experiments.}
\label{fig: check k=5 cos}
\end{figure}
\begin{figure}
\centering
\begin{subfigure}{.3\textwidth}
\includegraphics[width=\textwidth]{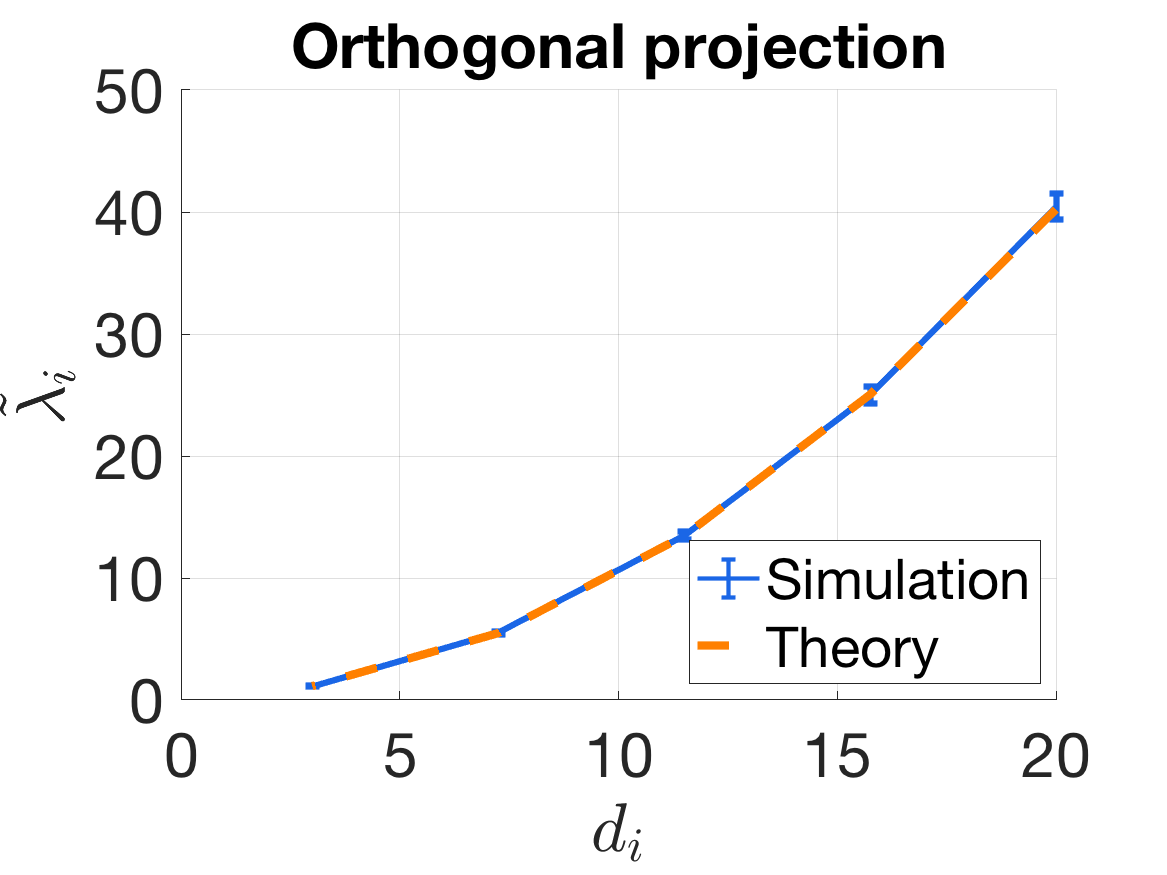}
\end{subfigure}
\begin{subfigure}{.3\textwidth}
\includegraphics[width=\textwidth]{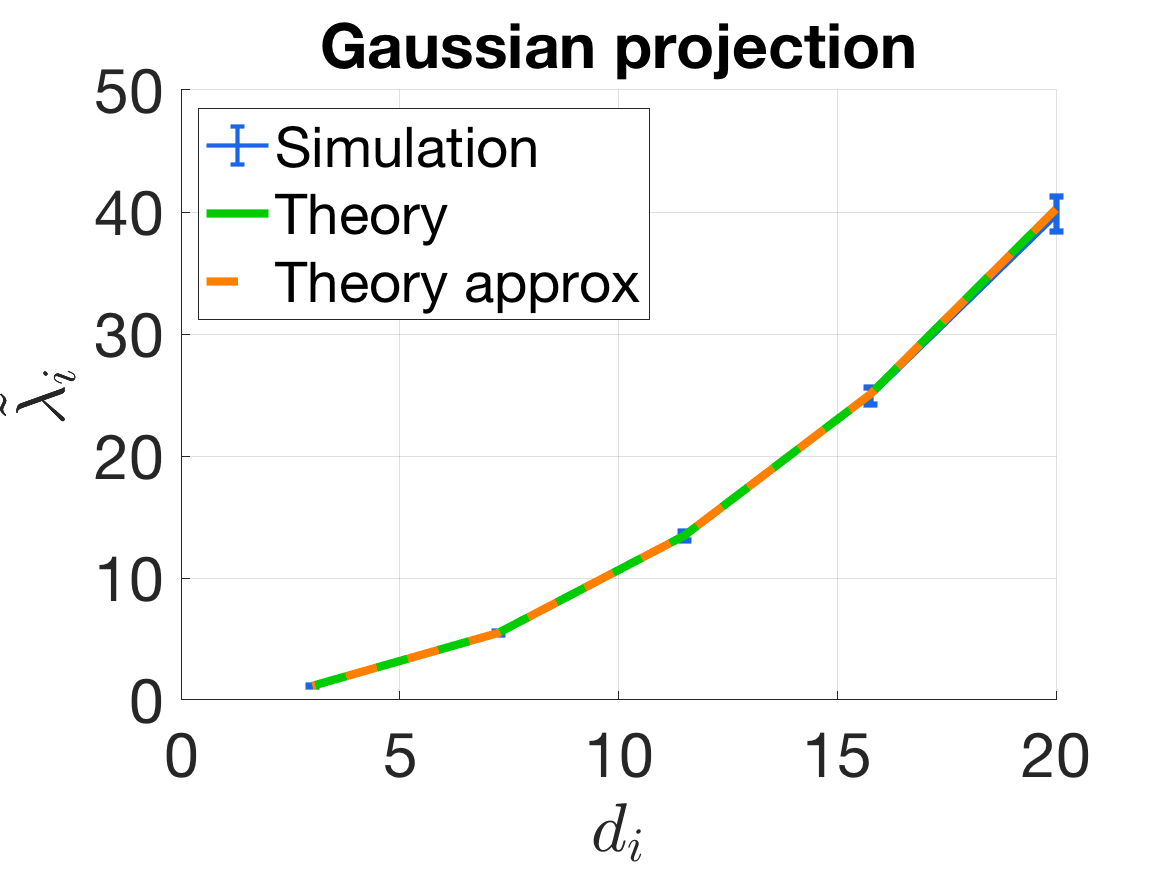}
\end{subfigure}
\begin{subfigure}{.3\textwidth}
\includegraphics[width=\textwidth]{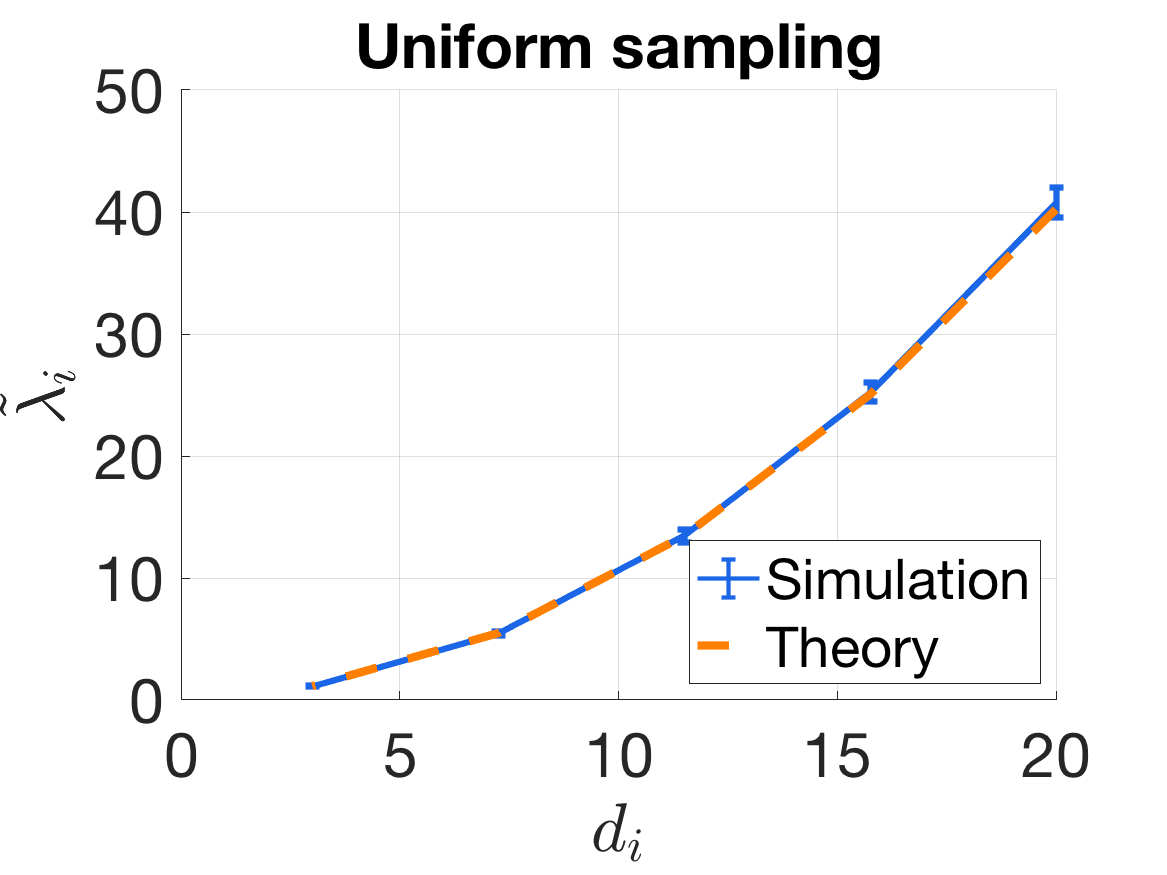}
\end{subfigure}
\begin{subfigure}{.3\textwidth}
\includegraphics[width=\textwidth]{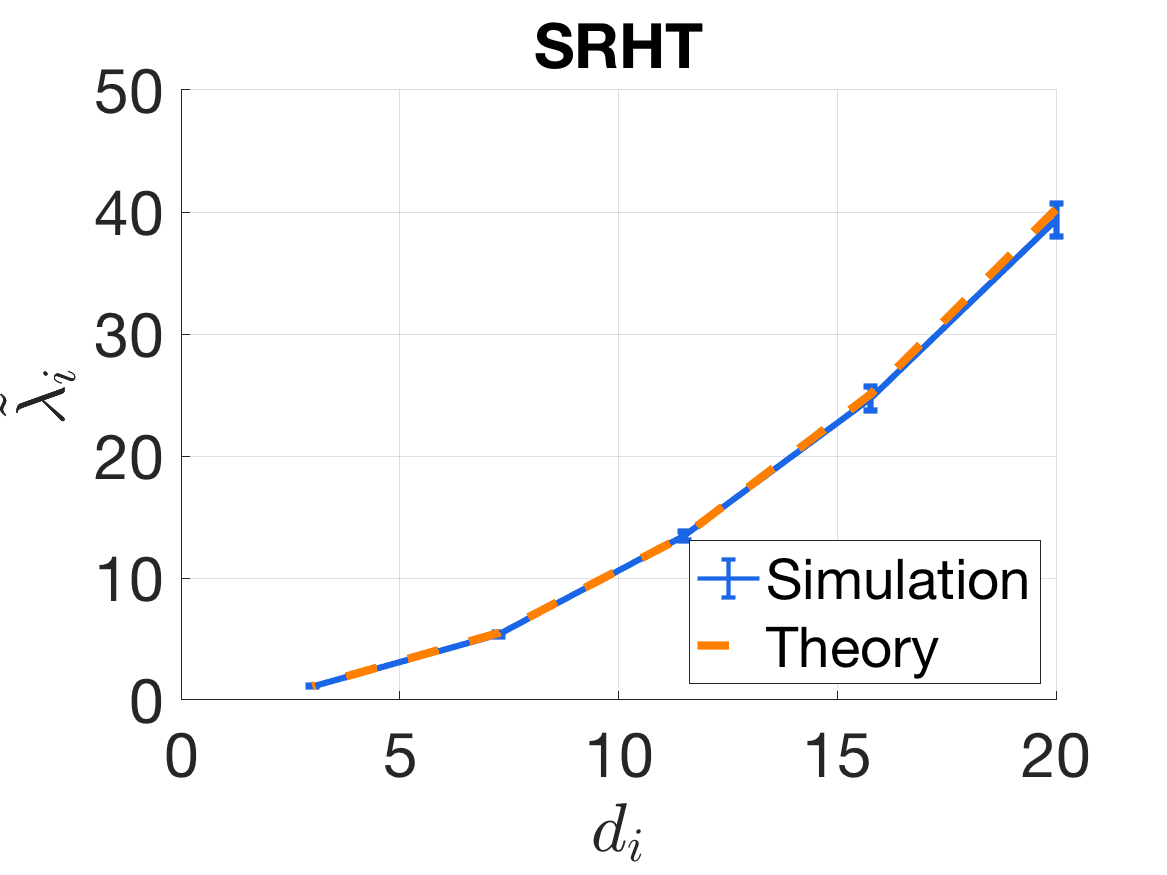}
\end{subfigure}
\begin{subfigure}{.3\textwidth}
\includegraphics[width=\textwidth]{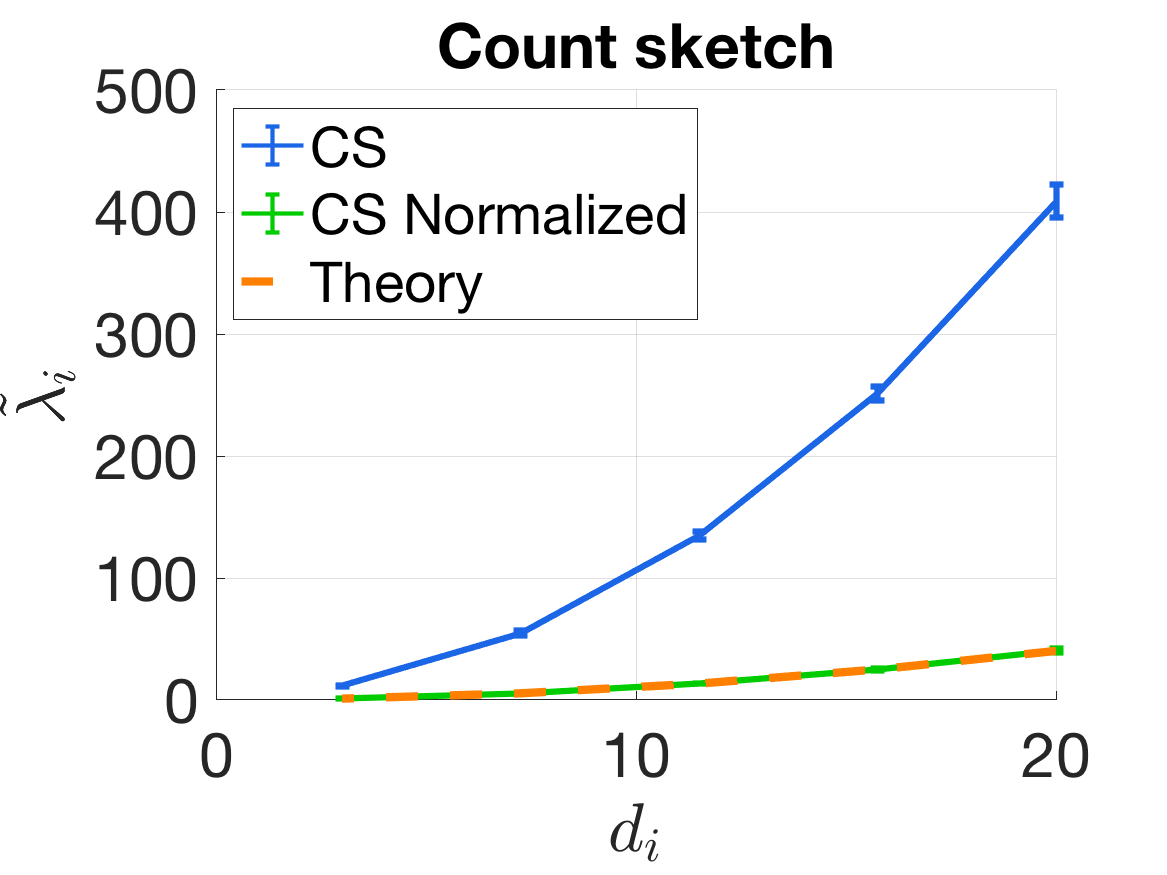}
\end{subfigure}
\caption{Checking the accuracy of the formulas for $\widetilde \lambda_i$ for different sketching methods. The protocol is the same as Figure \ref{fig: check k=5 cos}.}
\label{fig: check k=5 lambda}
\end{figure}

\section{Discussion}\label{sec_disc}

{  We have chosen to present results with convergence in probability in this paper, because we want our assumptions to be as general as possible. However, with certain stronger assumptions, it is possible to improve the results to almost sure convergence.

One obstacle for improving the result is the convergence of the block $W^{\top}S^{\top} (1+m_{1c}(x)SS^{\top})^{-1}SW $ of the master matrix $M(x)$ in equation \eqref{defnMx}. In the manuscript, we only show the in probability convergence of this matrix. To have a stronger convergence for uniform random sampling and CountSketch, we need stronger conditions on the vectors $w_i$ than \eqref{delocal}. Take uniform random sampling as an example. Suppose we assume a stronger delocalization condition on the vectors $w_i$,
$$\max_{1\le i \le k}\|w_i\|_\infty \le c(\log n)^{-1} \quad \text{for a small enough constant $c>0$.}$$
Then with a Chernoff type matrix concentration bound (or Bernstein inequality), we can obtain that 
$$\left\| W^\top S^\top SW - \frac{r}{n}I_k\right\|={\rm O}((\log n)^{-1/2}),\quad \text{with probability } 1-{\rm O}(e^{-C\log n}),$$
for a large constant $C>0$. Then using Borel-Cantelli lemma, we can show that $W^\top S^\top SW$ converges almost surely to $ \frac{r}{n}I_k$.

Another obstacle is the moment assumption on the matrix entries of $X$. Right now, we assume that the entries of $X$ have finite $(4+\epsilon)$-th moment as in equation \eqref{eq_highmoment}. Then we can truncate the entries as in \eqref{defwhX} such that for a small constant $\delta,$ $\max_{i,j} |x_{ij}|\le n^{-\delta}$ on an event with probability $1-{\rm O}(n^{-\delta})$. The bounded entry condition $\max_{i,j} |x_{ij}|\le n^{-\delta}$ is necessary for our Theorems \ref{lem_localout}  and \ref{lem_localin}, but the probability $1-{\rm O}(n^{-\delta})$ is not good enough. To improve this probability, we need to assume a stronger moment assumption. For example, if the entries of $X$ have finite $(6+\epsilon)$-moment, then  we can truncate the entries of $X$ such that for a small constant $\delta,$ $\max_{i,j} |x_{ij}|\le n^{-\delta}$ on an event with probability $1-{\rm O}(n^{-1-\delta})$. Then we can use Borel-Cantelli lemma to improve the results to almost sure convergence.}

In addition, in future work, it may be of interest to investigate other sketching methods that have been proposed. In particular, uniform sampling can work poorly when the data are highly non-uniform, because some datapoints are more influential than others for the PCs. There are more advanced sampling methods that sample each row of $X$ with some non-uniform probability $\pi_i$ which relates to the importance of the $i$th sample, such as $d_2$ sampling \citep{drineas2006sampling}, where $\pi_i$ is proportional to the squared norm of the $i$th row, or leverage score sampling, where the scores are proportional to the leverage scores \citep{mahoney2011randomized,ma2015statistical,chen2016statistical}. 

Another frequently used type of random projections are the so-called oblivious sparse norm-approximating projections (OSNAPs) \citep{Kane2010,OSNAP}. More precisely, an $r\times n$ random projection matrix $S$ is an OSNAP if $S_{ij}=\delta_{ij}\sigma_{ij}/\sqrt{s}$, where $s\ge 1$ is a fixed integer, $\sigma_{ij}$ are random signs, and $\delta_{ij}$ are indicator random variables satisfying the following properties:
\begin{itemize}
\item fixed number of nonzeros per column: for any $1\le j \le n$, $\sum_{i=1}^r \delta_{ij}=s$ with probability 1;

\item negative correlation between the nonzeros: for any $E\subset \{1,\cdots, r\}\times \{1,\cdots, n\}$, $\E \prod_{(i,j)\in E}\delta_{ij}\le (s/r)^{|E|}$.
\end{itemize}
A concrete example is when we independently choose $s$ nonzero locations for each column, uniformly at random over all possible subsets of size $s$. 

The difficulty in analyzing leverage score sampling and OSNAP lies in a complete understanding of the exact ESD of $SS^\top$, which is needed in both the study of the self-consistent equations in \eqref{sc} and the matrix \eqref{defnMx}. However, if the signal strengths are strong, then it is possible to obtain some approximate results using the argument in Theorem \ref{sketchthmlarge}, where only the first few moments of the ESD of $SS^\top$ will be needed.

We can compare leverage score sampling and OSNAP with the sketching methods analyzed in Section \ref{sec5types} through simulations; see Figure \ref{fig: compare all}. We take $n=4000$, $p=800$, $k=8$ with varying signal strengths and different $r$. The error bars are the standard deviations over 20 independent repetitions. We plot the overlap between the spiked population eigenvector and the sample eigenvector after sketching. Large overlaps show the methods are better at preserving the eigenspace of the signal. {  From the figures in Figure \ref{fig: compare all}, we observe the following common phenomena}:
\begin{itemize}
\item Haar projection, uniform sampling, subsampled randomized Hadamard transform, and normalized CountSketch have roughly the same efficiency, and they are all better than i.i.d. Gaussian projection, as discussed in Section \ref{sec Gauss}. Moreover, unnormalized CountSketch is less accurate than normalized CountSketch. 
(However, of course, CountSketch can have other advantages like running time adapted to input sparsity.)

\item Leverage score sampling behaves similarly to uniform sampling. Note that this is related to the choice of model in this paper. 
When the data model is highly non-uniform, we expect that leverage score sampling will be better than uniform sampling.

\item OSNAP is less accurate than all the other methods. Again, this method can have other advantages, like near-optimally small $r$ to ensure oblivious subspace embedding for sparse inputs. But the gap between OSNAP and other sketching methods gets smaller as $r$ increases.

\end{itemize} 
For Haar projection, uniform sampling, subsampled randomized Hadamard transform, and normalized CountSketch, the ESD of $SS^\top$ is a singleton at $1$. On the other hand, for i.i.d. Gaussian projection, unnormalized CountSketch and OSNAP, the ESD of $SS^\top$ is supported on an interval around 1. Thus based on the simulations and the discussion in Remark \ref{rem extra}, we see that in order to better preserve the eigenspace of the signal, it is better to have a more ``concentrated" ESD for $SS^\top$.

\begin{figure}
\centering
\begin{subfigure}{.45\textwidth}
\includegraphics[width=\textwidth]{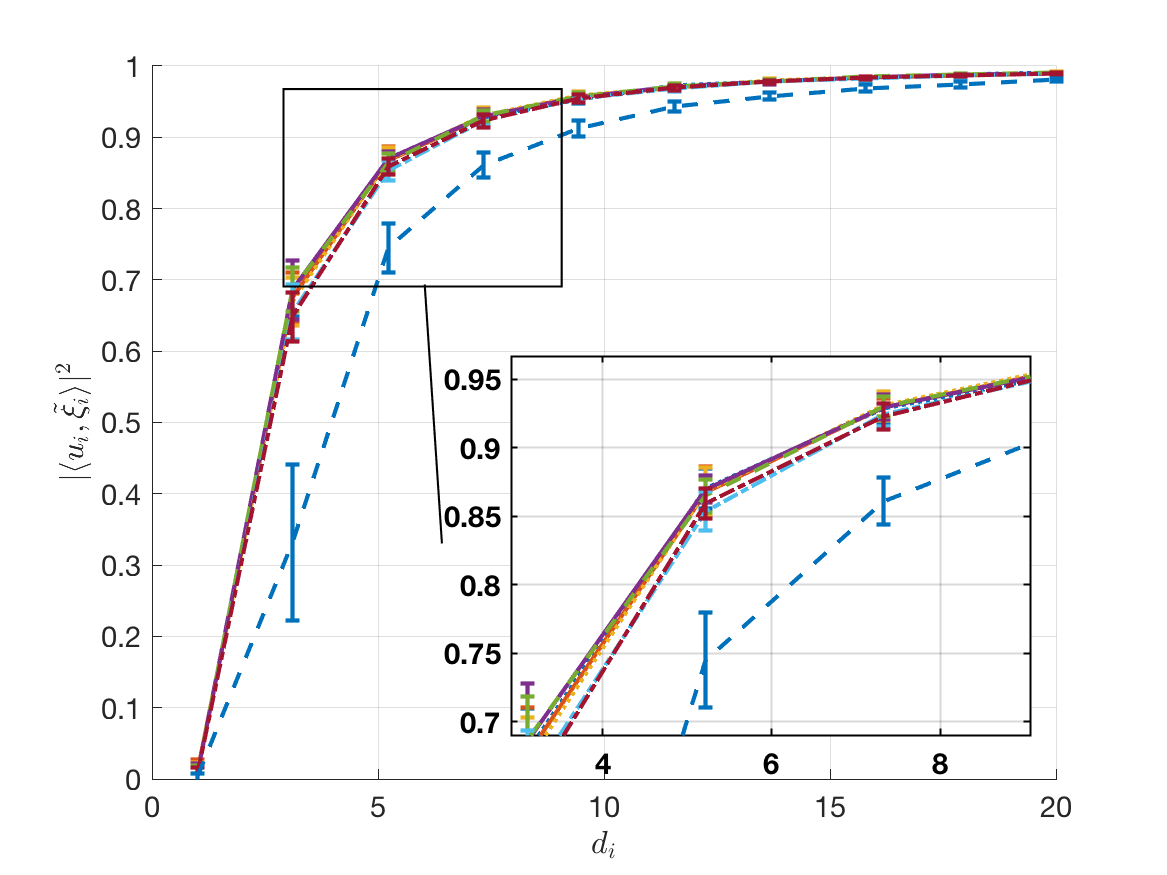}
\end{subfigure}
\begin{subfigure}{.45\textwidth}
\includegraphics[width=\textwidth]{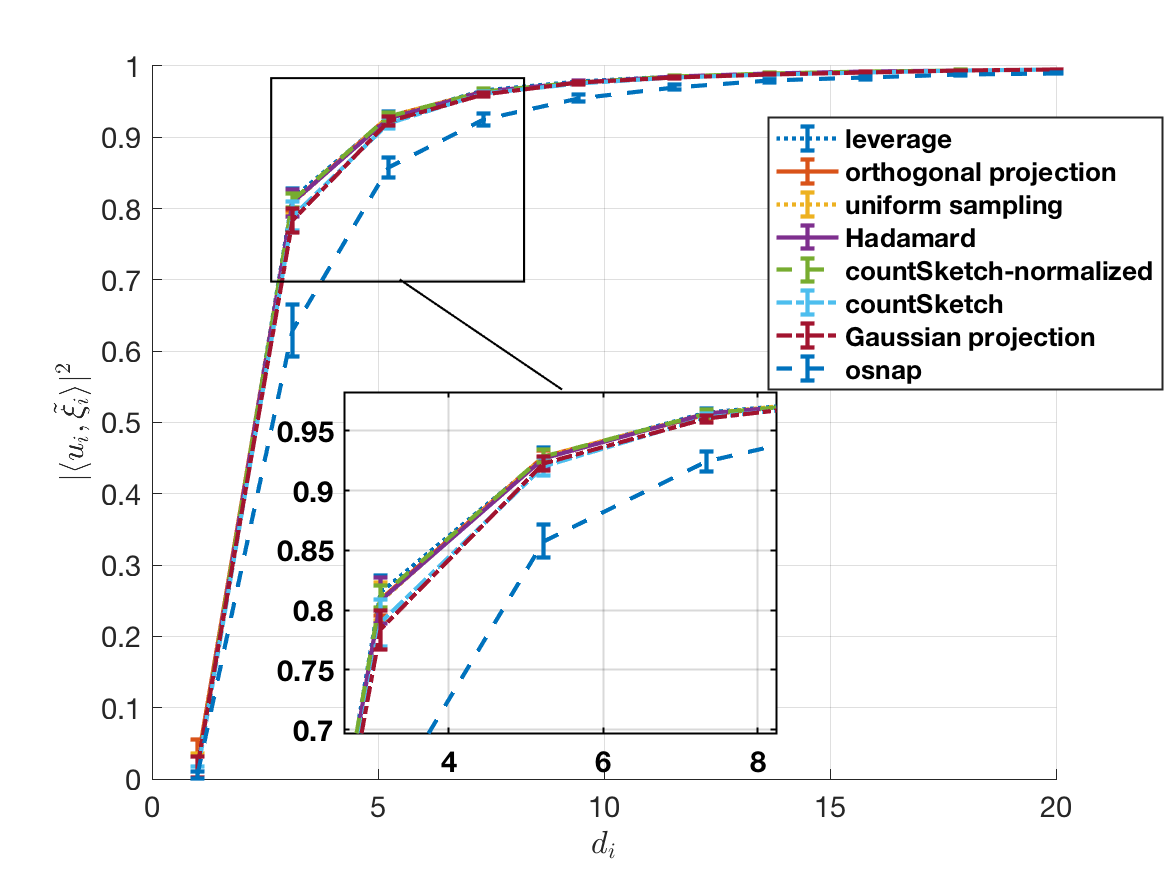}
\end{subfigure}
\begin{subfigure}{.45\textwidth}
\includegraphics[width=\textwidth]{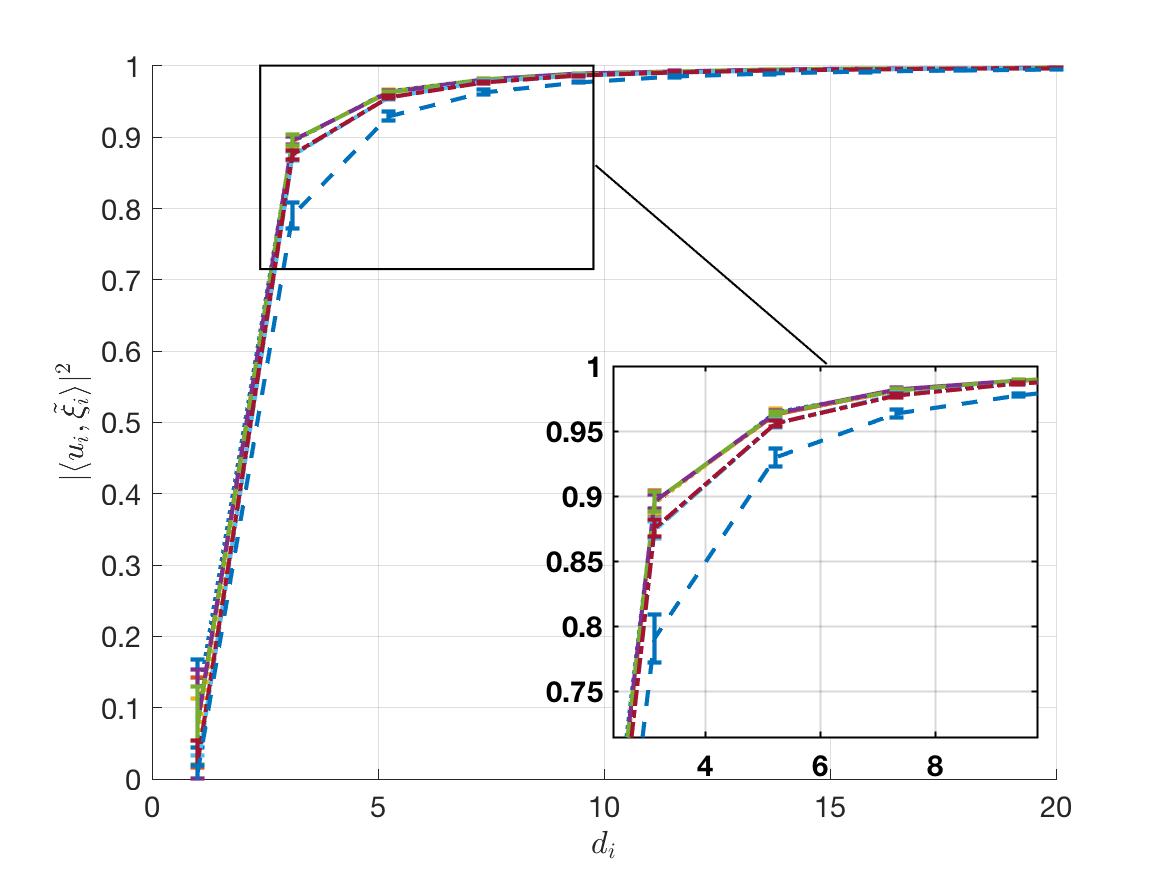}
\end{subfigure}
\begin{subfigure}{.45\textwidth}
\includegraphics[width=\textwidth]{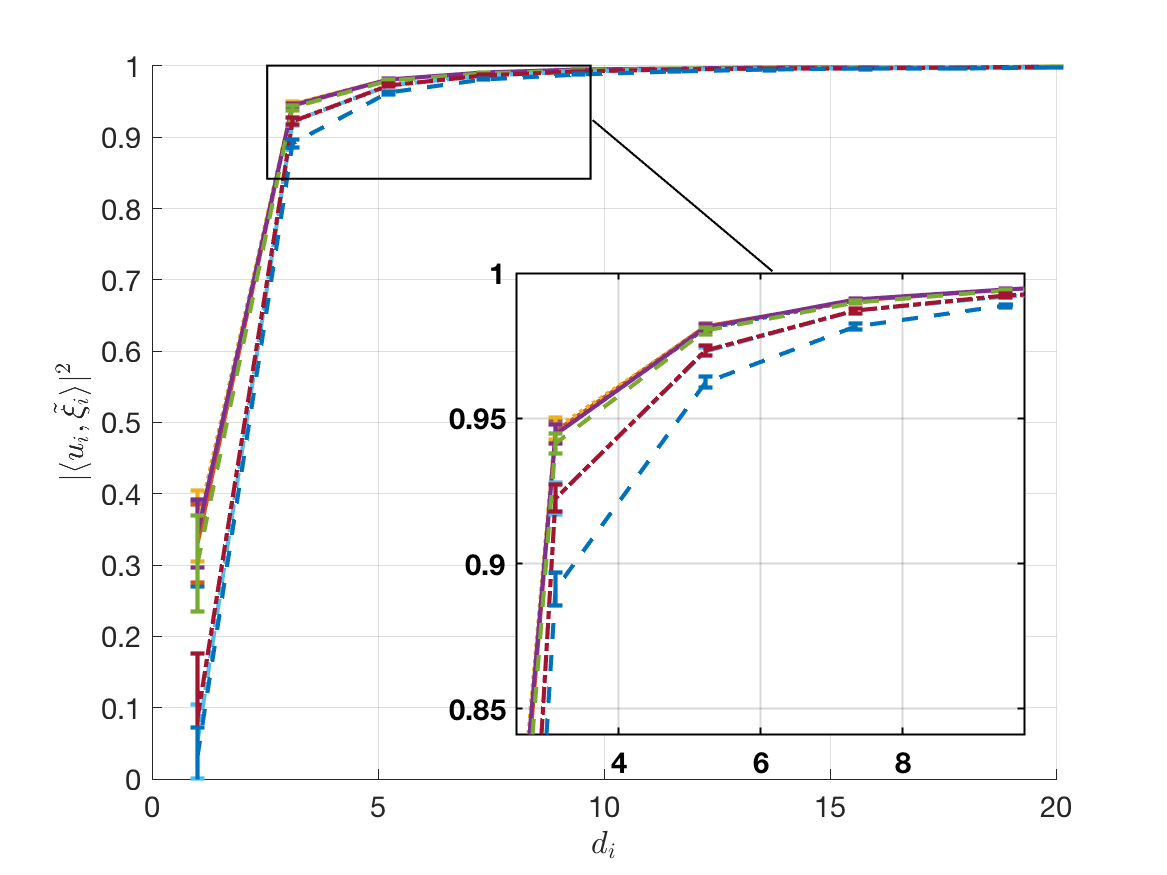}
\end{subfigure}
\caption{Comparing different sketching methods. Here $n=4000$, $p=800$, $k=8$, $r\in\{200, 400, 800, 1600\}$, and $X_{ij}\stackrel{iid}{\sim}\mathrm{Unif}(-\sqrt{3/n},\sqrt{3/n})$. The error bars are the standard deviations over 20 independent repetitions. We order the methods according to their accuracy.  
}
\label{fig: compare all}
\end{figure}

Finally, it could be of interest to generalize the argument to methods designed for the streaming data setting, such as core sketching \citep{tropp2017practical}, and iterative methods that can achieve arbitrary accuracy, such as blanczos \citep{rokhlin2009randomized,halko2011finding} and randomized block Krylov iteration \citep{Musco2015}.

\appendix
\section*{Acknowledgements}

We are grateful to the associate editor and referees, whose feedback has resulted in a significant improvement. We thank Ricardo Guerrero for valuable discussions. This work was supported in part by NSF BIGDATA grant IIS 1837992 and NSF TRIPODS award 1934960.

\section{The sketched spiked model}\label{sec_defspike}

\subsection{The model}\label{sec_defmodel}

Consider the ``signal plus noise" or ``spiked covariance" matrix model
$$ Y= \sum_{i=1}^k d_i w_iu_i ^{\top}+ X \Sigma^{1/2}.$$
Here $\sum_{i=1}^k d_i w_iu_i ^{\top}$ is the signal component, $\{d_i\}_{1\le i \le k}$ give the strengths of the signals, and $\{w_i\}_{1\le i \le k}$ and $\{u_i\}_{1\le i \le k} $ are the left and right singular vectors of the signals, respectively. Also, $X \Sigma^{1/2}$ is the noise component, where $\Sigma$ is a $p\times p$ deterministic covariance matrix, and $X=(x_{ij})$ is an $n\times p$ random matrix, where the entries $x_{ij}$, $1 \leq i \leq n$, $1 \leq j \leq p$, are real independent random variables satisfying
\begin{equation}\label{eq_12moment} 
\mathbb{E} x_{ij} =0, \ \quad \ \mathbb{E} \vert x_{ij} \vert^2  = n^{-1}.
\end{equation} 
We assume that the signal is independent of the noise matrix $X$.  Such signal plus noise or spiked models have been widely studied.  The special case $\Sigma=I_p$ is known as the \emph{standard} or (\emph{Johnstone's}) spiked model, \citep{spikedmodel}, and more general spiked models have been proposed and studied, see \cite{yao2015large,paul2014random,couillet2011random} and references therein. When $\Sigma$ is diagonal, and when $w_i$ have iid entries, this model can be viewed as a specific \emph{factor model}, and thus has a long history see e.g., \cite{spearman1904general,thurstone1947multiple,anderson1958introduction,dobriban2017factor}. This model is fundamental for understanding principal component analysis (PCA), and has been thoroughly studied under high-dimensional asymptotics. Its understanding will serve as a baseline in our study.

In this paper, we are interested in the PCA of the \emph{sketched} data matrix 
\beqs\wt Y=  SY \eeqs
where $S$ is an $r\times n$ random \emph{sketching matrix} that is independent of both the signal and the noise. This can be written as 
\be\label{sketchspike} \wt Y=  S X \Sigma^{1/2} + \sum_{i=1}^k d_i v_iu_i ^{\top} , \quad v_i\equiv Sw_i.\ee 
A similar spiked separable model has been studied in \cite{ding2019spiked}, although the setting there is a little different from our current setting, 
because the spikes are added to the population covariance matrices there. 
However, we will still follow the presentation from \cite{ding2019spiked} to some extent. 
We will study the spiked eigenvalues and eigenvectors of 
$$\wt{\cal Q}_1:= \wt Y^{\top}\wt Y,
\quad \mathrm{ a}
\quad p\times p 
\quad \mathrm{matrix,\quad and} 
\quad \wt{\mathcal Q}_2:=\wt Y \wt Y^{\top}
\quad\mathrm{ an } 
\quad r\times r 
\quad \mathrm{matrix}.$$
We denote their (nontrivial) eigenvalues in descending order as $\wt \lambda_1 \geq \ldots \geq \wt \lambda_{p\wedge r}$. On the other hand, we will also use the non-spiked matrix 
\be\label{sketchnonspike} 
\wt X=  SX \Sigma^{1/2} .
\ee
 We denote the corresponding non-spiked matrices as 
$${\cal Q}_1:=\wt X^{\top}\wt X,
\quad \mathrm{ a}
\quad p\times p 
\quad \mathrm{matrix,\quad and} 
\quad {\mathcal Q}_2:=\wt X\wt X^{\top}
\quad\mathrm{ an } 
\quad r\times r 
\quad \mathrm{matrix},$$
 with eigenvalues $ \lambda_1 \geq \ldots \geq \lambda_{p\wedge r}$. 
If we study the centered sample covariance matrices, then we can still adopt the current setting by using  
$$ \wt Y= (I-ee^{\top}) S  X  \Sigma^{1/2} + (I-ee^{\top})\sum_{i=1}^k d_i v_i u_i^{\top}, \quad e:=r^{-1/2}(1,\cdots, 1)^{\top} \in \R^r.$$
or
$$\wt Y:= S (I-ee^{\top})  X  \Sigma^{1/2} + S\sum_{i=1}^k d_i  (I-ee^{\top})w_i u_i^{\top},\quad e:=n^{-1/2}(1,\cdots, 1)^{\top} \in \R^n.$$
For simplicity, we put aside this issue right now and assume that all the entries of $X$ are centered for most part of this paper. {In Appendix \ref{sec_centermodel}, we will show that introducing $(I-ee^{\top}) $ does not affect our results.}


We assume that the number of signals $k$ is a finite fixed integer, the strengths $d_1 > d_2 > \cdots > d_k >0$ are fixed constants, and $ u_i$,  $w_i$ are deterministic unit vectors. We shall consider the \emph{high-dimensional} setting in this paper. More precisely, we assume that the aspect ratios 
\be\label{eq_ratio}\gamma_n:= p /n \to \gamma,\quad \xi_n:= r / n \to \xi, \quad \text{as } \ n\to \infty \ee
for some constants $\gamma\in (0,\infty)$ and $\xi \in (0,1)$. 

We assume that the noise covariance $\Sigma$ and the outer product of the sketcing matrix $B:=SS^{\top}$ (an $r\times r$ matrix) have eigendecompositions
\be \label{eigen}
\Sigma= O_1 \Sig_1 O_1^{\top} , \quad B= O_2 \Sig_2 O_2^{\top}  ,\quad \Sig_1=\text{diag}(\si_1, \cdots, \si_p), \quad \Sig_2=\text{diag}( s_1, \cdots,  s_r),
\ee
where the eigenvalues of $\Sigma$ and $B$ are
$$\si_1 \ge \si_2 \ge \ldots \ge \si_p \ge 0, \quad s_1 \ge s_2 \ge \ldots \ge s_r \ge 0.$$
We denote the empirical spectral distributions (ESDs) of $\Sigma$ and $B=SS^{\top} $ by
\begin{equation}\label{sigma_ESD}
\pi_{\Sigma} := \frac{1}{p} \sum_{i=1}^p \delta_{\si_i} \ ,\quad \pi_{B}:= \frac{1}{r} \sum_{i=1}^r \delta_{s_i}\ .
\end{equation}
We assume that there exists a small constant $0<\tau<1$ such that for all $n$ large enough,
\begin{equation}\label{assm3}
\max\{\si_1, s_1\} \le \tau^{-1}, \quad \max\left\{\pi_{\Sigma}([0,\tau]), \pi_{B}([0,\tau])\right\} \le 1 - \tau .
\end{equation}
Both of these conditions are natural: the first condition means that the operator norms of $\Sigma$ and $B$ are bounded by $\tau^{-1}$, and the second condition means that the spectra of $\Sigma$ and $B$ do not concentrate at zero. Moreover, we assume that $\pi_{\Sigma}$ and $\pi_B$ converge to certain probability distributions as $n\to \infty$. We will also assume some regularity conditions on $\pi_{\Sigma}$ and $\pi_B$ later.

Finally we assume that the random variables $x_{ij}$ have finite (4+$\e$)-moments, in the following sense: there exists a constant $\tau>0$ such that 
\begin{equation}\label{eq_highmoment} 
\max_{i,j} \mathbb E |\sqrt{n}x_{ij}|^{4+\tau} \le \tau^{-1}.  
\end{equation}

\subsection{Resolvents and limiting laws}
\label{resll}
As usual in random matrix theory dating back to the seminal work of \cite{marchenko1967distribution}, we study the eigenvalue statistics of $\mathcal Q_{1,2}$ and $\wt {\mathcal Q}_{1,2}$ through their {\it{resolvents}} (or  {\it{Green's functions}}). 
Throughout the following, we shall denote the upper half complex plane and the right half real line by 
$$\mathbb C_+:=\{z\in \mathbb C: \im z>0\}, \quad \mathbb R_+:=[0,\infty).$$ 

\begin{definition}[Resolvents]\label{defn_resolvent}
For $z = E+ \ii \eta \in \mathbb C_+,$ we define the following resolvents as 
\begin{equation}\label{def_green}
\mathcal G_{1,2}(X,z):=\left({\mathcal Q}_{1,2}(X) -z\right)^{-1} , \ \ \ \wt{\mathcal G}_{1,2} (X,z):=\left(\wt{\mathcal Q}_{1,2}(X)-z\right)^{-1} .
\end{equation}
Note that the subscript $``1"$ indicates $p\times p$ matrices, while the subscript $``2"$ indicates $r\times r$ matrices. We denote the ESD $\rho^{(p)}$ of ${\mathcal Q}_{1}$ and its Stieltjes transform as 
\be\label{defn_m}
\rho\equiv \rho^{(p)} := \frac{1}{p} \sum_{i=1}^p \delta_{\lambda_i({\mathcal Q}_1)},\quad m(z)\equiv m^{(n)}(z):=\int \frac{1}{x-z}\rho^{(p)}(dx)=\frac{1}{p} \mathrm{Tr} \, \mathcal G_1(z).
\ee
We also introduce the following quantities, which can be viewed as weighted Stieltjes transforms of ${\mathcal Q}_1$ and ${\mathcal Q}_2$, respectively:
\begin{equation}\label{defn_m1m2}
m_1(z)\equiv m_1^{(n)}(z):= \frac{1}n\tr \left(\Sigma \mathcal G_1(z)\right) ,\quad m_2(z)\equiv m_2^{(n)}(z):=\frac{1}{n}\tr\left( B\mathcal G_2(z)\right). 
\end{equation}
\end{definition}


Notice that in equation \eqref{defn_m1m2S} below, different from \eqref{defn_m1m2}, we used the factor $p^{-1}$ instead of $n^{-1}$ in the definition of $m^S_1$. We adopted that notation there because $m_{1}^S(z)$ in that case actually plays the role of $m(z)$ in \eqref{defn_m}. 

We now describe the limiting laws of the density $\rho$ and its Stieltjes transform $m(z)$.
We define $(m_{1c}(z),m_{2c}(z))$ $\in \mathbb C_+^2$ as the unique solution to the following system of \emph{self-consistent equations}
\begin{equation}\label{separa_m12}
\begin{split}
& {m_{1c}(z)} = \frac1n\sum_{i = 1}^p \frac{\sigma_i}{-z(1+\sigma_i m_{2c})} = \gamma_n  \int\frac{x}{-z\left[1+xm_{2c}(z) \right]} \pi_{\Sigma}(dx) ,\\
& {m_{2c}(z)} = \frac{1}{n} \sum_{\mu= 1}^r \frac{s_\mu}{-z(1+s_\mu m_{1c})} = \xi_n \int\frac{x}{-z\left[1+xm_{1c}(z) \right]} \pi_B (dx) .
\end{split}
\end{equation}
It is known that this system admits a unique solution, see e.g., \cite{Zhang_thesis,el2009concentration}. Then we define the fundamental quantity $m_c$ in terms of the solution $m_{2c}$ found above:
\begin{equation}\label{def_mc}
m_c(z):= \frac{1}{p}\sum_{i = 1}^p \frac{1}{-z(1+\sigma_i m_{2c})}   =\int\frac{1}{-z\left[1+xm_{2c}(z) \right]} \pi_{\Sigma}(dx).
\end{equation}
It is easy to verify that $m_c(z)\in \mathbb C_+$ for $z\in \mathbb C_+$. It turns out that this is \emph{the Stieltjes transform of the limiting spectral distribution} of the non-spiked sketched matrix $\cal Q_1= \wt X^{\top} \wt X$, where recall that $\wt X$ is defined in \eqref{sketchnonspike}. We can recover the distribution of eigenvalues in the usual way, by inverting the Stieltjes transform. Letting $\eta \downarrow 0$, we obtain the probability measure $\rho_{c}$ which describes the limiting distribution of the eigenvalues with the inverse formula
\begin{equation}\label{ST_inverse}
\rho_{c}(E) = \lim_{\eta\downarrow 0} \frac{1}{\pi}\Im\, m_{c}(E+\ii \eta).
\end{equation}
Moreover, under the assumption \eqref{assm3}, the supremum of the support of $\rho_c(E)$ is at a finite value $\lambda_+$, known as ``the right edge", which is also the ``classical location" of the largest eigenvalue of $\cal Q_1$. This means that 
the largest eigenvalue will actually converge almost surely to $\lambda_+$.

The known results on the existence and uniqueness of the solution $(m_{1c}(z),m_{2c}(z))$ to \eqref{separa_m12}, the continuous density of $\rho_c$ and the rightmost edge $\lambda_+$ are collected into the following lemma.

\begin{lemma} [Existence, uniqueness, and continuous density]\label{lambdar}
For any $z\in \mathbb C_+$, there exists a unique solution $(m_{1c},m_{2c})\in \mathbb C_+^2$ to the systems of equations in (\ref{separa_m12}), such that both functions $m_{1c},m_{2c}$ are Stieltjes transforms of two measures (not necessarily probability measures) $\mu_{1c}$ and $\mu_{2c}$ supported on $\mathbb R_+$. The function $m_c$ in (\ref{def_mc}) is the Stieltjes transform of a probability measure $\mu_c$ supported on $\mathbb R_+$. Moreover, $\mu_c$ (resp. $\mu_{1,2c}$) has a continuous density $\rho_c(x)$ (resp. $\rho_{1,2c}(x)$) on $(0,\infty)$, which is defined by \eqref{ST_inverse}. The densities $\rho_{c}$ and $\rho_{1,2c}$ all have the same support on $(0,\infty)$, which is a union of intervals:  
\begin{equation}\label{support_rho1c}
{\rm{supp}} \, \rho_{1,2c} \cap (0,\infty) ={\rm{supp}} \, \rho_{c} \cap (0,\infty) = \bigcup_{k=1}^\fa [e_{2k}, e_{2k-1}] \cap (0,\infty),
\end{equation}
where the number of components $\fa\in \mathbb N$ depends only on $\pi_{\Sigma}$ and $\pi_{B}$. Here we order the components so that $e_{2k}<e_{2k-1}<e_{2k-2}$, hence $e_1$ is the supremum of the support--or ``right edge"--of the $\rho$-s.  Under the first assumption in \eqref{assm3} (i.e., when the top eigenvalues of the population spectra are bounded), we have that the right edge of the empirical spectrum is also bounded, i.e., $e_1=\OO(1)$. Finally we know that evaluating the Stieltjes transforms $m_{1c},m_{2c}$ at this right edge places us in a certain special range, so that $m_{1c}(e_1) \in (- s_1^{-1}, 0)$ and $m_{2c}(e_1) \in (-\sigma_1^{-1}, 0)$. 
\end{lemma}
\begin{proof}
The proof of this lemma is contained in {\cite[Theorem 1.2.1]{Zhang_thesis}}, {\cite[Theorem 2.4]{Hachem2007deterministic}} and {\cite[Section 3]{Separable_solution}}.
\end{proof}

 We shall call $e_k$ the \emph{spectral edges}. In particular, we will only focus on the rightmost edge $\lambda_+ := e_1$. 
Now we make the following assumption. It guarantees a regular square-root behavior of the spectral densities $\rho_{c}$ near $\lambda_+$ and rules out the existence of spikes for $\cal Q_{1,2}$. In other words, the spikes of $\wt{\cal Q}_{1,2}$ are only caused by the signals in \eqref{sketchspike}. We note that this is a mild condition, and holds in particular when the ESDs of $\Sigma$ and $B$ are well behaved. Specifically, when $B=I_n$ (i.e., when there is no projection), then it is known that the square root behavior holds as long as the limit of the ESD of $\Sigma$ is sufficiently ``regular" at its right edge. For instance, it is enough if the right edge of the limiting ESD is a point mass, or the end of a uniform distribution, see e.g., \cite{silverstein1995analysis,Bai2006}. This is a mild condition that, while hard to check in applications, does not appear to be a significant limitation. 


\begin{assumption}[Right edge regularity] \label{ass:unper} 
There exists a constant $\tau>0$ such that 
\begin{equation}\label{assm_gap}
1+m_{1c}(\lambda_+) s_1 \geq \tau, \quad 1+m_{2c}(\lambda_+) \sigma_1 \geq \tau.
\end{equation}
\end{assumption}
 Under this assumption, we have the following lemma.
 
\begin{lemma}[Lemma 2.6 of \cite{yang2019spiked}, square root density at edge]\label{lambdar_sqrt}
Under assumptions \eqref{eq_ratio}, \eqref{assm3} and \eqref{assm_gap}, there exists a constant $a >0$ such that
\be\label{sqroot3}
\rho_{c}(\lambda_+ - x) = a  x^{1/2} + \OO(x), \quad \mathrm{as} \quad x\downarrow 0,
\ee
and
\be\label{sqroot4}
\quad m_{c}(z) = m_{c}(\lambda_+) + \pi  a (z-\lambda_+)^{1/2} + \OO(|z-\lambda_+|), \quad z\to \lambda_+ , \ \ \im z\ge 0.
\ee
The estimate \eqref{sqroot4} also holds for $m_{1,2c}$ with possibly different constants $a_{1,2}>0$. 
\end{lemma}
 

We introduce a convenient and classical self-adjoint linearization trick. This idea dates back at least to Girko, see e.g., the works \cite{girko1975random,girko1985spectral,girko2012theory} and references therein. Define the \emph{linearization} matrix as the following $(p+r)\times (p+r)$ self-adjoint block matrix, which is a linear function of $X$:
 \begin{equation}\label{linearize_block}
   H \equiv H(X,z): = z^{1/2} \left( {\begin{array}{*{20}c}
   { 0 } & {\wt X^{\top}}   \\
   {\wt X} & {0}  \\
   \end{array}} \right),  \quad z\in \mathbb C_+ .
 \end{equation}
where recall that $\wt X=S X \Sigma^{1/2}$ is the projected non-spiked matrix, and $z^{1/2}$ is taken to be the branch cut with positive imaginary part. Then we define its resolvent (Green's function) as
 \begin{equation}\label{eq_gz} 
 G \equiv G (X,z):= \left(H(X,z)-z\right)^{-1} .
 \end{equation}
By the Schur complement formula, we can verify that (recall that by \eqref{def_green}, $\mathcal G_1$ is the resolvent of $\wt X^{\top}\wt X$, and $\mathcal G_2$ is the resolvent of $\wt X\wt X^{\top}$)
\begin{align} 
G(z) = \left( {\begin{array}{*{20}c}
   { \mathcal G_1} & z^{-1/2}\mathcal G_1 \wt X^{\top} \\
   {z^{-1/2}\wt X \mathcal G_1} & { \mathcal G_2 }  \\
\end{array}} \right) = \left( {\begin{array}{*{20}c}
   { \mathcal G_1} & z^{-1/2}\wt X^{\top} \mathcal G_2   \\
   {z^{-1/2} \mathcal G_2 \wt X} & { \mathcal G_2 }  \\ 
\end{array}} \right). \label{green2} 
\end{align}
Thus a control of $G$ yields directly a control of the resolvents $\mathcal G_{1,2}$. Similarly, we can define $\wt H$ and $\wt G$  by replacing $\wt X$ with the spiked version $\wt Y$. 
For simplicity of notation, we will sometimes use the index sets
\[\mathcal I_1:=\{1,...,p\}, \quad \mathcal I_2:=\{p+1,...,p+r\}, \quad \mathcal I:=\mathcal I_1\cup\mathcal I_2, \quad \cal I_2^{n}:=\{p+1, \cdots, p+n\}.\]
Then we shall label the indices of the matrices in the natural way. For instance, since $X$ is an $n\times p$ matrix, we will label its row indices according to $\cal I_2^{n}$, and its column indices according to $\mathcal I_1$:
$$X= (X_{\mu i})_{\mu \in \cal I_2^n,i\in \cal I_1}, \quad \Sigma=(\Sigma_{ij})_{i,j \in \cal I_1},\quad S=(S_{\mu\nu})_{\mu \in \cal I_2, \nu \in \cal I_2^n}.$$  
In the rest of this paper, 
we will consistently use the latin letters $i,j\in\mathcal I_1$ and greek letters $\mu,\nu\in\mathcal I_2$ or $\mathcal I_2^{n}$. 


We define the following matrix, which turns out to be the deterministic limit of the resolvent $G$ of the linearization matrix $H$, as 
\begin{equation}\label{defn_pi}
\Pi (z):=\left( {\begin{array}{*{20}c}
   { \Pi_1} & 0  \\
   0 & { \Pi_2}  \\
\end{array}} \right), \quad \Pi_1:  =-z^{-1}\left(1+m_{2c}(z)\Sigma\right)^{-1},\quad \Pi_2:=- z^{-1} (1+m_{1c}(z)B )^{-1}.
\end{equation}
Note that from (\ref{separa_m12}) we can express the Stieltjes transforms $m_c$ and $m_{1,2c}$ (which determine the limiting spectral distribution), as the following weighted traces of the functionals of  $\Pi$:
\be\label{mcPi}
\frac1{n}\tr \Pi_{1} =m_c, \quad  \frac1{n}\tr \left(\Sigma \Pi_{1}\right) =m_{1c}, \quad \frac1{n}\tr \left(B\Pi_{2}\right) =m_{2c}.
\ee
In \cite{yang2019spiked,ding2019spiked}, an anisotropic local law away from the support of $\rho_{c}$ was proved in the form of Theorem \ref{lem_localout} below. 
Roughly speaking, the local law means that the random resolvent matrix $G$ is well approximated by the deterministic matrix $\Pi$ defined above. This holds in the sense that linear combinations of entries of $G$ can be approximated by the same linear combinations of entries of $\Pi$. This has been more formal in work on deterministic equivalents, see e.g., \cite{Hachem2007deterministic,dobriban2018distributed}.

Before stating the local law, for convenience, we introduce the following notion of stochastic domination, which was introduced in \cite{Average_fluc} and subsequently used in many works on random matrix theory, such as \cite{isotropic,principal,local_circular,Delocal,Semicircle,Anisotropic}. It simplifies the presentation of the results and their proofs by systematizing statements of the form ``$\xi$ is bounded by $\zeta$ with high probability up to a small power of $n$".

\begin{definition}[Stochastic domination]\label{stoch_domination}
(i) Let
\[\xi=\left(\xi^{(n)}(u):n\in\bbN, u\in U^{(n)}\right),\hskip 10pt \zeta=\left(\zeta^{(n)}(u):n\in\bbN, u\in U^{(n)}\right)\]
be two families of nonnegative random variables, where $U^{(n)}$ is a possibly $n$-dependent parameter set. We say $\xi$ is stochastically dominated by $\zeta$, uniformly in $u$, if for any fixed (small) $\epsilon>0$ and (large) $D>0$, 
\[\sup_{u\in U^{(n)}}\bbP\left(\xi^{(n)}(u)>n^\epsilon\zeta^{(n)}(u)\right)\le n^{-D}\]
for large enough $n \ge n_0(\epsilon, D)$, and we shall use the notation $\xi\prec\zeta$. 
If for some complex family $\xi$ we have $|\xi|\prec\zeta$, then we will also write $\xi \prec \zeta$ or $\xi=\OO_\prec(\zeta)$.

(ii) We extend the definition of $\OO_\prec(\cdot)$ to matrices in the operator norm sense as follows. Let $A$ be a family of random matrices and $\zeta$ be a family of nonnegative random variables. Then $A=\OO_\prec(\zeta)$ means that $\|A\|\prec \zeta$.

(iii) We say an event $\Xi$ holds with high probability if for any constant $D>0$, $\mathbb P(\Xi^c)\le n^{-D}$ for large enough $n$. We say an event $\Xi$ holds with high probability on an event $\Omega$ if for any constant $D>0$, $\mathbb P(\Omega\setminus \Xi)\le  n^{-D}$ for large enough $n$.
\end{definition}

The following lemma collects basic properties of stochastic domination $\prec$, which will be used repeatedly in the proof.

\begin{lemma}[Lemma 3.2 in \cite{isotropic}, Closure properties of stochastic domination]\label{lem_stodomin}
Let $\xi$ and $\zeta$ be families of nonnegative random variables.
\begin{itemize}

\item[(i)] {\bf Sums.} Suppose that $\xi (u,v)\prec \zeta(u,v)$ uniformly in $u\in U$ and $v\in V$. If $|V|\le n^C$ for some constant $C$, then $\sum_{v\in V} \xi(u,v) \prec \sum_{v\in V} \zeta(u,v)$ uniformly in $u$.

\item[(ii)] {\bf Products.} If $\xi_1 (u)\prec \zeta_1(u)$ and $\xi_2 (u)\prec \zeta_2(u)$ uniformly in $u\in U$, then $\xi_1(u)\xi_2(u) \prec \zeta_1(u)\zeta_2(u)$ uniformly in $u$.

\item[(iii)] {\bf Taking expectations.} Suppose that $\Psi(u)\ge n^{-C}$ is deterministic and $\xi(u)$ satisfies $\mathbb E\xi(u)^2 \le n^C$ for all $u$. Then if $\xi(u)\prec \Psi(u)$ uniformly in $u$, we have $\mathbb E\xi(u) \prec \Psi(u)$ uniformly in $u$.
\end{itemize}
\end{lemma}

In this paper, given (possibly complex) vectors $u,v$ and a matrix $A$ of conformable dimensions, we denote the inner product by
$$\langle u, A v\rangle :=u^{\top} Av,$$
where $u^{\top}$ is the complex conjugate of $u$. For simplicity, we shall also write $\langle u, A v\rangle $ as a generalized entry  $A_{uv}\equiv \langle u, A v\rangle$.

Now we are ready to state the anisotropic local law for $G$, which will be the main tool of this paper. It essentially follows from Theorem 4.10 of \cite{ding2019spiked}. However, our setting is a little different from the setting there, so we will give the necessary details in Appendix \ref{sec_pflocallaw} to adapt the proof in \cite{ding2019spiked} to our setting. As mentioned, this shows that linear combinations of entries of $G$ can be approximated by the same linear combinations of entries of $\Pi$. 

\begin{theorem}[Anisotropic local law outside of the spectrum]\label{lem_localout}  
Suppose the setting in Section \ref{sec_defmodel} and Assumption \ref{ass:unper} hold. Let $\mathscr A$ be any set of (complex) deterministic unit vectors of cardinality $|\mathscr A| \le n^C$ for some constant $C>0$. Fix any small constant $c_0>0$ and large constant $C_0>0$, define the spectral parameter domain
\begin{equation}\label{eq_paraout}
z\in S_{out}(c_0,C_0):=\left\{ E+ \ii\eta: \lambda_+ + c_0 \le E \le C_0,  \eta\in [0,C_0]\right\}.
\end{equation}
There there exists a set $\Omega$ with $\mathbb P(\Omega) \ge 1-n^{-\delta}$ for some constant $0<\delta \le 1/2$ depending on $\tau$ in \eqref{eq_highmoment} only, such that the following anisotropic local law holds: 
\begin{equation}\label{aniso_outstrong}
1(\Omega) \max_{u,v\in \mathscr A} \left| \langle u, G(X,z) v\rangle - \langle u, \Pi (z)v\rangle \right|  \prec n^{-\delta} 
\end{equation}
uniformly in $z\in S_{out}(c_0,C_0)$.
\end{theorem}

We remark that Theorem 2.4 of \cite{isotropic} is actually a special case of our Theorem \ref{lem_localout} by replacing $\wt X=S$ (i.e. we replace $X\to S$, $\Sigma\to I_p$ and $S\to I_n$), but on a bigger domain of $z$. In fact, our Theorem \ref{lem_localout} can be also generalized to such a bigger domain of $z$ by Theorem 3.6 of \cite{yang2019spiked}, and the reader can check that \eqref{localmS0} below holds due to the claim in \eqref{aniso_outstrong}.

Moreover, we have the following local law for $z$ near the edge $\lambda_+$ of the spectrum, which will be used to study the non-spiked eigenvalues and eigenvectors. It is a corollary of Theorem 3.6 of \cite{yang2019spiked}, and we shall give the proof in Appendix \ref{sec_pflocallaw}. The only difference between the local law outside the spectrum and the one near the edge is that the argument $z= E+ \ii\eta$ of the resolvent $G(X,z)$ is restricted to have real part $E$ strictly larger than the right edge $\lambda_+$ for the law outside the spectrum, and there are no restrictions on the imaginary part $\eta$. For the law near the edge, $z$ is restricted to have real part $E$ around the right edge $\lambda_+$, but the imaginary part $\eta$ must have absolute value at least of the order of $n^{-1/2+c_1}$ for some $c_1>0$.

\begin{theorem}[Anisotropic local law near the edge]\label{lem_localin}  
Suppose the assumptions of Theorem \ref{lem_localout} hold. Fix any small constants $c_0,c_1>0$ and large constant $C_0>0$, and define the spectral parameter domain
\begin{equation}\label{eq_parain}
z\in S_{edge}(c_0,C_0,c_1):=\left\{ E+ \ii\eta: \lambda_+ - c_0 \le E \le C_0,  \eta\in [n^{-1/2+c_1},C_0]\right\}.
\end{equation}
There there exists a set $\Omega$ with $\mathbb P(\Omega) \ge 1-n^{-\delta}$ for some constant $0<\delta \le 1/2$ depending on $\tau$ in \eqref{eq_highmoment} only, such that the following anisotropic local law holds: 
\begin{equation}\label{aniso_in}
1(\Omega) \max_{u,v\in \mathscr A} \left| \langle u, G(X,z) v\rangle - \langle u, \Pi (z)v\rangle \right|  \prec n^{-\delta} 
\end{equation}
uniformly in $z\in S_{edge}(c_0,C_0,c_1)$. Moreover, fixing any $\varpi \in \N$, we have that
\be\label{rigiditye}
1(\Omega)\max_{1\le i \le \varpi}|\lambda_i - \lambda_+| \prec n^{-\delta}.
\ee
\end{theorem}


We mention that such local laws are part of a much broader line of work in random matrix theory, going back to the Marchenko-Pastur law \cite{marchenko1967distribution}. See e.g., \cite{erdos2012universality,erdos2017dynamical} for more recent results on related topics such as universality. The topic of deterministic equivalents is also related, see e.g., \cite{Hachem2007deterministic,dobriban2018distributed}. 

\subsection{The spiked eigenvalues and eigenvectors}\label{sec method}

 With the anisotropic local law, we can derive a so-called master equation for the outlier eigenvalues and eigenvectors. We write the sketched signal matrix as
$$\sum_{i=1}^k d_i  v_iu_i^{\top}=VDU^{\top}, \quad D=\text{diag}(d_1,\cdots, d_k),  $$
where $U$, $V$ and $W$ are $p\times k$, $r\times k$ and $n\times k$ matrices:
$$ U=(u_1,\cdots, u_k), \quad V= (v_1,\cdots, v_k)=SW, \quad W= (w_1,\cdots, w_k).$$
Then we define the \emph{linearization of the sketched signal} as the following $(p+r)\times (p+r)$ block matrix:
$$ \Delta H:= z^{1/2}\left( {\begin{array}{*{20}c}
   { 0 } & {UDV^{\top}}   \\
   {VDU^{\top}} & {0}  \\
   \end{array}} \right) =z^{1/2} A\cal DA^{\top}, \quad A:=  \left( {\begin{array}{*{20}c}
   { U } & {0}   \\
   {0} & {V}  \\
   \end{array}} \right) , \quad \cal D:= \left( {\begin{array}{*{20}c}
   { 0 } & {D}   \\
   {D} & {0}  \\
   \end{array}} \right) .$$

\begin{lemma}\label{lem_pertubation} 
If $x>\lambda_+$ is not an eigenvalue of $\mathcal Q_1=\wt X^{\top}\wt X$, then it is an eigenvalue of $\widetilde{Q}_1 = \wt  Y^{\top} \wt  Y$ if and only if the following determinant (of a $2k\times 2k$ matrix) vanishes:
\begin{equation}\label{masterx}
\det\left(\mathcal{D}^{-1}+x^{1/2} A^{\top} G(x) A\right)=0.
\end{equation} 
\end{lemma}
\begin{proof} 
The proof is similar to the one for Lemma 5.1 of \cite{ding2019spiked}. But our setting is a little different from the one there, so we give a full proof. Note that the non-zero eigenvalues of $z^{-1/2}\wt H$ are given by 
$$\pm \sqrt{\lambda_1(\wt Q_1)}, \ \pm \sqrt{\lambda_2(\ctQ_1)}, \ \cdots \ , \ \pm \sqrt{\lambda_{p\wedge r}(\ctQ_1)}.$$
Hence it is easy to see that $x>0$ is an eigenvalue of $\ctQ_1$ if and only if
\be\label{detH1}
\det\left(\wt H(X,x) - x\right)=0,
\ee
from which we obtain that 
\begin{align*}
0&=\det(H+\Delta H -x)=\det(H-x) \det\Big(1 + G(x) \Delta H \Big)=\det(H-x) \det(1+ x^{1/2}A^{\top} G(x) A \cal D)\\
&=\det(\cal D)\det(H-x) \det(\cal D^{-1}+ x^{1/2}A^{\top} G(x) A ),
\end{align*}
where in the third step we used identity $\det(1+\cal C\cal B)=\det(1+\cal B\cal C)$ for any two matrices $\cal B$ and $\cal C$ of conformable dimensions. The claim then follows if $x$ is not an eigenvalue of $\cal Q_1$, i.e. $\det(H-x)\ne 0$.
\end{proof}

Using Theorem \ref{lem_localout}, up to some small error of order $\OO_\prec(n^{-\delta})$, equation \eqref{masterx} gives approximately the following \emph{eigenvalue master equation} that has to hold for any possible spike $x$. 
\be\label{master_evalue}
\det M(x) \equiv 
\det \left( {\begin{array}{*{20}c}
   {  -x^{-1/2}U^{\top} \left(1+m_{2c}(x)\Sigma\right)^{-1}U } & {D^{-1}}   \\
   {D^{-1}} & {- x^{-1/2} V^{\top} (1+m_{1c}(x)B )^{-1} V}  \\
   \end{array}} \right) = 0.
\ee
To be clear, we have the following steps needed to find the spikes:
\begin{enumerate}
\item We are given the population covariance ${\Sigma}$ ($p\times p$).
\item We have the sketching matrix $S$ ($r \times n$).
\item We are given the left and right matrices of eigenvectors $V$ and $U$ ($r\times k$ and $p\times k$).
\item We have the $k\times k$ diagonal matrix $D$ of population spikes.
\item For any given $x$, we calculate the pair $(m_{1c}(x),m_{2c}(x))$, arising as the solution to the self-consistent equations \eqref{separa_m12}. This is determined entirely by the eigenvalues of ${\Sigma}$ and $S$.
\item We combine the above quantities into the $2k\times 2k$ master matrix $M(x)$ given above. 
\item We solve for the values $x$ for which this matrix is singular, i.e., solve equation \eqref{master_evalue}. In general we expect at most $k$ such values. These are all possible candidates for the empirical spikes of the sketched data.
\end{enumerate}

To get some explicit results, we consider some special cases as in Section \ref{sec5types}. We try to solve this master equation in a case by case manner. 

\vspace{5pt}

Next we discuss the sample eigenvectors for the outliers. For now, suppose we know that the $i$-th largest outlier $\wt\lambda_i$ lies around a classical location $\theta_i$, which is some fixed definitive value. Moreover, assume that these values are well-separated from each other (i.e. there exists a constant $\e>0$ such that $|\theta_i-\theta_j| \ge \e$ for any $i\ne j$). We want to study the overlap between the sample eigenvector and the population eigenvector $u_i$. 
Let
$$\wt Y  = \sum_{k = 1}^{p \wedge r} \sqrt {\wt\lambda_k} \wt{\zeta} _{k} \wt{ \xi}_k^{\top} ,$$
be a singular value decomposition of the sketched spiked matrix, where
$$\wt\lambda_1\ge \wt\lambda_2 \ge \ldots \ge \wt\lambda_{p\wedge r} \ge 0 = \wt\lambda_{p\wedge r+1} = \ldots = \wt\lambda_{p\vee r}$$
are the eigenvalues of $\ctQ_1 = \wt  Y^{\top} \wt  Y,$ while
$\{\wt{ \zeta}_{k}\}_{k=1}^{r}$
and
$\{\wt{ \xi}_{k}\}_{k=1}^{p}$ are the left and right singular vectors of $\wt Y,$ respectively. 
Then using (\ref{green2}) for $\wt G$, we can get that for $i,j\in \mathcal I_1$ and $\mu,\nu\in \mathcal I_2$,
\begin{align}
\wt G_{ij} = \sum_{k = 1}^{p} \frac{\wt{ \xi}_k(i) \wt{ \xi}_k^{\top}(j)}{\wt\lambda_k-z},\ \quad \ &\wt G_{\mu\nu} = \sum_{k = 1}^{r} \frac{\wt{ \zeta}_k(\mu) \wt{ \zeta}_k^{\top}(\nu)}{\wt\lambda_k-z}, \label{spectral1}\\
\wt G_{i\mu} = z^{-1/2}\sum_{k = 1}^{p\wedge r} \frac{\sqrt{\wt\lambda_k}\wt{ \xi}_k(i) \wt{ \zeta}_k^{\top}(\mu)}{\wt\lambda_k-z}, \ \quad \ &\wt G_{\mu i} =  z^{-1/2}\sum_{k = 1}^{p\wedge r} \frac{\sqrt{\wt\lambda_k} \wt{ \zeta}_k(\mu) \wt{ \xi}_k^{\top}(i)}{\wt\lambda_k-z}.\label{spectral2}
\end{align}

We also recall the following well known lemma for matrix perturbation, which follows from a simple algebraic calculation.

\begin{lemma} [Woodbury matrix identity] \label{lem_woodbury} For $\mathcal{A},S,\mathcal{B},T$ of conformable dimensions, we have 
\begin{equation}\label{Woodbury}
(\mathcal A+S\mathcal BT)^{-1}=\mathcal A^{-1}-\mathcal A^{-1}S(\mathcal B^{-1}+T\mathcal A^{-1}S)^{-1}T\mathcal A^{-1}.
\end{equation}
as long as all the operations are well defined. As a special case, we have the following equation, sometimes known as Hua's identity:
\begin{equation}\label{Huaineq}
\mathcal A-\mathcal A(\mathcal A+\mathcal B)^{-1}\mathcal A=\mathcal B-\mathcal B(\mathcal A+\mathcal B)^{-1}\mathcal B
\end{equation}
if $\mathcal A+\mathcal B$ is non-singular.
\end{lemma}

With \eqref{Woodbury}, we can write that
\begin{align*}
A^{\top} \wt G(z) A = A^{\top} \frac{1}{H - z + z^{1/2} A\cal DA^{\top}} A =A^{\top} \left(G(z) - G(z) A \frac{1}{z^{-1/2}\cal D^{-1}  + A^{\top}G(z)A} A^{\top}G(z)\right) A.
\end{align*}
Our goal is to study $|\langle u_j, \wt{\bxi}_i\rangle|^2 $ for some spiked eigenvector $\wt{\bxi}_i$. We consider a small contour $\Gamma_i$ around $\theta_i$, which only encloses $\wt\lambda_i$ but no other eigenvalues. Then using Cauchy's Theorem, we obtain the following \emph{angle master equation}:
\begin{align} \label{ev_origin}
|\langle u_j, \wt{\bxi}_i\rangle|^2 = \frac{-1}{2\pi \ii} \oint_{\Gamma_i}e_j^{\top} A^{\top} \wt G(z) A e_j\dd z = \frac{1}{2\pi \ii (\wt\lambda_i)^{1/2}} \left( \oint_{\Gamma_i}e_j^{\top}\cal D^{-1}\frac{1}{\cal D^{-1}  + z^{1/2}A^{\top}G(z)A} \cal D^{-1}e_j  \dd z\right) .
\end{align}
This gives an expression for the inner product of the true and empirical spike eigenvectors. To evaluate it in specific cases, again we need to study the master matrix $M(z)^{-1}=\left(\cal D^{-1}  + z^{1/2}A^{\top}G(z)A\right)^{-1}$.

\section{Proof of Theorem \ref{sketchthm1}}\label{pf sketch1}



In this section, we prove Theorem \ref{sketchthm1} based on the master equations \eqref{master_evalue} and \eqref{ev_origin}, and the local laws, Theorems \ref{lem_localout} and \ref{lem_localin}. We shall give all the details for the proof in this section, which can be applied to Theorems \ref{sketchthm4}-\ref{sketchthmlarge} directly. In fact, the only differences will be the analysis of the master equations, which we will perform in a case by case manner; all the other parts of the proof are essentially the same. 

We first introduce some preliminary estimates. 
For $ z=E+\mathrm{i} \eta,$ we define the distance to the rightmost edge as
 \begin{equation}\label{KAPPA}
 \kappa \equiv \kappa_E:=|E-\lambda_+|.
 \end{equation}
Then we summarize some basic properties of $m_{1,2c}$. We define the domain
\begin{equation}\label{Sc0C0}
\wt S(c_0,C_0):= \left\{z=E+ \ii \eta: \lambda_+ - c_0 \leq E \leq C_0 , 0 \leq \eta \leq C_0 \right\}.
\end{equation}

\begin{lemma}[Lemma 3.4 of \cite{yang2019spiked}]\label{lem_mplaw} 
Suppose \eqref{eq_ratio}, \eqref{assm3}, and Assumption \ref{ass:unper} hold. Fix any constant $C_0>0$. Then there exists sufficiently small constant $c_0>0$ such that the following estimates hold for $z=E+\ii\eta\in \wt S(c_0,C_0)$: 
\begin{itemize}
\item[(i)] for $z =E+\ii \eta\in \wt S(c_0,C_0)$, 
\begin{equation} \label{eq_estimm} 
\vert m_{1,2c}(z) \vert \sim 1,  \quad  \im m_{1,2c}(z) \sim 
\begin{cases}
    \frac{\eta}{\sqrt{\kappa+\eta}}, & \text{ if } E\geq \lambda_+ \\
    \sqrt{\kappa+\eta}, & \text{ if } E \le \lambda_+\\
  \end{cases};
\end{equation}

\item[(ii)] there exists constant $\tau'>0$ such that
\begin{equation}\label{Piii}
\min_{ \mu } \vert 1 + m_{1c}(z)s_\mu \vert \ge \tau', \quad \min_{i } \vert 1 + m_{2c}(z)\sigma_i  \vert \ge \tau',
\end{equation}
for any $z \in \wt S (c_0,C_0)$. In fact, \eqref{Piii} holds if we replace $s_\mu$ (resp. $\sigma_i$) with any positive value that is smaller than $s_1$ (resp. $\sigma_1$).
\end{itemize}
\end{lemma}

The functions $m_{1c}(z)$ and $m_{2c}(z)$ are holomorphic on the right half complex plane $\{z : \re z>\lambda_+\}$. Moreover, they are one-to-one in the region near real axis, so that we can define their inverse functions as $g_{1c}$ and $g_{2c}$. The following lemma gives some basic estimates on $m_{1,2c}$, $g_{1,2c}$ and their derivatives.


\begin{lemma}[Lemma 4.5 of \cite{ding2019spiked}] \label{lem_complexderivative}
Suppose the assumptions of Lemma \ref{lem_mplaw} hold. Then for any constant $\varsigma>0$, there exist constants $\tau_0, \tau_1, \tau_2>0$ such that the following statements hold.
\begin{itemize}
\item[(i)] $m_{1c}$ and $m_{2c}$ are holomorphic homeomorphisms on the spectral domain
$$D(\tau_0,\varsigma):=\{z=E+\mathrm{i} \eta: \lambda_+< E < \varsigma, \ -\tau_0< \eta < \tau_0\}.$$
 As a consequence, the inverse functions of $m_{1c}$ and $m_{2c}$ exist and we denote them by $g_{1c}$ and $g_{2c}$, respectively.

\item[(ii)] We have $D_1(\tau_1, \varsigma)\subset m_{1c} (D(\tau_0,\varsigma))$ and $D_2(\tau_2, \varsigma)\subset m_{2c} (D(\tau_0,\varsigma))$, where 
$$D_1(\tau_1,\varsigma):=\{\xi=E+\mathrm{i} \eta: m_{1c}(\lambda_+)< E < m_{1c}(\varsigma), \ -\tau_1< \eta < \tau_1\},$$
and
$$D_2(\tau_2,\varsigma):=\{\zeta=E+\mathrm{i} \eta: m_{2c}(\lambda_+)< E < m_{2c}(\varsigma), \ -\tau_2< \eta < \tau_2\}.$$
In other words, $g_{1c}$ and $g_{2c}$ are holomorphic homeomorphisms on $D_1(\tau_1,\varsigma)$ and $D_2(\tau_2,\varsigma)$, respectively.

\item[(iii)] For $z\in D(\tau_0,\varsigma)$, we have
\begin{equation}\label{eq_mcomplex}
|m_{1c}(z) - m_{1c}(\lambda_+)| \sim |z-\lambda_+|^{1/2}, \quad  |m_{2c}(z) - m_{2c}(\lambda_+)|   \sim |z-\lambda_+|^{1/2},
\end{equation}
and
\begin{equation}\label{eq_mcomplexd}
|m_{1c}'(z) | \sim |z-\lambda_+|^{-1/2}, \quad  |m_{2c}'(z)|   \sim |z-\lambda_+|^{-1/2}.
\end{equation}


\item[(iv)] For $z_1, z_2\in D(\tau_0,\varsigma)$, we have 
\begin{equation}\label{eq_mdiff}
|m_{1c}(z_1) - m_{1c}(z_2)| \sim |m_{2c}(z_1) - m_{2c}(z_2)|  \sim \frac{|z_1-z_2|}{\max_{i=1,2}|z_i-\lambda_+|^{1/2}}.
\end{equation}
\end{itemize}
\end{lemma}

The following eigenvalue interlacing result follows directly from the Cauchy interlacing theorem. 

\begin{lemma} [Eigenvalue interlacing]  \label{lem_weylmodi} 
Recall that the eigenvalues of $\ctQ_1$ and $\mathcal{Q}_1$ are denoted by $\{\wt\lambda_i\}$ and $\{\lambda_i\}$, respectively. Then we have 
\begin{equation}\label{interlacing_eq0}
\wt\lambda_i \in [\lambda_{i +k}, \lambda_{i-k}],
\end{equation}
where we adopt the convention that  $\lambda_{i}=\infty$ if $i<1$ and $\lambda_i = 0$ if $i>p\wedge r$. 
\end{lemma}

With the above preparations, we are ready to prove Theorem \ref{sketchthm1}. We first prove the near-orthogonality of columns of partial orthogonal matrices, that is, estimate \eqref{claim unifV}. 

%
\begin{proof}[Proof of \eqref{claim unifV}]
 Let us represent $\wh S$ as the upper $r\times n$ sub-matrix of some $n\times n$ Haar distributed matrix $T$. Then we have
\begin{align*}
 \mathbb E (V^{\top} V)_{11} &= \mathbb E\sum_{j=1}^r \wh S_{j1}^2 = \frac{1}{n}\mathbb E\sum_{j=1}^r \sum_{k=1}^n\wh S_{jk}^2 = \frac{r}{n}, \\
 \text{Var} \left[(V^{\top} V)_{11}\right] & = \mathbb E\sum_{j,j'=1}^r \wh S_{j1}^2\wh S_{j'1}^2 - \frac{r^2}{n^2}=  \mathbb E \sum_{j=1}^r \wh S_{j1}^4 + \sum_{j\ne j' \in \llbracket 1,r\rrbracket} \mathbb E \wh S_{j1}^2\wh S_{j'1}^2  - \frac{r^2}{n^2} \\
&=  \mathbb E \sum_{j=1}^r T_{j1}^4 + \frac{r(r-1)}{n(n-1)}\mathbb E\sum_{j\ne j' \in \llbracket 1,n\rrbracket} T_{j1}^2 T_{j'1}^2  - \frac{r^2}{n^2} \\
&=  \mathbb E \sum_{j=1}^r T_{j1}^4 - \frac{r(r-1)}{n(n-1)}\mathbb E\sum_{j=1}^n T_{j1}^4 + \frac{r(r-1)}{n(n-1)}\mathbb E\sum_{j, j' =1}^n T_{j1}^2 T_{j'1}^2  - \frac{r^2}{n^2} =\OO(n^{-1}),
\end{align*}
where we used that $\mathbb ET_{j1}^4 = \OO(n^{-2})$, since the random vector $t_1:=(T_{j1})$ has the same distribution as a normalized Gaussian vector:
$$t_1 \stackrel{d}{=} {g}/{\|g\|}.$$
Here $g\in \R^n$ has i.i.d. standard normal entries. Similarly, we can calculate that
\begin{align*}
 \mathbb E (V^{\top} V)_{12} & = \sum_{j=1}^r \mathbb E \wh S_{j1} \wh S_{j2}  = \frac{r}{n}\mathbb E\left(\sum_{k=1}^n T_{k1}T_{k2}\right) =0, \\
 \text{Var} \left[ (V^{\top} V)_{12}\right] & = \mathbb E\sum_{j,j'=1}^r  \wh S_{j1} \wh S_{j2}  \wh S_{j'1} \wh S_{j'2}  =\frac{1}{n(n-1)}\sum_{j, j' =1 }^r  \mathbb E \sum_{k\ne k' \in \llbracket 1,n\rrbracket} T_{jk} T_{jk'}  T_{j'k} T_{j'k'}   \\
&=\frac{1}{n(n-1)}\sum_{j, j' =1 }^r  \mathbb E \sum_{k, k' =1}^n T_{jk} T_{jk'}  T_{j'k} T_{j'k'}  - \frac{1}{n(n-1)}\sum_{j, j' =1 }^r  \mathbb E \sum_{k=1}^n T_{jk}^2  T_{j'k}^2 \le \frac{r}{n(n-1)}.
\end{align*}
Then we conclude \eqref{claim unifV} by Chebyshev's inequality. 
\end{proof}

In fact, we know that a much stronger bound holds:
\be\label{bestV}
V^{\top} V = \xi_n I_k + \OO_\prec(n^{-1/2}) 
\ee
using more advanced tools from random matrix theory. Although we will not use such a strong bound in this paper,  it may be helpful to keep in mind that our result can be improved to give much better convergence rates. {  For example, if the entries of $X$ have finite $a$-th moment for a constant $a>4$, then the results in Theorem \ref{sketchthm1} can be obtained with an explicit convergence rate $\OO_\prec(n^{-\delta})$ for $\delta=1/2-2/a$. In particular, if the entries of $X$ have finite moments up to any order (e.g. when the entries of $X$ are sub-Gaussian), then we can get the optimal convergence rate $ \OO_\prec(n^{-1/2}) $ using \eqref{bestV} in our proof.} The behavior of submatrices of random orthogonal matrices has been well studied, see e.g., \cite{jiang2006many,jiang2019distances} and references therein. These works study the approximation by Gaussian random random variabes, and require more than what we need in this work.

Now we are ready to prove the eigenvalue estimates in Theorem \ref{sketchthm1}. 

\begin{proof}[Proof of \eqref{outlierevalue} and \eqref{outlierevalue2}]
 Our starting point is Lemma \ref{lem_pertubation}, so we need to study the behavior of $A^{\top} G(x)A$. By Theorem \ref{lem_localout} and Theorem \ref{lem_localin}, we can choose a high-probability event $\Xi\subset\Omega$, such that the following estimates hold for some constants $c_0,c_1,C_0>0$ and fixed large integer $\varpi\in \N$: 
\begin{equation}\label{aniso_lawev}
\mathbf{1}(\Xi) \norm{A^{\top} (G(z)-\Pi(z)) A} \le n^{-\delta/2} ,\quad \text{for }z\in S_{edge}(c_0,C_0,c_1)\cup S_{out}(c_0,C_0);
\end{equation}
\begin{equation} \label{eq_bound2ev}
\mathbf{1}(\Xi)\left|\lambda_i -\lambda_+\right| \leq n^{-\delta/2}, \quad  \text{ for }1\le i \le \varpi . 
\end{equation}
We remark that the randomness of $X$ only comes into play to ensure that $\Xi$ holds with high probability. The rest of the proof is restricted to $\Xi$ only, and will be entirely deterministic. 

We denote $d_c:= \sqrt{\gamma_n/\xi_n}$, and define the index sets
\begin{equation} \label{eq_otau}
\mathcal{O}_{+}:=\left\{1\le i \le k: d_i > d_c \right\},
\end{equation}
which is the set of the indices of outliers. We also denote $k_+:=|\cal O_+|$.

\vspace{5pt}

%
\noindent{\bf Step 1:}  Our first step is to prove that on $\Xi$, there are no eigenvalues outside a neighborhood of the classical outlier locations $\theta_i$. 
For each $1\le i \le k_+,$ we define the permissible interval 
\begin{equation*}
\rI_i  \equiv \rI_i (\e):=\left[\theta_i- \epsilon, \theta_i+\e\right],
\end{equation*} 
where $\e$ is a constant that can be arbitrarily small as long as we have 
\be\label{non-overlap}\rI_i \cap \rI_j = \emptyset, \quad i\ne j.\ee
Moreover, we define the permissible interval $\rI_0\equiv \rI_0(\e):=\left[0, \lambda_+ + \e\right]$ for other eigenvalues, and denote
\begin{equation}\label{I0}
\rI :=\rI_0 \cup \Big(\bigcup_{i \in \mathcal{O}_+}\rI_i \Big).
\end{equation}
We claim the following result.
\begin{lemma}\label{lem_gapI}
The complement of $\rI $ contains no eigenvalues of $\ctQ_1.$
\end{lemma}
\begin{proof}
By \eqref{masterx}, \eqref{aniso_lawev} and \eqref{eq_bound2ev}, we see that $x \notin \rI_0$ is an eigenvalue of $\ctQ_1$  if and only if 
\begin{equation}\label{eq_pertubhold}
 \mathcal{D}^{-1}+x^{1/2}A^{\top}G(x) A = \mathcal{D}^{-1}+x^{1/2}A^{\top} \Pi(x) A + \OO( n^{-\delta/2})
\end{equation}
is singular. 
By \eqref{eq_bound2ev}, we know on $\Xi$, $\wt\lambda_1 \le (\sqrt{\lambda_1} + d_1)^2 \le C_0$ as long as $C_0$ is taken large enough. Here we used the trivial estimate for the operator norms,
$$\wt\lambda_1^{1/2}=\|Y\|\le \|X\|+d_1=\sqrt{\lambda_1} + d_1.$$ 
Moreover, by \eqref{claim unifV}, we have  that with probability $1-\oo(1)$,
\be\nonumber
\mathcal{D}^{-1}+x^{1/2}A^{\top} \Pi(x) A= \left( {\begin{array}{*{20}c}
   {  -x^{-1/2}\left(1+m_{2c}(x)\right)^{-1} }I_k & {D^{-1}}   \\
   {D^{-1}} & {x^{1/2} m_{2c}(x)} I_k \\
   \end{array}} \right) + \oo(1) \quad \text{for all $x\in [0,C_0]\setminus \rI$.}
\ee
Thus to prove the lemma, it suffices to show that if $x\in [0,C_0]\setminus \rI$, then
\be\label{diag_big000}
\left|\frac{m_{2c}(x)}{1+m_{2c}(x)} + d_i^{-2}\right| \ge c, \quad 1\le i \le k,
\ee
for some constant $c>0$ depending only on $\e$. If \eqref{diag_big000} holds, then we immediately obtain that 
$$\left\|  \left( {\begin{array}{*{20}c}
   {  -x^{-1/2}\left(1+m_{2c}(x)\right)^{-1} }I_k & {D^{-1}}   \\
   {D^{-1}} & {x^{1/2} m_{2c}(x)} I_k \\
   \end{array}} \right)^{-1} \right\| =\OO(1),$$
and hence $(\mathcal{D}^{-1}+x^{1/2}A^{\top} G(x) A)$ must be non-singular. This means that $x$ cannot be an eigenvalue of $\wt Q_1$. 

For the proof of \eqref{diag_big000}, recall that we have defined $\theta_i$ such that (see \eqref{getthetai})
$$ \frac{m_{2c}(\theta_i)}{1+m_{2c}(\theta_i)} =-  d_i^{-2},\quad 1\le i \le k_+.$$
Thus we have for $1\le i \le k_+$,
$$ \left|\frac{m_{2c}(x)}{1+m_{2c}(x)} + d_i^{-2}\right| =\left|\frac{m_{2c}(x)}{1+m_{2c}(x)} - \frac{m_{2c}(\theta_i)}{1+m_{2c}(\theta_i)}\right| \gtrsim |m_{2c}(x) - m_{2c}(\theta_i)| \gtrsim 1,$$
where we used \eqref{Piii} in the second step and \eqref{eq_mdiff} in the last step. Moreover, using $0>m_{2c}(x) \ge m_{2c}(\lambda_+ + \e) > -1$ for $x\in [0,C_0]\setminus \rI$ and $m_{2c}(\lambda_+)=-(1+d_c^2)^{-1}$, we get that $k_+\le i \le k$,
$$ \frac{m_{2c}(x)}{1+m_{2c}(x)} + d_i^{-2} \ge \frac{m_{2c}(x)}{1+m_{2c}(x)} + d_c^{-2} \gtrsim m_{2c}(x) + \frac{1}{1+d_c^2} \ge m_{2c}(\lambda_+ +\e) -m_{2 c}(\lambda_+) \gtrsim 1 ,$$ 
where  again we used \eqref{Piii} in the second step and \eqref{eq_mdiff} in the last step. This concludes \eqref{diag_big000}, which further proves Lemma \ref{lem_gapI}.
\end{proof}



\noindent{\bf Step 2:} 
In this step, we claim the following result.   

\begin{lemma}Each $\mathbf{\rI}_i $, $1\le  i \le k_+$, contains precisely one eigenvalue of $\ctQ_1$. \end{lemma}
\begin{proof}
Fix any $1\le i \le k_+$ and pick up a sufficiently small positively oriented closed contour $\mathcal{C} \subset \mathbb{C}/[0, \lambda_+]$ that encloses $\theta_i$ but no other point of the set $\{\theta_i\}_{i=1}^{k_+}.$ By \eqref{non-overlap}, we can choose the contour $\mathcal{C}$ as a circle around $\theta_i$ with radius $\e$.  

Now we define two functions
\begin{equation*}
h(z):=\det(\mathcal{D}^{-1}+z^{1/2} A^{\top} G(z) A), \quad  l(z)=\det(\mathcal{D}^{-1}+z^{1/2} A^{\top}\Pi(z) A).
\end{equation*}  
The functions $h,l$ are holomorphic on and inside $\mathcal{C}$ when $n$ is sufficiently large by \eqref{eq_bound2ev}. Moreover, by the construction of $\mathcal{C},$ the function $l$ has precisely one zero inside $\mathcal{C}$ at $\theta_i.$ By \eqref{aniso_lawev} and a similar argument as for \eqref{diag_big000}, we have  
\begin{equation*}
\min_{z \in \mathcal{C}}|l(z)| \gtrsim 1, \quad |h(z)-l(z)|=\OO(n^{-\delta/2}).
\end{equation*}
The lemma then follows from Rouch{\' e}'s theorem.
\end{proof}
Combining Steps 1 and 2 with a simple eigenvalue counting argument, we obtain that 
\be\label{firste}1(\Xi) |\wt\lambda_i - \theta_i| \le \e, \quad 1\le i \le k_+, \ee
and
\be\label{seconde}1(\Xi)  \wt\lambda_i  \le 1(\Xi) \lambda_+ + \e, \quad k_+\le i \le k, \ee
for any small constant $\e>0$. The first estimate \eqref{firste} concludes \eqref{outlierevalue}. To prove \eqref{outlierevalue2}, we still need to provide a lower bound for $\wt \lambda_i$, $ k_+\le i \le k$. In fact, with \eqref{interlacing_eq0} and \eqref{eq_bound2ev}, we obtain that 
 $$1(\Xi)  \wt\lambda_i  \ge 1(\Xi) \lambda_+ - n^{-\delta/2}, \quad k_+\le i \le k. $$
Together with \eqref{seconde}, we conclude \eqref{outlierevalue2}.
\end{proof}

Finally we prove the eigenvector estimates in Theorem \ref{sketchthm1}. 

\begin{proof}[Proof of \eqref{outlierevector} and \eqref{outlierevector2}]
In the following proof, we again always work on the event $\Xi$ such that \eqref{aniso_lawev} and \eqref{eq_bound2ev} hold. Again the randomness of $X$ only comes into play to ensure that $\Xi$ holds with high probability, and the rest of the proof is deterministic on $\Xi$.

We denote $\mathcal E(z)=z^{1/2}A^{\top}(\Pi(z) -G(z))A.$ Then we can write 
\begin{equation*}
z^{1/2}A^{\top}G(z) A= z^{1/2} A^{\top} \Pi(z) A -\mathcal E(z). 
\end{equation*} 
By \eqref{aniso_lawev}, we have 
\be\label{boundEE}
\|\mathcal E(z)\|\le n^{-\delta/2} \quad \text{ for } z\in S_{out}(c_0, C_0). 
\ee
We now perform a resolvent expansion for the denominator in (\ref{ev_origin}) as
\be\label{resolvent_3rd}
\begin{split}
\frac{1}{\mathcal D^{-1}+ z^{1/2}A^{\top}G(z)A}&= \frac{1}{\mathcal D^{-1}+  z^{1/2}A^{\top}\Pi(z)A}  + \frac{1}{\mathcal D^{-1}+  z^{1/2}A^{\top}\Pi(z)A} \mathcal E \frac{1}{\mathcal D^{-1}+  z^{1/2}A^{\top}G(z)A} . 
\end{split}
\ee
We define the contour $\Gamma_i=\{z : |z-\theta_i| =\e\}$, where $\e>0$ is a sufficiently small constant such that 
\be\label{well separa} \inf_{ z\in \Gamma_i}\left( |z-\lambda_+|\wedge \min_{1\le i \le k_+}|z-\theta_i| \right)\ge \e. \ee
By \eqref{outlierevalue} and \eqref{outlierevalue2}, for large enough $n$, we have (i) $\Gamma_i$ only encloses $\wt\lambda_i$, and no other eigenvalue of $\ctQ_1$; (ii) $\Gamma_i$ 
does not enclose any pole of $G$ (i.e. any eigenvalue of $\cal Q_1$). Note (i) implies that $\Gamma_i$  only encloses one pole of $(\mathcal D^{-1}+ z^{1/2}A^{\top}G(z)A)^{-1}$ at $\wt\lambda_i$. Moreover, with a similar argument as for \eqref{diag_big000}, one can obtain that
$$\max_{z\in \Gamma_i}\|(\mathcal D^{-1}+  z^{1/2}A^{\top}\Pi(z)A)^{-1}\|\le c^{-1}$$
for some constant $c>0$ depending on $\e$ only. Together with \eqref{boundEE}, we get that 
\be\label{boundEE2}
\max_{z\in \Gamma_i}\left\|\frac{1}{\mathcal D^{-1}+  z^{1/2}A^{\top}\Pi(z)A} \mathcal E \frac{1}{\mathcal D^{-1}+  z^{1/2}A^{\top}G(z)A}\right\|\lesssim n^{-\delta/2} . 
\ee

Now inserting \eqref{resolvent_3rd} into (\ref{ev_origin}), choosing $\Gamma_i$ as above, and using \eqref{boundEE2}, we obtain from the Cauchy's integral formula that for $1\le i \le k_+$ and $1\le j \le k$,
\begin{align} 
|\langle u_j, \wt{\bxi}_i\rangle|^2 &= \frac{\delta_{ij}}{2\pi \ii (\wt\lambda_i)^{1/2}} \oint_{\Gamma_i}(0,d_i^{-1})\left( {\begin{array}{*{20}c}
   {  -z^{-1/2}\left(1+m_{2c}(z)\right)^{-1} } & {d_i^{-1}}   \\
   {d_i^{-1}} & {z^{1/2} m_{2c}(z)}  \\
   \end{array}} \right)^{-1}  \begin{pmatrix} 0  \\ d_i^{-1}\end{pmatrix} \dd z +\oo(1) \nonumber\\
 &=  \frac{\delta_{ij}}{2\pi \ii \theta_i \left(1+d_i^2\right)} \oint_{\Gamma_i} \frac1{  {m_{2c}(z)}  + (1+d_i^{2})^{-1}}\dd z +\oo(1)= \frac{1}{\theta_i \left(1+d_i^2\right) m_{2c}'(\theta_i)}   +\oo(1)  \nonumber\\
 &= \delta_{ij}\frac{g_{2c}'(-(1+d_i^2)^{-1})}{\theta_i \left(1+d_i^2\right) } +\oo(1) = \frac{\xi_n - \frac{\gamma_n}{d_i^4}}{\xi_n + \frac{\gamma_n}{d_i^2}} +\oo(1).
\label{cos_outer2}
\end{align}
where we used \eqref{outlierevalue} in the second and third steps, and \eqref{solv g2c} in the last step. This concludes \eqref{outlierevector}.

Next we prove \eqref{outlierevector2}. For $k_+ \le i \le k$, we choose a specific spectral parameter as 
$z_i=\wt\lambda_i+\ri \eta_i $, where $\eta_i := n^{-\e}$ for some sufficiently small constant $\e>0$. Note that by \eqref{outlierevalue2}, we have $z_i \in S_{edge}(c_0, C_0, c_1)$. 
With the spectral decomposition \eqref{spectral1}, we obtain that
\begin{equation}\label{eq_nonspike1}
\im \wt G_{uu} (E+\ii\eta)= \sum_{j = 1}^{p} \frac{\eta |\langle u,\wt{ \xi}_j\rangle|^2}{(\wt\lambda_j-E)^2 + \eta^2} \Rightarrow \left|\langle u, \wt\bxi_i \rangle\right|^2 \leq \eta_i \im \langle u, \widetilde{G}(z_i) u\rangle.
\end{equation}
Applying \eqref{Woodbury} to $\wt G(z_i) = (H + z_i^{1/2}A \cal D A^{\top} - z_i)^{-1}$, we obtain that
\begin{equation} \label{eq_evgenerralreprest}
\begin{split}
\langle u,\widetilde{G}(z_i)u\rangle&={G}_{uu}(z_i)  - z_i^{1/2}u^{\top} G(z_i) A \frac{1}{\mathcal D^{-1}+ z_i^{1/2}A^{\top}G(z_i)A} A^{\top} G(z_i)u
\end{split}
\end{equation}
For the denominator, we claim that for sufficiently small constant $\e$, 
\be\label{claim opbound1}
\left\|\left(\mathcal D^{-1}+ z^{1/2}A^{\top}G(z_i)A\right)^{-1}\right\| \lesssim \left(\im m_{2c}(z_i)\right)^{-1}.
\ee
To prove this claim, we first notice that 
\be\label{diag_big}
\left|\frac{m_{2c}(z_i)}{1+m_{2c}(z_i)} + d_i^{-2}\right| \gtrsim \left| m_{2c}(z_i) + \frac1{1+d_i^{2}}\right| \gtrsim \im m_{2c}(z_i) \gtrsim  \eta_i,  
\ee
where we used \eqref{Piii} in the second step and \eqref{eq_estimm} in the last step. This shows that the smallest singular value of 
$$M(z_i)= \mathcal D^{-1}+ z^{1/2}A^{\top}\Pi(z_i)A$$
is at least of order $\gtrsim \im m_{2c}(z_i)$. Then by \eqref{aniso_lawev} we have that
\begin{align*}
\mathcal D^{-1}+ z^{1/2}A^{\top}G(z_i)A =M(z_i) + \OO(n^{-\delta/2}).
\end{align*}
Thus as long as we choose $\e <\delta/2$, the bound \eqref{claim opbound1} holds.

Now using \eqref{aniso_lawev}, we get
$${G}_{uu}(z_i) =\OO(1), \quad \|u^{\top} G(z_i) A\|=\OO(1).$$
Together with \eqref{eq_evgenerralreprest} and \eqref{claim opbound1}, we obtain that
$$\eta_i\im \langle u, \widetilde{G}(z_i) u\rangle \lesssim \frac{\eta_i}{\im m_{2c}(z_i)} \lesssim \max\{ \sqrt{\eta_i}, \sqrt{\kappa_{\wt\lambda_i}}\}.$$
where $\kappa_{\wt\lambda_i}=|\wt\lambda_i -\lambda_+|$ (recall \eqref{KAPPA}) and  in the last step we used
$$\im m_{2c}(\eta_i) \gtrsim \min\left\{\sqrt{\eta_i}, \frac{\eta_i}{\sqrt{\kappa_{\wt\lambda_i}} }\right\}  $$
by \eqref{eq_estimm}. Hence with $\eta_i= n^{-\e}$ and $\kappa_{\wt\lambda_i}=\oo(1)$ by \eqref{outlierevalue2}, we conclude from \eqref{eq_nonspike1} that 
$$\left|\langle u, \wt\bxi_i \rangle\right|^2 \leq \eta_i \im \langle u, \widetilde{G}(z_i) u\rangle =\oo(1). $$
This completes the proof of \eqref{outlierevector2}.
\end{proof}

\section{Proof of Theorem  \ref{sketchthm4}}\label{pf sketch1.5}

In this section, we prove Theorem \ref{sketchthm4}. Somewhat informally we have the following calculations. We will turn them into a fully rigorous proof afterwards.

We will need the following resolvents of $S$ (compare them with Definition \ref{defn_resolvent}):
\begin{equation}\label{def_greenRS}
\mathcal R_{1}(S,z):=\left(SS^{\top} -z\right)^{-1} , \quad \mathcal R_{2}(S,z):=\left(S^{\top} S -z\right)^{-1} ,
\end{equation}
and the normalized traces 
\begin{equation}\label{defn_m1m2S}
m_1^S(z):=  \frac{1}p\tr  \mathcal R_1(z) ,\quad  m_2^{S}(z):=\frac{1}{n}\tr \mathcal R_2(z) . 
\end{equation}
Let $m_{1c}^S$ and $m_{2c}^S$ be the limiting Stieltjes transforms of $SS^{\top}$ and $S^{\top} S$. They are determined by the following self-consistent equations:
\begin{equation}\label{m1m2Sc}
\begin{split}
& m_{1c}^S(z) =  \frac{1}{-z\left[1+m_{2c}^S(z) \right]} ,\quad m_{2c}^S(z) =   \frac{1}{-z\left[1+\xi_n m_{1c}^S(z) \right]}  .
\end{split}
\end{equation}
Solving \eqref{m1m2Sc}, we can obtain that 
\be\label{solv m12cS}
\begin{split}
& m_{1c}^S(z)=\frac{-(z-1 +\xi_n) + \sqrt{(z-\lambda_+^S)(z-\lambda_-^S)} }{2z\xi_n}, \quad m_{2c}^S(z)=\frac{-(z +1 - \xi_n) + \sqrt{(z-\lambda^S_+)(z-\lambda^S_-)} }{2z}.
\end{split}
\ee
where $\lambda^S_{\pm}$ are the edges of the support of the standard Marchenko-Pastur (MP) distribution,
$$\lambda^S_{\pm}= (1 \pm \sqrt{\xi_n})^2.$$
Denoting the inverse functions of $m_{1,2c}^S$ by $g_{1,2c}^S$, we also obtain from the equations in \eqref{m1m2Sc} that 
\be\label{g12cS}g_{1c}^S(m)= \frac{1}{1+\xi_n m}-\frac{1}{m}, \quad g_{2c}^S(m)= \frac{\xi_n}{1+m}-\frac{1}{m}.\ee

By the \emph{local law} for isotropic sample covariance matrices, Theorem 2.4 of \cite{isotropic}, we know that for any deterministic unit vectors $u_1, u_2\in \R^r$ and $v_1, v_2\in \R^n$, 
\be\label{localmS0}
\langle u_1, \cal R_{1}(z) u_2\rangle= m_{1c}^S(z) \langle u_1,u_2\rangle +\oo(1),\quad \langle v_1, \cal R_{2}(z) v_2\rangle= m_{2c}^S(z) \langle v_1,v_2\rangle +\oo(1),
\ee
with high probability, uniformly in the following region bounded away from the support of the MP law:
\be\label{domainiid} z \in S_\tau:= \{z\in \C: \text{dist}(z,[(1-\sqrt{\xi_n})^2 , (1+\sqrt{\xi_n})^2  ]) \ge \tau\}\ee
for any constant $\tau>0$. In particular, \eqref{localmS0} implies that
\be\label{localmS}
m_{1}^S(z) = m_{1c}^S(z) +\oo(1),\quad m_{2}^S(z) = m_{2c}^S(z) +\oo(1),
\ee
uniformly in $z \in S_\tau$. 

Now using $m_1^S$, we can write the self-consistent equations in \eqref{separa_m12} as
\begin{equation} \label{selfiid}
\begin{split}
{m_{1c}(z)} = \gamma_n \frac{1}{-z\left[1+m_{2c}(z) \right]} , \quad {m_{2c}(z)} &=\frac{\xi_n}{-zm_{1c}(z)}\left( 1- \frac{1}{m_{1c}(z)}m^{S}_{1}(-m_{1c}^{-1}(z)) \right).
\end{split}
\end{equation}
Suppose that \eqref{localmS} can be applied to $m_1^S$. Then we obtain the following self-consistent equation satisfied by $m_{1c}$:
\begin{equation} \label{solv g1c0}
\begin{split}
 \frac{\gamma_n}{m_{1c}(z)} &= -z + \frac{\xi_n}{m_{1c}(z)}\left( 1- \frac{1}{m_{1c}(z)}m^{S}_{1c}(-m_{1c}^{-1}(z)) \right) +\oo(1).
\end{split}
\end{equation}
This immediately gives the inverse function $g_{1c}$ of $m_{1c}$:
\be\label{solv g1c}
g_{1c}(m)=- \frac{\gamma_n}{m}   + \frac{\xi_n}{m}\left( 1- \frac{1}{m}m^{S}_{1c}(-m^{-1}) \right) + \oo(1).
\ee
Next we find the function $m_{1c}(z)$. Using the function $g_{1c}^S$ as an inverse function of $m_{1c}^S$, we can obtain that $m_{1c}$ in \eqref{solv g1c0} satisfies (approximately) the following equation:
\begin{equation}\nonumber
\begin{split}
& -\frac{1}{m_{1c}} =g_{1c}^S\left(-\frac{\gamma_n - \xi_n }{\xi_n} m_{1c}- \frac{z}{\xi_n}m_{1c}^2\right) = \frac{1}{1-(\gamma_n - \xi_n)m_{1c} - zm_{1c}^2 } + \frac{\xi_n}{(\gamma_n - \xi_n)m_{1c}+zm_{1c}^2} , 
\end{split}
\end{equation}
which can be reduced to a cubic equation
\be \label{cube_g} z^2m_{1c}^3 - z(1+\xi_n-2\gamma_n)m_{1c}^2 - \left( z + (1-\gamma_n)(\gamma_n-\xi_n)\right) m_{1c} - \gamma_n =0. \ee
There is only one solution to this equation such that $\im m_{1c}(z) > 0$ whenever $\im z > 0$. After obtaining $m_{1c}(z)$, we immediately obtain that the Stieltjes transform $m$, the limit of $\frac{1}{p} \mathrm{Tr} (Y^{\top} Y- z)^{-1}$, has the form 
$$ m(z) =m_c(z) + \oo(1),\quad  m_c(z) :=  {\xi^{-1}_n}m_{1c}(z),$$
with high probability. Hence we can define the asymptotic spectral density $\rho_c$ using the inversion formula $\rho_{c}(E) = \lim_{\eta\downarrow 0} \pi^{-1}\Im\, m_{c}(E+\ii \eta)$, and find its right edge $\lambda_+$.

To study the spiked eigenvalues and eigenvectors, we again need to study the master matrix $M(x)$ in \eqref{defnMx}. 
With the Woodbury matrix identity \eqref{Woodbury}, 
we obtain that
$$- x^{-1/2} W^{\top} S^{\top}\frac{1}{1+m_{1c} SS^{\top}} SW = - \frac{1}{x^{1/2}m_{1c}} W^{\top}\left( {1} - \frac{1}{1 +m_{1c} S^{\top} S}\right)W . $$
Applying the local law \eqref{localmS0}, we obtain that with high probability,
$$- x^{-1/2} W^{\top} S^{\top}\frac{1}{1+m_{1c} SS^{\top}} SW = - \frac{1}{x^{1/2}m_{1c}} \left( {1} - \frac{1}{m_{1c}}m_{2c}^S (-m_{1c}^{-1})\right)I_k +\oo(1) . $$
Now the eigenvalue master equation \eqref{master_evalue} becomes, approximately,
\be\label{master_evalue2}
\det \left( {\begin{array}{*{20}c}
   {  x^{1/2}\gamma_n^{-1} m_{1c}}I_k & {D^{-1}}   \\
   {D^{-1}} & { - \frac{1}{x^{1/2}m_{1c}} \left( {1} - \frac{1}{m_{1c}}m_{2c}^S (-m_{1c}^{-1})\right)}I_k  \\
   \end{array}} \right) = 0,
\ee
which gives the following equations for $1\le i \le k,$
\begin{align*}
-\gamma_n^{-1}\left( {1} - \frac{1}{m_{1c}}m_{2c}^S (-m_{1c}^{-1})\right)=d_i^{-2} .
\end{align*}
Using the inverse function of $m_{2c}^S$, $g_{2c}^S$, in \eqref{g12cS}, we obtain that
\be\label{solveoutlier}-m_{1c}^{-1}  = g_{2c}^S\left( (1+\gamma_n d_i^{-2})m_{1c}\right) \Rightarrow m_{1c}(x)=-\frac{\gamma_n d_i^{-2}}{\left(1+\gamma_n d_i^{-2}\right)\left(\xi_n+\gamma_n d_i^{-2}\right)} .\ee
Similarly to \eqref{cond outlier}, in order for the signal strength $d_i$ to give an outlier, we need to have that 
\be\label{go}
\al(d_i):= -\frac{\gamma_n d_i^{-2}}{\left(1+\gamma_n d_i^{-2}\right)\left(\xi_n+\gamma_n d_i^{-2}\right)}> m_{1c}(\lambda_+) .
\ee
In particular, there exists an $d_c>0$ determined by the equation 
\be\label{defn_dc}
\al(d_c)= m_{1c}(\lambda_+) ,
\ee
such that \eqref{go} holds if and only if $d_i > d_c$. Suppose $d_i>d_c$, then the $i$-th outlier $\wt\lambda_i$ will concentrate around 
\be\label{thetaiid}
\theta_i = g_{1c}  \left( -\frac{\gamma_n d_i^{-2}}{\left(1+\gamma_n d_i^{-2}\right)\left(\xi_n+\gamma_n d_i^{-2}\right)}\right)
\ee
by \eqref{solveoutlier}, where $g_{1c}$ is defined in \eqref{solv g1c}.



Next we study the spiked eigenvector corresponding to the outlier $\wt\lambda_i$ using the angle master equation \eqref{ev_origin}. First it is easy to see that $|\langle u_j, \wt{\bxi}_i\rangle|^2 =\oo(1)$ if $j\ne i$. If $j=i$, then we have with high probability,
\be \label{cos gauss}
\begin{split}
|\langle u_i, \wt{\bxi}_i\rangle|^2 &= \frac{1}{2\pi \ii \sqrt{\theta_i}d_i^{2}} \oint_{\Gamma_i}\frac{ z^{1/2}\gamma_n^{-1} m_{1c}(z)}{-\gamma_n^{-1}\left( {1} - \frac{1}{m_{1c}(z)}m_{2c}^S (-m_{1c}^{-1}(z))\right)-d_i^{-2}} \dd z + \oo(1)\\
&=- \frac{m_{1c}^2(\theta_i)}{2\pi \ii d_i^{2}} \oint_{\Gamma_i}\frac{ 1}{(1+\gamma_n d_i^{-2})m_{1c}(z) - m_{2c}^S (-m_{1c}^{-1}(z))} \dd z + \oo(1)\\
&=- \frac{m_{1c}^2(\theta_i)}{2\pi \ii d_i^{2}} \oint_{g_{1c}(\Gamma_i)}\frac{g_{1c}'(\zeta)}{(1+\gamma_n d_i^{-2})\zeta - m_{2c}^S (-\zeta^{-1})} \dd \zeta + \oo(1)\\
&= \frac{\al_i^2}{d_i^{2}} \frac{g_{1c}'(\al_i)}{\al_i^{-2}(m_{2c}^S)' (-\al_i^{-1})-(1+\gamma_n d_i^{-2}) } + \oo(1),
\end{split}
\ee
where we used that $\al_i \equiv \al(d_i)  =m_{1c}(\theta_i) +\oo(1)$.


One can see that in order to make the above calculations rigorous, we only need to repeat the arguments in the proof for Theorem \ref{sketchthm1}, except that there are two extra complications to deal with. (i) We need to verify that the ``right edge regularity" condition \eqref{assm_gap} holds, such that the anisotropic local law \emph{outside} the spectrum (Theorem \ref{lem_localout}) can be applied. (ii) We need to verify that $-m_{1c}^{-1}(z)\in S_\tau $ for some constant $\tau>0$ for all $z\in S_{out}(c_0, C_0) \cup S_{edge}(c_0, C_0,c_1)$, such that we can apply \eqref{localmS0}.

\begin{proof}[Proof of Theorem \ref{sketchthm4}]

We first verify the condition \eqref{assm_gap}. The self-consistent equation \eqref{separa_m12} now becomes
\begin{equation}\label{separa_m12app}
\begin{split}
& -z {m_{1c}(z)} = \frac{\gamma_n}{ 1+m_{2c}(z)},\quad  -z{m_{2c}(z)} = \frac{1}{n} \sum_{\mu= 1}^r \frac{s_\mu}{ 1+s_\mu m_{1c}(z) } .
\end{split}
\end{equation}
Note that by the last statement of Lemma \ref{lambdar}, the two sums on the right-hand side of the above two equations are all positive sums if we take $z=\lambda_+$. Suppose $1+m_{2c}(\lambda_+)  =\oo(1)$, then from the first equation we get that $| m_{1c}(\lambda_+)|\gg 1$, which contradicts the fact that $| m_{1c}(\lambda_+)|\le s_1^{-1}$. 

On the other hand, suppose 
\be\label{hype}1+m_{1c}(\lambda_+) s_1=\oo(1).\ee From \eqref{separa_m12app}, we obtain the following self-consistent equation for $m_{2c}$:
\begin{equation}\nonumber
\begin{split}
f(m_{2c}(z))=0, \quad f(m_{2c}(z)):= {m_{2c}(z)} - \frac{1}{n} \sum_{\mu= 1}^r \frac{s_\mu(1+m_{2c})}{ -z (1+m_{2c})+s_\mu \gamma_n  } .
\end{split}
\end{equation}
If we regard $f$ as a function of $m_{2c}$, then by Lemma 2.5 of \cite{yang2019spiked}, we know that $\partial_{m_{2c}}f=0$ at $z=\lambda_+$. Hence we get
\begin{equation}\label{separa_m2capp}
\begin{split}
1= \frac{1}{n} \sum_{\mu= 1}^r \frac{s_\mu^2\gamma_n }{ [-\lambda_+ (1+m_{2c}(\lambda_+))+s_\mu \gamma_n]^2  } = \frac{m_{1c}^2(\lambda_+)}{n} \sum_{\mu= 1}^r \frac{s_\mu^2\gamma^{-1}_n }{ [1+s_\mu m_{1c}(\lambda_+)]^2  } .
\end{split}
\end{equation}
By the eigenvalue rigidity result for $SS^{\top}$, Theorem 2.10 of \cite{isotropic} or Theorem 3.8 of \cite{yang2019spiked}, we know that for any small constant $\e>0$,
$$\max_{1\le \mu \le \e n} | s_\mu - s_1 | \le C\e^{2/3} \quad \text{with high probability,}$$
for some constant $C>0$ that is independent of $\e$. Together with the hypothesis \eqref{hype}, we obtain from \eqref{separa_m2capp} that with high probability,
$$m_{1c}^2(\lambda_+) \le C \e^{1/3}$$
for some constant $C>0$ that is independent of $\e$. However, this contradicts \eqref{hype} if we take $\e$ to be sufficiently small.  

In sum, we see that \eqref{assm_gap} must hold with high probability.

\vspace{5pt}

Next we show that $-m_{1c}^{-1}(z)\in S_\tau$ for some constant $\tau>0$ for all $z\in \wt S(c_0, C_0)$ as long as $c_0$ is sufficiently small (recall \eqref{Sc0C0}). Again by the eigenvalue rigidity result for $SS^{\top}$, Theorem 2.10 of \cite{isotropic}, we know that $|s_1 - (1+\sqrt{\xi_n})^2|=\oo(1)$ with high probability. Hence together with \eqref{assm_gap}, we have
$$ - m^{-1}_{1c}(\lambda_+) \ge (1+\sqrt{\xi_n})^2 + c_1 $$
for some constant $c_1>0$ depending on $\tau$. Moreover, since $m_{1c}(\lambda_+) \le m_{1c}(x) <0$ for $x>\lambda_+$ and $m_{1c}(x)$ is monotonically increasing in $x\in (\lambda_+,\infty)$, we obtain that 
$$ \inf_{x\ge \lambda_+}\left[- m^{-1}_{1c}(x)\right]\ge (1+\sqrt{\xi_n})^2 + c_1 .$$
Next if $\text{dist}(z, [ \lambda_+, C_0]) \le \delta$ for some constant $\delta>0$, by \eqref{eq_mcomplex} we obtain that
$$ \inf_{z: \text{dist}(z, [ \lambda_+, C_0]) \le \delta}\min\{ 1+ s_\mu m_{1c}(z)\} \ge c_1/2$$
as long as $\delta$ is taken sufficiently small. If we take $c_0\le \delta$, the above estimate covers all the domain $\wt S(c_0, C_0)$ except for the part $\{ z\in \wt S(c_0, C_0): \im z\ge c_0\}$. On this part of domain, we use \eqref{eq_estimm} to get that 
$$ \inf_{z\in \wt S(c_0, C_0): \im z\ge c_0 }\left[- m^{-1}_{1c}(z)\right] \ge c_2'  \inf_{z\in \wt S(c_0, C_0): \im z\ge c_0 }\im m_{1c}(z) \ge c_2$$
for some constants $c_2, c_2' >0$ depending on $c_0$. 

In sum, we get that $-m_{1c}^{-1}(z)\in S_{\tau'}$ for some constant $\tau' \equiv \tau'_{c_1,c_2}>0$ for all $z\in \wt S(c_0, C_0)\supset  S_{out}(c_0, C_0) \cup S_{edge}(c_0, C_0,c_1)$. 

\vspace{5pt}

The above proof justifies our calculations between \eqref{selfiid} and \eqref{cos gauss}, and the rest of the proof is exactly the same as the one for Theorem \ref{sketchthm1}. So we omit the details.
\end{proof}

\section{Proof of Theorem \ref{sketchthm2} and Theorem \ref{sketchthm5}}
\label{pfunif}

As remarked at the beginning of Appendix \ref{pf sketch1}, we only need to analyze the master equations \eqref{master_evalue} and \eqref{ev_origin} under the settings of  Theorem \ref{sketchthm2} and Theorem \ref{sketchthm5}, respectively. The rest of the proof will be exactly the same as the one for Theorem \ref{sketchthm1} in Appendix \ref{pf sketch1}. 
\begin{proof}[Proof of Theorem \ref{sketchthm2}]
We define the random variable 
$$\wh \xi_n := \frac{1}{n}\sum_{i=1}^n S_{ii}$$
to be the fraction of non-zero diagonal entries of $S$. We fix a realization of $S$. Then equations in \eqref{separa_m12} become
\begin{equation}
\begin{split}
& {m_{1c}(z)} = \gamma_n \frac{1}{-z\left[1+m_{2c}(z) \right]} ,\quad {m_{2c}(z)} =  \wh \xi_n  \frac{1}{-z\left[1+m_{1c}(z) \right]}  .
\end{split}
\end{equation}
Thus \eqref{solv m2c} and \eqref{solv g2c} are accurate asymptotically, since $\wh \xi_n$ concentrates around $\xi_n$ for large $r$ and $n$. 

We can calculate that 
$$- x^{-1/2} W^{\top} S^{\top} (1+m_{1c}(x)SS^{\top} )^{-1} SW =- \frac{1}{x^{1/2} (1+m_{1c}(x))} W^{\top} S^2 W .$$
This equality holds because $ S^2$ is a diagonal matrix with 0-1 entries. 
Now under the assumption \eqref{delocal}, we claim that 
\be\label{WSW2} W^{\top} S^2 W \to \xi_n I_k \quad \text{in probability}.\ee
Again this follows from a simple moment calculation. We can calculate that 
$$ \mathbb E \sum_{l=1}^n S_{ll}^2w_i(l)w_j(l) = \frac{r}{n} \sum_{l=1}^n w_i(l)w_j(l) = \frac{r}{n}\delta_{ij}.$$
Then we can calculate the variances: for $i\ne j$,
\begin{align*}
 \mathbb E \left(\sum_{l=1}^n S^2_{ll}w_i(l)w_j(l) \right)^2 &=\mathbb E \sum_{l=1}^n \e_{l}w_i^2(l)w_j^2(l) + \mathbb E  \sum_{l\ne l'} \e_{l}\e_{l'}w_i(l)w_i(l')w_j(l)w_j(l') \\
& \le \frac{r}{n} \sum_{l=1}^n w_i^2(l)w_j^2(l) +\left( \frac{r}{n} \right)^2 \sum_{l, l'=1}^n w_i(l)w_i(l')w_j(l)w_j(l')  \\
&=\frac{r}{n} \sum_{l=1}^n w_i^2(l)w_j^2(l)  \le \|w_i\|_\infty^2 \to 0,
\end{align*}
and 
\begin{align*}
 \mathbb E \left(\sum_{l=1}^n \left(S^2_{ll}- \frac{r}{n}\right)w_i^2(l) \right)^2 &=\frac{r}{n} \left(1-\frac{r}{n} \right)\sum_{l=1}^n w_i^4(l)\le \|w_i\|_\infty^2\to 0.
\end{align*}
Hence \eqref{WSW2} holds and we obtain that 
$$- x^{-1/2} W^{\top} S^{\top} (1+m_{1c}(x)SS^{\top} )^{-1} SW  \to - x^{-1/2} \frac{\xi_n}{1+m_{1c}} I_k= x^{1/2}m_{2c}(x) I_k$$
in probability. Hence the matrix $M(x)$ takes the same form as the one in the uniform orthogonal random projection case in Section \ref{sec unifproj}, which concludes Theorem \ref{sketchthm2} with the same arguments in Appendix \ref{pf sketch1}.
\end{proof}

\begin{proof}[Proof of Theorem \ref{sketchthm5}]
In general, we can write
	$$B=\wh S \wh S^{\top} = \diag (\wh c_1 , \cdots, \wh c_r),\quad \wh c_i \equiv 1_{c_i > 0} .$$
Let $\wh r$ be the random number of nonzero $c_i$-s, and denote 
	$$\wh \xi_n := \xi_n \left[1-\exp(-1/\xi_n)\right] .$$
	By the Poisson convergence theorem, we have 
	$$\frac{\wh r}{n} = \xi_n\left[1- P(Poisson(1/\xi_n)=0)\right] +\oo(1)=\wh\xi_n + \oo(1)$$
	in probability. Thus the self-consistent equation \eqref{separa_m12} becomes
\begin{equation}\label{self_count}
\begin{split}
& {m_{1c}(z)} = \gamma_n \frac{1}{-z\left[1+m_{2c}(z) \right]} ,\quad {m_{2c}(z)} =  \wh\xi_n  \frac{1}{-z\left[1+m_{1c}(z) \right]}+\oo(1) 
\end{split}
\end{equation}
in probability. We claim that under the delocalization condition \eqref{delocal},  
\be\label{claim countV}
W^{\top} \wh S^{\top} \wh SW = \wh\xi_n I_k+\oo(1) \quad \text{in probability}.
\ee
Suppose \eqref{claim countV} holds, then we have 
\be \nonumber
W^{\top} \wh S^{\top} (1+m_{1c}(x)\wh S\wh S^{\top} )^{-1} \wh SW = \frac{1}{1+m_{1c}(x)}W^{\top} \wh S^{\top} \wh SW = \frac{\wh\xi_n}{1+m_{1c}(x)} +\oo(1) \quad \text{in probability},
\ee
which shows that the master matrix has the following form
$$M(x)= \left( {\begin{array}{*{20}c}
   {  -x^{-1/2}\left(1+m_{2c}(x)\right)^{-1} }I_k & {D^{-1}}   \\
   {D^{-1}} & {x^{1/2} m_{2c}(x)} I_k \\
   \end{array}} \right) +\oo(1)$$
   in probability. Again $M(x)$ takes the same form as in the uniform orthogonal random projection case in Section \ref{sec unifproj}, except that we replace $\xi_n$ with $\wh \xi_n$.  Then one can conclude Theorem \ref{sketchthm5} with the same arguments in Appendix \ref{pf sketch1}.

It remains to prove the concentration claim \eqref{claim countV}. We again calculate the means and variances. 
Note that for any vector $v$,
$$(\wh Sv)(i) = 1_{c_i>0} \frac{1}{\sqrt{c_i}}\sum_{i:h(j)=i} S_{h(j),j}v(j).$$
For $1\le \al,\beta \le k$, we have
$$\left(W^{\top} \wh S^{\top} \wh SW\right)_{\al\beta}
=\sum_{i=1}^r  \frac{1_{c_i>0} }{ c_i}{\left(\sum_{j:h(j)=i} a_j w_\al(j) \right)\left(\sum_{j':h(j')=i} a_{j'}w_\beta(j') \right)}.$$
We first calculate the mean. Notice that 
the following conditional expectation can be calculated exactly as 
\be\label{exact cond}\mathbb E \left[ \left.\sum_{j:h(j)=i}  w_\al(j) w_\beta(j) \right| c_i \right] =\frac{c_i}{n}\delta_{\al\beta},\ee 
because by definition the subset $\{ j:h(j)=i\}$ is a randomly chosen subset of size $c_i$, and the vectors $w_\al$-s are orthonormal. 
Applying \eqref{exact cond}, we get
\begin{align*}
\mathbb E\left(W^{\top} \wh S^{\top} \wh SW\right)_{\al\beta}
&= \mathbb E \left\{ \sum_{i=1}^r  \frac{1_{c_i>0} }{ c_i} \mathbb E \left[ \left.\sum_{j:h(j)=i}  w_\al(j)  w_\beta(j) \right| c_i \right] \right\} = \mathbb E \left\{ \sum_{i=1}^r  \frac{1_{c_i>0} }{ c_i} \frac{c_i}{n} \delta_{\al\beta}\right\}=\left(\wh \xi_n +\oo(1)\right)\delta_{\al\beta}.
\end{align*}
 Then for $\al \ne \beta$, we have
\begin{align*}
& \mathbb E\left|\left(W^{\top} \wh S^{\top}  \wh SW\right)_{\al\beta}\right|^2 \\
&=\mathbb E \left\{ \sum_{i_1,i_2=1}^r  \frac{1_{c_{i_1}>0,c_{i_2}>0} }{ c_{i_1}c_{i_2}} \sum_{ h(j_1)=h(j_1')=i_1} \sum_{ h(j_2)=h(j_2')=i_2}  a_{j_1} a_{j_1'}a_{ j_2} a_{j_2'} w_\al(j_1)  w_\beta(j_1')w_\al(j_2)  w_\beta(j_2')  \right\} \\
& = \mathbb E \left\{ \sum_{i_1\ne i_2}   \frac{1_{c_{i_1}>0,c_{i_2}>0} }{ c_{i_1}c_{i_2}}   \sum_{j_1 :h(j_1) =i_1} w_\al(j_1)  w_\beta(j_1) \sum_{j_2 :h(j_2)= i_2}   w_\al(j_2)  w_\beta(j_2)  \right\} \\
&+ \mathbb E \left\{ \sum_{i =1}^r  \frac{1_{c_{i}>0} }{ c^2_{i}}  \sum_{j_1 \ne j_2:h(j_1) = h(j_2) =i}  2w_\al(j_1)  w_\beta(j_1) w_\al(j_2)  w_\beta(j_2)  +\sum_{i =1}^r  \frac{1_{c_{i}>0} }{ c^2_{i}}  \sum_{j_1,j_2:h(j_1)=h(j_2) =i}  w_\al(j_1)^2  w_\beta(j_2)^2  \right\} \\
& =-  \mathbb E \left\{ \sum_{i_1\ne i_2} \frac{1_{c_{i_1}>0, c_{i_2}>0} }{ c_{i_2} (n-c_{i_2})}  \left(\sum_{j_2 :h(j_2)= i_2}   w_\al(j_2)  w_\beta(j_2) \right)^2 \right\} + \oo(1) =\oo(1).
\end{align*}
Here in the third step we used \eqref{delocal} to get that 
\begin{align*}
& \sum_{i =1}^r  \frac{1_{c_{i}>0} }{ c^2_{i}}  \sum_{j_1 \ne j_2:h(j_1) = h(j_2) =i} 2 |w_\al(j_1)  w_\beta(j_1) w_\al(j_2)  w_\beta(j_2) |+\sum_{i =1}^r  \frac{1_{c_{i}>0} }{ c^2_{i}}  \sum_{j_1,j_2:h(j_1)=h(j_2) =i}  w_\al(j_1)^2  w_\beta(j_2)^2 \\
&  \le  2\|w_{\al}\|_\infty \|w_{\beta}\|_\infty  \sum_{i =1}^r  \sum_{j_1:h(j_1)=i}  |w_\al(j_1)   w_\beta(j_1) |+   \|w_{\beta}\|_\infty^2  \sum_{i =1}^r  \sum_{j_1:h(j_1)=i}    |w_\al(j_1) |^2 \\
&  \le 2\|w_{\al}\|_\infty \|w_{\beta}\|_\infty+\|w_{\beta}\|_\infty^2 =\oo(1),
\end{align*}
and a similar results as in \eqref{exact cond}, that is, given $c_{i_1}$ and the $j_2$-s such that $h(j_2)=i_2$, 
$$ \mathbb E\left[  \left. \sum_{j_1 :h(j_1) =i_1} w_\al(j_1)  w_\beta(j_1) \right| h^{-1}(i_2) \right] = \frac{c_{i_1}}{n-c_{i_2}} \sum_{j: h(j) \ne i_2}w_\al(j)  w_\beta(j) = -\frac{c_{i_1}}{n-c_{i_2}} \sum_{j : h(j)=i_2}w_\al(j)  w_\beta(j).$$ 
With similar methods, we can calculate the variance of $(W^{\top} \wh S^{\top} \wh SW)_{\al\al}$:
\begin{align*}
& \mathbb E\left|\left(W^{\top} \wh S^{\top}  \wh SW\right)_{\al\al}\right|^2 \\
&=\mathbb E \left\{ \sum_{i_1,i_2=1}^r  \frac{1_{c_{i_1}>0,c_{i_2}>0} }{ c_{i_1}c_{i_2}} \sum_{ h(j_1)=h(j_1')=i_1} \sum_{ h(j_2)=h(j_2')=i_2}  a_{j_1} a_{j_1'}a_{ j_2} a_{j_2'} w_\al(j_1)  w_\al(j_1')w_\al(j_2)  w_\al(j_2')  \right\} \\
& = \mathbb E \left\{ \sum_{i_1\ne i_2}   \frac{1_{c_{i_1}>0,c_{i_2}>0} }{ c_{i_1}c_{i_2}}   \sum_{j_1 :h(j_1) =i_1} w_\al(j_1)^2 \sum_{j_2 :h(j_2)= i_2}   w_\al(j_2)^2  \right\} \\
&+ \mathbb E \left\{ \sum_{i =1}^r  \frac{1_{c_{i}>0} }{ c^2_{i}}  \sum_{j_1 \ne j_2:h(j_1) = h(j_2) =i_1}  3w_\al(j_1)^2  w_\al(j_2)^2  + { \sum_{i =1}^r  \frac{1_{c_{i}>0} }{ c^2_{i}}  \sum_{j_1:h(j_1) =i}  w_\al(j_1)^4 }  \right\} \\
& = \mathbb E \left\{ \sum_{i_1\ne i_2}   \frac{1_{c_{i_1}>0,c_{i_2}>0} }{ c_{i_2}(n-c_{i_2})} \left( 1-\OO(c_{i_2}\|w_\al\|_\infty^2)  \right) \sum_{j_2 :h(j_2)= i_2}   w_\al(j_2)^2\right\}  + \oo(1) \\
&= \wh \xi_n \mathbb E \left\{ \sum_{  i_2}   \frac{1_{ c_{i_2}>0} }{ c_{i_2} }  \sum_{j_2 :h(j_2)= i_2}   w_\al(j_2)^2\right\}  + \oo(1) = \wh \xi_n^2 +\oo(1) = \left|\mathbb E\left(W^{\top} \wh S^{\top}  \wh SW\right)_{\al\al}\right|^2 +\oo(1),
\end{align*}
where in the third step we again used \eqref{delocal}, and that given $c_{i_1}$ and the $j_2$-s such that $h(j_2)=i_2$, 
$$ \mathbb E\left[  \left. \sum_{j_1 :h(j_1) =i_1} w_\al(j_1)^2 \right| h^{-1}(i_2) \right] = \frac{c_{i_1}}{n-c_{i_2}} \sum_{j: h(j) \ne i_2}w_\al(j)^2 = \frac{c_{i_1}}{n-c_{i_2}}\left[1- \sum_{j : h(j)=i_2}w_\al(j)^2\right].$$ 
Together with Chebyshev's inequality, this concludes the concentration result \eqref{claim countV}, which further concludes Theorem \ref{sketchthm5} . 
\end{proof}

\section{Proof of Theorem \ref{sketchthmlarge}}\label{sec_pflarge}

In the following proof, we only consider the uniform random projection. However, as we have already seen in Section \ref{sec5types}, the same result also holds for uniform random sampling under the delocalization condition \eqref{delocal},  for randomized Hadamard sampling, and for CountSketch under the delocalization condition \eqref{delocal} but with $\xi $ replaced by $\wh \xi $. 

Now we study the $i$-th spiked eigenvalue and its eigenvector under the assumption \eqref{largegap} for some large enough constant $C_0>0$. First, the self-consistent equations \eqref{separa_m12} become
\begin{equation}\label{large self}
 {m_{1c}(z)} = \gamma_n  \int\frac{x}{-z\left[1+xm_{2c}(z) \right]} \pi_{\Sigma}(dx) ,\quad
 {m_{2c}(z)} =\frac{ \xi_n }{ -z\left[1+m_{1c}(z) \right]} ,
\end{equation}
which are generalizations of \eqref{self_ortho} with $\Sigma=I$. If $\theta_i$ is the classical location for the largest eigenvalue, then we have $\theta_i \sim d_i^2 $ and 
$$0<  - m_{1,2c}(\theta_i) = - \int \frac{\dd\rho_{1,2c}(x)}{x-\theta_i}\sim d_i^{-2}.$$ 
Then for $x$ around $\theta_i$, we study the master matrix in \eqref{master_evalue} (up to an $\oo(1)$ error {in probability})
\begin{align*}
M(x) &=  \left( {\begin{array}{*{20}c}
   {  -x^{-1/2}U^{\top} \left(1+m_{2c}(x)\Sigma\right)^{-1}U } & {D^{-1}}   \\
   {D^{-1}} & {x^{1/2}m_{2c}(x)}  \\
   \end{array}} \right) \\
   &=  \left( {\begin{array}{*{20}c}
   {  -x^{-1/2} \left(1 - m_{2c}(x)E+\OO(l_i^{-2})\right)} & {D^{-1}}   \\
   {D^{-1}} & {x^{1/2}m_{2c}(x)}  \\
   \end{array}} \right) ,
\end{align*}
where recall that $E = U^{\top}\Sigma U$. Then using Schur complement formula, we obtain  
\begin{align*}
 \det  M(x) &= \det\left(- \left(1 - m_{2c}(x)E+\OO(l_i^{-2})\right) \right)\det\left(m_{2c}(x) + D^{-1} \left(1+m_{2c}(x)E + \OO(l_i^{-2})\right)D^{-1} \right) \\
    &= \det\left(- \left(1 - m_{2c}(x)E+\OO(l_i^{-2})\right) \right)\det\left(m_{2c}(x) D^{-2} \right)\det\left((D^2+E)+\OO(l_i^{-1})+m_{2c}^{-1}(x)\right) .
\end{align*}
By standard results from perturbation theory \citep[e.g.,][]{stewart1990matrix}, we know that the first order perturbation of the $i$-th eigenvalue of $D^2+E$ is given by $d_i^2 +  E_{ii} + \OO(l_i^{-1})$. Hence, by solving $\det\left((D^2+E)+m_{2c}^{-1}(x)\right)$  for $\theta_i=x$, we get
\be\label{outlierevaluelarge}
\theta_i = g_{2c}\left( - \frac{1}{d^2_i +E_{ii}  } + \OO(l_i^{-3})\right)= g_{2c}\left( - \frac{1}{d^2_i + E_{ii} } \right) + \OO(l_i^{-1})
\ee
in probability, where we recall that $g_{2c}$ is the inverse function of $m_{2c}$. 

Then we consider the corresponding eigenvectors: using \eqref{ev_origin} and Schur complement, we get  (up to an $\oo(1)$ error {in probability})
\begin{align*} 
|\langle u_j, \wt{\bxi}_i\rangle|^2 &= \frac{1}{2\pi \ii \sqrt{\theta_i}} \oint_{\Gamma_i} e_j^{\top} D^{-1}\left(z^{1/2}m_{2c}(z) + z^{1/2} D^{-1}\left(1+m_{2c}(z)E + \OO(l_i^{-2})\right)D^{-1}\right)^{-1}  D^{-1}e_j \dd z \\
&=  \frac{1}{2\pi \ii \theta_1} \oint_{\Gamma_i} e_j^{\top} \frac{1}{m_{2c}(z) (D^2+E + \OO(l_i^{-2})) + 1} e_j \dd z .
\end{align*}
Again, standard perturbation theory \citep[e.g.,][]{stewart1990matrix} tells us that the eigenvector of $D^2+E$ up to the first order perturbation is given by 
$$ e_i + \sum_{j\ne i, 1\le j \le k} e_j \frac{E_{ji}}{d_i^2 - d_j^2}. $$
Thus we get that in probability, 
\begin{align} \label{outlierevectorlarge1}
|\langle u_i, \wt{\bxi}_i\rangle|^2 =  \frac{g_{2c}'\left( - \al_i\right)\al_i}{g_{2c}( -\al_i) } + \OO(l_i^{-2}) , \quad \al_i:=  ({d^2_i + E_{ii}+ \OO(l_i^{-1})  })^{-1},
\end{align}
and
\begin{align}\label{outlierevectorlarge2} 
|\langle u_j, \wt{\bxi}_i\rangle|^2 =  \frac{g_{2c}'\left( - \al_i\right)\al_i}{g_{2c}( -\al_i) } \left|  \frac{E_{ji}}{d^2_i - d^2_j} \right|^2+ \OO(l_i^{-3}) , \quad j\ne i. 
\end{align}

It remains to study the expression for $g_{2c}$. From \eqref{large self}, we obtain that
$$ z = \gamma_n  \int\frac{x}{ 1+xm_{2c}(z) } \pi_{\Sigma}(dx) - \frac{\xi_n}{m_{2c}(z)}\Rightarrow g_{2c}(m)=\gamma_n  \int\frac{x}{ 1+xm } \pi_{\Sigma}(dx) - \frac{\xi_n}{m }.$$
Then we can calculate that
\be\label{thetailarge}\theta_i = g_{2c}\left( - \frac{1}{d^2_i + E_{ii} } \right) + \OO(l_i^{-1}) = \xi_n (d_i^2 +  E_{ii}) + \gamma_n \rho_1+ \OO(l_i^{-1}) ,\ee
and
\begin{align}
  \frac{g_{2c}'\left( - \al_i\right)\al_i}{g_{2c}( -\al_i) } &= \frac{ \xi_n - \gamma_n  \int\frac{x^2\al_i^2}{ (1-x\al_i)^2 } \pi_{\Sigma}(dx) }{ \xi_n + \gamma_n  \int\frac{x\al_i}{ 1-x\al_i } \pi_{\Sigma}(dx) } = \frac{\xi_n - \gamma_n \left( \al_i^2 \rho_2 \right)}{\xi_n + \gamma_n(\al_i \rho_1 + \al_i^2 \rho_2)} + \OO (l_i^{-3}) \nonumber\\
&  =  \frac{\xi_n -\frac{\gamma_n}{d_i^{4}}  \rho_2 }{\xi_n + \frac{\gamma_n}{d_i^2} \left( \rho_1 + d_i^{-2} (\rho_2 - \rho_1 E_{ii}) \right)} + \OO (l_i^{-3}), \label{coslarge}
  \end{align}
  where $\rho_i$ are the moments of the spectral distribution of $\Sigma$ as in \eqref{rhoisigma}.
 
 With the above calculations, the rest of the proof is exactly the same as the one for Theorem \ref{sketchthm1} in Appendix \ref{pf sketch1}. We omit the details. 

\section{Proof of Theorem \ref{lem_localout} and Theorem \ref{lem_localin}}\label{sec_pflocallaw}

In this section, we provide the necessary details to complete the proof of Theorem \ref{lem_localout} and Theorem \ref{lem_localin} based on the results in \cite{yang2019spiked}  and \cite{ding2019spiked}.

We use a standard cutoff argument. We choose the constant $\delta>0$ small enough such that $\left(n^{1/2-\delta}\right)^{4+\tau} \ge n^{2+\delta}$. Then we introduce the following truncation  
\be\label{defwhX}\wh X :=1_{\Omega} X, \quad \Omega :=\left\{\max_{i,j} |x_{ij}|\le n^{-\delta} \right\}.\ee
By the moment conditions \eqref{eq_highmoment} and a simple union bound, we have
\begin{equation}\label{XneX}
\mathbb P(\wh X \ne X ) =\OO ( n^{-\delta}).
\end{equation}
Using (\ref{eq_highmoment}) and integration by parts, it is easy to verify that 
\begin{align*}
\mathbb E  \left|x_{ij}\right|1_{|x_{ij}|> n^{-\delta}} =\OO(n^{-2-\delta}), \quad \mathbb E \left|x_{ij}\right|^2 1_{|x_{ij}|> n^{-\delta}} =\OO(n^{-2-\delta}),
\end{align*}
which imply that
\be\label{meanvartrunc}|\mathbb E  \wh x_{ij}| =\OO(n^{-2-\delta}), \quad  \mathbb E |\wh x_{ij}|^2 = n^{-1} + \OO(n^{-2-\delta}).\ee
Moreover, we trivially have
$$\mathbb E  |\wh  x_{ij}|^4 \le \mathbb E  |x_{ij}|^4 =\OO(n^{-2}).$$
We define the following centered version of $\wh X$: $Z = \wh X - \mathbb E\wh X$. Then we have the following proposition for the resolvent $G(W,z)$.  
\begin{proposition}
Suppose the assumptions of Theorem \ref{lem_localin} hold and define $Z$ as above. Then we have
\begin{equation}\label{aniso_outappd}
\max_{u,v\in \mathscr A} \left| \langle u, G(Z,z) v\rangle - \langle u, \Pi (z)v\rangle \right|  \prec n^{-\delta} 
\end{equation}
uniformly in $z\in S_{edge}(c_0,C_0,c_1)$. Moreover, we have that for any fixed $\varpi \in \N$, 
\begin{equation} \label{rigidityev}
\max_{1\le i \le \varpi}|\lambda_i - \lambda_+| \prec n^{-\delta},
\end{equation}
where $\lambda_i$ denotes the $i$-th largest eigenvalue of $\cal Q_1(Z):= (SZ \Sigma^{1/2})^{\top}SZ \Sigma^{1/2}$.
\end{proposition}
\begin{proof}
The estimates \eqref{aniso_outappd} and \eqref{rigidityev} has essentially been proved in Theorem 3.6 and Theorem 3.8 of \cite{yang2019spiked}, respectively. The only difference is that in \cite{yang2019spiked}, the entries of the random matrix $Z$ all have variances $n^{-1/2}$, while in the current case we have
$$  \mathbb E |Z_{ij}|^2 = n^{-1} + \OO(n^{-2-\delta}) . $$
However, one can check that the error $\OO(n^{-2-\delta})$ is sufficiently small such that it is negligible at each step of the proof in \cite{yang2019spiked}, which concludes \eqref{aniso_outappd} and \eqref{rigidityev}. We remark that the first author actually proved stronger results in Theorem 3.6 and Theorem 3.8 of \cite{yang2019spiked}, but they are not necessary for our purpose in this paper.
\end{proof}

Then we show that $G(\wh X, z)$ is sufficiently close to $G(Z,z)$ in the sense of anisotropic local law.
\begin{proposition} \label{lem_localoutapp2}  
\eqref{aniso_outappd} holds uniformly for $G(\wh X,z)$ in $z\in S_{edge}(c_0,C_0,c_1)$. Moreover, \eqref{rigidityev} holds with high probability for $\lambda_i(\cal Q_1(\wh X))$. 
\end{proposition}
\begin{proof}
We write $\wh X = Z+ \mathbb E\wh X,$ where by \eqref{meanvartrunc}, we have
\begin{equation}\label{boundB}
\max_{i,\mu}|\mathbb E\wh X_{i\mu}|= \OO(n^{-2-\delta}).
\end{equation}
In particular, this gives that $\|\mathbb E\wh X\| \le \|\mathbb E\wh X\|_F= \OO(n^{-1-\delta}) $, which implies \eqref{rigidityev} for $\lambda_i(\cal Q_1(\wh X))$. 

For \eqref{aniso_outappd}, we abbreviate $G(\wh X,z)\equiv \wh G$ and $G(Z,z)\equiv G$. Then it suffices to show that for any deterministic unit vectors $u, v $,
\begin{equation}\label{aniso_central}
\left| \langle u, \wh G(z)  v\rangle - \langle u, G(z) v\rangle \right| \prec n^{-\delta}
\end{equation}
uniformly in $z\in S_{edge}(c_0,C_0,c_1)$. We can write 
\begin{equation}\nonumber
 \wh G(z) = \left(G^{-1}(z) + V\right)^{-1}, \quad V := z^{1/2}\left( {\begin{array}{*{20}c}
   {0} & \Sigma^{1/2} (\mathbb E\wh X)^{\top} S^{\top}  \\
   S(\mathbb E\wh X) \Sigma^{1/2} & {0}  \\
   \end{array}} \right).
 \end{equation}
Then we expand $G$ using the resolvent expansion
\begin{equation}\label{rsexp1}
\wh G = G - G V \wh G.  
\end{equation}
Using the spectral decomposition for $\wh G$ as in \eqref{spectral1}-\eqref{spectral2}, one can easily see that the following deterministic bound holds:
$$ \|\wh G(z)\| = \OO(\eta^{-1}).$$
Then we can estimate the second part on the right-hand side of \eqref{rsexp1} as
\begin{align*} 
|\langle u, G V \wh Gv\rangle | \lesssim \eta^{-1}\left(\sum_{a\in \cal I} |\langle u, G V e_a\rangle |^2\right)^{-1/2} \prec \eta^{-1}\left(\sum_{a\in \cal I} \sum_{b\in \cal I} |V_{ba}|^2 \right)^{-1/2} \lesssim n^{-1/2-\delta}. 
\end{align*}
where $e_a$ denotes the standard unit vector along $a$-th direction, and in the second step we applied \eqref{aniso_outappd} to $\langle u, G V e_a\rangle = G_{u w}$ by taking $w:=Ve_a$, and in the third step we used that $\|V\|_{F} \lesssim \|\mathbb E\wh X\|_F \lesssim n^{-1-\delta}$. This concludes \eqref{aniso_outappd} for $G(\wh X,z)$
\end{proof}

Finally we show that \eqref{aniso_outstrong} holds for $G(\wh X,z)$ for $z$ with imaginary part down to the real axis in the spectral domain $S_{out}(c_0,C_0)$. 

\begin{proposition} \label{lem_localoutapp}  
\eqref{aniso_outappd} holds uniformly for $G(\wh X,z)$ in $z\in S_{out}(c_0,C_0)$.
\end{proposition}
\begin{proof}
In this proof, we abbreviate $G(\wh X,z)\equiv G$ and still use the notation $\wt X= S\wh X \Sigma^{1/2}$ as in \eqref{sketchnonspike}. It remains to show that \eqref{aniso_outappd} holds for $z= E+\ii\eta\in S_{out}(c_0,C_0)$ with $\eta\le n^{-1/2+c_1}$. We denote $\eta_0:=n^{-1/2+c_1}$ and $z_0:= E+ \ii \eta_0$. With \eqref{aniso_outappd} at $z_0$, it suffices to prove that
\be\label{prof_m}
\langle u, \left(\Pi (z)-\Pi(z_0)\right)v\rangle \prec n^{-1/2+c_1} ,
\ee
and 
\be\label{prof_G}
\langle u, \left(G (z)-G(z_0)\right)v\rangle \prec n^{-\delta} .
\ee
With \eqref{Piii}, to prove \eqref{prof_m} it is enough to show that 
\be\label{prof_m2}
|m_{1c}(z)-m_{1c}(z_0)|+|m_{2c}(z)-m_{2c}(z_0)| \prec n^{-1/2+c_1} ,
\ee
which follows immediately from \eqref{eq_mdiff}.

For \eqref{prof_G}, we write $u=\begin{pmatrix} u_1 \\ u_2\end{pmatrix}$ and $v=\begin{pmatrix} v_1 \\ v_2\end{pmatrix}$, and in the following proof, we will always identify vectors $v_1$ and $v_2$ with their embeddings $\begin{pmatrix} v_1 \\ 0\end{pmatrix}$ and $\begin{pmatrix} 0 \\ v_2\end{pmatrix}$, respectively. Let
$$ \wt X  = \sum_{k = 1}^{p \wedge r} \sqrt {\lambda_k} {\zeta} _{k} { \xi}_k^{\top} ,$$
be the singular value decomposition of $\wt X $. We shall use \eqref{spectral1}-\eqref{spectral2} with $\wt G$ replaced by $G$.  First, the upper left block gives that 
\begin{align}\label{zz0}
\left|\langle u_1, \left(G(z) - G(z_0)\right) v_1\rangle\right| \le \sum_{k = 1}^{p} \frac{\eta_0  |\langle v_1,{ \xi}_k\rangle|^2}{\left[(E-\lambda_k)^2 + \eta^2 \right]^{1/2}\left[(E-\lambda_k)^2 + \eta_0^2 \right]^{1/2}}.
\end{align}
  By \eqref{rigidityev}, we have for any $k$, $E-\lambda_k\ge E - \lambda_1 \ge c_0/2\gg \eta_0$ with high probability for $z\in S_{out}(c_0,C_0)$. Hence we can bound \eqref{zz0} by
\begin{align*}
& \left|\langle u_1, \left(G(z) - G(z_0)\right) v_1\rangle\right| \lesssim \sum_{k = 1}^{p} \frac{\eta_0  |\langle u_1,{ \xi}_k\rangle|^2}{(E-\lambda_k)^2 + \eta_0^2}+\sum_{k = 1}^{p} \frac{\eta_0  |\langle v_1,{ \xi}_k\rangle|^2}{(E-\lambda_k)^2 + \eta_0^2}  \\
& = \im \langle u_1,\sum_{k = 1}^{p} \frac{ { \xi}_k{ \xi}_k^{\top} }{E-z_0} u_1\rangle + \im \langle v_1,\sum_{k = 1}^{p} \frac{ { \xi}_k{ \xi}_k^{\top} }{E-z_0} v_1\rangle = \im  G_{u_1u_1} (z_0)+\im  G_{v_1v_1} (z_0) \\
& \prec n^{-\delta} + \im \Pi_{u_1u_1}(z_0)+ \im \Pi_{v_1v_1}(z_0)  \prec n^{-\delta},
\end{align*}
where in the fourth step we used \eqref{aniso_outappd} for $G(z_0)$, and in the last step we used \eqref{defn_pi}, \eqref{Piii} and \eqref{eq_estimm} to get
$$\im \Pi_{u_1u_1}(z_0)+\im \Pi_{v_1v_1}(z_0) \prec  \eta_0.$$
Similarly, for the upper right block we have
\begin{align*}
& \left|\langle u_1, \left(G(z) - G(z_0)\right) v_2\rangle\right| \prec  \left|z^{-1/2}- (z_0z^{-1})^{1/2}\right| \left|\langle u_1, G(z_0)v_2\rangle\right| + \sum_{k = 1}^{p\wedge n} \frac{\eta_0\left|\langle u_1 ,{ \xi}_k\rangle \langle { \zeta}_k,v_2\rangle\right|}{|\lambda_k-z||\lambda_k-z_0|}\\
& \prec \eta_0 + \sum_{k = 1}^{p\wedge n} \frac{\eta_0\left|\langle u_1 ,{ \xi}_k\rangle \right|^2 }{|\lambda_k-z_0|^2} +   \sum_{k = 1}^{p\wedge n} \frac{\eta_0\left| v_2,{ \zeta}_k\rangle\right|^2}{|\lambda_k-z_0|^2}=\eta_0+ \im G_{u_1u_1}(z_0) +\im G_{v_2v_2}(z_0) \prec  n^{-\delta}.
\end{align*}
The lower left and lower right blocks can be handled in the same way. This proves \eqref{prof_G}, which completes the proof. 
\end{proof}

Finally, with Proposition \ref{lem_localoutapp2}, Proposition \ref{lem_localoutapp}, and the definition of the truncation \eqref{defwhX}, we conclude Theorem \ref{lem_localout} and Theorem \ref{lem_localin}.

\section{Extension to centered model}\label{sec_centermodel}

In this section, we explain how to extend our results to centered sample covariance matrices. We shall only consider the following model
$$
\wt Y_a =\wt X _a + (I-ee^{\top})\sum_{i=1}^k d_i v_i u_i^{\top},\quad \wt X _a=(I-ee^{\top}) S  X  \Sigma^{1/2}  .
$$
However, the other model
$$\wt Y_b := S (I-ee^{\top})  X  \Sigma^{1/2} + \sum_{i=1}^k d_i S (I-ee^{\top})w_i u_i^{\top}$$
can be studied with exactly the same method. 

Our goal is to study the principal components of $\wt{\cal Q}_1^a:= \wt Y_a^{\top} \wt Y_a$ using the methods in Section \ref{sec method}. Then we have the following claim. 
\begin{claim}\label{perturbclaim}
As long as we have
\be\label{inprobS}
\max_i |e^{\top} Sw_i | =\oo(1) ,
\ee 
and uniformly in $z\in S_{out}(c_0,C_0)$,
\be\label{inprobS2}
\max_i \left| e^{\top} \frac{1}{1+m_{1c}(z)SS^{\top}}Sw_i \right| =\oo(1) , 
\ee 
then the spike eigenvalues and eigenvectors of $\wt{\cal Q}_1^a$ have the same asymptotic behavior as those of $\wt{\cal Q}_1 =\wt Y^{\top} \wt Y$. 
\end{claim}
\begin{proof}
Note that under \eqref{inprobS}, we have 
$$\|ee^{\top} \sum_{i=1}^k d_i v_i u_i^{\top}\|=\|ee^{\top}S \sum_{i=1}^k d_i w_i u_i^{\top}\| =\oo(1). $$
Then by \eqref{master_evalue} and \eqref{ev_origin}, it suffices to show that the same local law holds for $G_a(z)$: with high probability in $\Omega$, 
\begin{equation}\label{aniso_outstrong2222}
\max_{1\le i,j \le k} \left| \langle Sw_i , \left[G_a(z)-\Pi (z)\right] S w_j\rangle \right|  =\oo(1) 
\end{equation}
uniformly in $z\in S_{out}(c_0,C_0)$,  where
$$G_a(z):=  \left[z^{1/2} \left( {\begin{array}{*{20}c}
   { 0 } & {\wt X _a^{\top}}   \\
   {\wt X _a} & {0}  \\
   \end{array}} \right) - z\right]^{-1}.$$

We denote $S_a:=(I-ee^{\top}) S $, and $B_a:= S_aS_a^{\top}$ with eigenvalues $s_1^a \ge s_2^a \ge \ldots \ge s_r^a \ge 0$. Then we can define $m_{1c}^a$ and $m_{2c}^a$ using the self-consistent equation \eqref{separa_m12} by replacing $\pi_B$ with $\pi_{B_a}$, and define $\Pi_a(z)$ by replacing $B$ with $B_a$, and $m_{1,2c}$ with $m_{1,2c}^a$. We claim that 
\be\label{stab_m12}|m_{1c}^a-m_{1c}|+|m_{2c}^a-m_{2c}|=\OO(n^{-1}) \ee
uniformly in $z\in S_{out}(c_0,C_0)$. We postpone the proof of \eqref{stab_m12} until we complete the proof of Claim \ref{perturbclaim}. Now using $s_1^a \le s_1$ we get that \eqref{assm_gap} holds for $m_{1,2c}^a$ and $\sigma_1, s_1^a$. Thus by Theorem \ref{lem_localout}, we have that with high probability in $\Omega$, 
\begin{equation} \nonumber
\max_{u,v\in \mathscr A} \left| \langle u, G_a(X,z) v\rangle - \langle u, \Pi_a (z)v\rangle \right|  =\oo(1) 
\end{equation}
uniformly in $z\in S_{out}(c_0,C_0)$. Hence to show \eqref{aniso_outstrong2222}, it suffices to prove that
\be\label{perturb111}\max_{1\le i,j \le k} \left| \langle Sw_i , \left[\Pi_a(z)-\Pi (z)\right] S w_j\rangle \right|  =\oo(1)  .\ee
Using \eqref{inprobS}, \eqref{inprobS2} and \eqref{stab_m12}, we obtain that
\begin{align*} 
& \frac{-z}{m_{1c}(z)}w_i^{\top} S^{\top} \left[\Pi_a(z)-\Pi (z)\right] S w_j \\
& =w_i^{\top} S^{\top} \left[\frac{1}{1+m_{1c}(1-ee^{\top})B(1-ee^{\top})  }\left(Bee^{\top} + ee^{\top}B +  ee^{\top}Bee^{\top}\right)\frac{1}{1+m_{1c}B  }\right] S w_j \\
& = \oo(1) + w_i^{\top} S^{\top} \left[\frac{1}{1+m_{1c}(1-ee^{\top})B(1-ee^{\top})  }\left(  ee^{\top}B \right)\frac{1}{1+m_{1c}B  }\right] S w_j \\
& = \oo(1) + w_i^{\top} S^{\top} \left[\frac{1}{1+m_{1c}(1-ee^{\top})B(1-ee^{\top})  }m_{1c}^{-1} ee^{\top}\left(1-  \frac{1}{1+m_{1c}B  } \right)\right] S w_j =\oo(1).
\end{align*}
This concludes \eqref{perturb111}.
\end{proof}

\begin{proof}[Proof of \eqref{stab_m12}]
We claim that approximately, $m_{1,2c}$ satisfy the self-consistent equations for $m_{1,2c}^a$:
\begin{equation} \label{separa_m12appd}
\begin{split}
& {m_{1c}(z)} = \frac1n\sum_{i = 1}^p \frac{\sigma_i}{-z(1+\sigma_i m_{2c})},\quad {m_{2c}(z)} = \frac{1}{n} \sum_{\mu= 1}^r \frac{s_\mu^a}{-z(1+s_\mu^a m_{1c})} +\OO(n^{-1}).
\end{split}
\end{equation}
Then \eqref{stab_m12} follows from Lemma 5.11 of \cite{yang2019spiked}, which gives the stability of the self-consistent equations. Roughly speaking, stability means that if $(u_{1},u_2)$ are satisfy the self-consistent equation for $(m_{1c}^a,m_{2c}^a)$ up to some sufficiently small error $\epsilon$, then we also have
$$ |u_1(z) - m_{1c}(z)| +  |u_2(z) - m_{2c}(z)| \lesssim \epsilon$$
uniformly in $z\in S_{out}(c_0,C_0)$. 

It remains to prove \eqref{separa_m12appd}. The first equation is the first one in \eqref{separa_m12}, while for the second equation, we claim that
\begin{equation}\label{separa_m12appd2}
\left|\frac{1}{n} \sum_{\mu= 1}^r \frac{s_\mu}{-z(1+s_\mu m_{1c})}- \frac{1}{n} \sum_{\mu= 1}^r \frac{s^a_\mu}{-z(1+s^a_\mu m_{1c})}\right|=\OO(n^{-1}).
\ee
For the imaginary part, we have
\begin{align*}
\im \left(\frac{1}{n} \sum_{\mu= 1}^r \frac{s_\mu}{ 1+s_\mu m_{1c}}- \frac{1}{n} \sum_{\mu= 1}^r \frac{s^a_\mu}{ 1+s^a_\mu m_{1c}}\right)= - \frac{\im m_{1c}}{n} \left(  \sum_{\mu= 1}^r \frac{1 }{ | s_\mu^{-1} + m_{1c}|^2} -  \sum_{\mu= 1}^r \frac{1}{ |(s^a_\mu)^{-1}+ m_{1c}|^2 }\right).
\end{align*}
By the Stieltjes transform 
$$m_{1c}(z)=\int_0^{\lambda_+} \frac{\dd\rho_{1c}(x)}{x-z},$$
we obtain that for $z= E+ \ii\eta\in S_{out}(c_0,C_0)$, 
$$\im m_{1c}(z) \ge 0,\quad 0> \re m_{1c}(E) \ge m_{1c}(\lambda_+) \ge - s_1^{-1} \ge - (s_1^a)^{-1}.$$ 
Hence the function $| x + m_{1c}|^2$ is increasing in $x\in (s_1^{-1},\infty)$. Using the Cauchy interlacing 
$$s_r^a \le s_r \le \cdots \le s^a_2 \le s_2 \le s^a_1 \le s_1,$$
we get that 
$$ \sum_{\mu= 1}^r \frac{1 }{ | s_\mu^{-1} + m_{1c}|^2} -  \sum_{\mu= 1}^r \frac{1}{ |(s^a_\mu)^{-1}+ m_{1c}|^2 }=\OO(1).$$
Hence we obtain that 
$$\im \left(\frac{1}{n} \sum_{\mu= 1}^r \frac{s_\mu}{ 1+s_\mu m_{1c}}- \frac{1}{n} \sum_{\mu= 1}^r \frac{s^a_\mu}{ 1+s^a_\mu m_{1c}}\right)=\OO(n^{-1}).$$
Together with a similar estimate for the real part, we get \eqref{separa_m12appd2}, which concludes \eqref{separa_m12appd}.
\end{proof}

Finally, we show that \eqref{inprobS} and \eqref{inprobS2} holds (at least in probability) for all the sketching methods we used in this paper. 
 \begin{itemize}
 \item[(1)] {\bf Uniform random projection in Section \ref{sec unifproj}:} In this case $SS^{\top}=I_r$, so that we only need to check that \eqref{inprobS} holds. By the rotational invariance of $S$, we have 
 $$ e^{\top} Sw_i \stackrel{d}{=} \frac{1}{\sqrt{r}}\sum_{i=1}^r \wh S_{i1},$$
 where $\wh S$ is also an $r\times n$ uniform random projection matrix. By exchangeability, we have
   $$\frac1r\mathbb E\Big| \sum_{i=1}^r \wh S_{i1} \Big|^2 = \frac{1}{r}\mathbb E \sum_{i,j=1}^r\wh S_{i1}\wh S_{j1} = \frac{1}{nr}\mathbb E \sum_{i,j=1}^r \sum_{\mu=1}^n\wh S_{i\mu}\wh S_{j\mu} = \frac1n.$$
   Hence we have $ e^{\top} Sw_i\to 0$ in probability.
   
    \item[(2)] {\bf  i.i.d. projection in Section \ref{sec Gauss}:} Note that $x_k := \sum_{l=1}^n S_{kl}w_i(l)$ are i.i.d. random variables with mean zero and variance $n^{-1}$. Hence by LLN, we have
  $$ e^{\top} Sw_i = \frac{1}{\sqrt{r}}\sum_{k=1}^r x_k \to 0 \quad \text{a.s.}$$
  For the estimate \eqref{inprobS2}, we use the local law, Theorem \ref{lem_localout}. If we take $Y= S$, then 
  $$ \frac{1}{m_{1c}(z)^{-1}+SS^{\top}}S$$
  is (proportional to) the lower left block of $G(-m_{1c}(z)^{-1})$ in \eqref{green2}, and the local law \eqref{aniso_outstrong} gives that \eqref{inprobS2} holds with high probability. If one is worried about that $-m_{1c}(z)^{-1}$ may not be in the domain given in Theorem \ref{lem_localout}, we remark that a local law in \cite{isotropic} for the $Y=S$ case was proved on a more general domain as in \eqref{localmS0}.
 
  \item[(3)] {\bf Random sampling in Section \ref{sec unifsample}:} 
{  In this case, the leading principal components of the centered model $\wt{\cal Q}_1^a= \wt Y_a^{\top} \wt Y_a$ will behave differently from those of the model $\wt{\cal Q}_1=\wt Y\wt Y^\top$ under the sketching \eqref{defnSsampling}. However, we can consider a slightly different random sampling $\wt S$ with random signs: 
\be\label{defnSsampling_app}
  \wt S_{ii} = \epsilon_i a_i,\ee
where $a_i$ is a Rademacher random variable uniform on $\{-1,1\}$ independent of $\e_i$. In this case, we have that $Y^\top \wt S^\top \wt SY= Y^\top S^\top SY$, hence Theorem \ref{sketchthm2} still holds for $\wt{\cal Q}_1$ under the sketching \eqref{defnSsampling_app}. On the other hand, we now check that \eqref{inprobS} and \eqref{inprobS2} hold for $\wt S$, so that Theorem \ref{sketchthm2} also holds for the centered model $\wt{\cal Q}_1^a$ under the sketching \eqref{defnSsampling_app}.} Note that we have 
 \be\label{simpleappd}  \frac{1}{1+m_{1c}(z)\wt S\wt S^{\top}}\wt S = \frac{1}{1+m_{1c}(z) }\wt S. \ee
  Hence again we only need to check that \eqref{inprobS} holds. We can calculate that
   $$\mathbb E| e^{\top} \wt Sw_i|^2 =\frac{1}{n}\mathbb E\Big|\sum_{l=1}^n \wt S_{ll}w_i(l)\Big|^2 = \frac{1}{n} \frac{r}{n}\sum_{l=1}^n |w_i(l)|^2 \le  \frac1n.$$
   Hence we have $ e^{\top}\wt Sw_i\to 0$ in probability.  
   
 \item[(4)] {\bf Randomized Hadamard sampling in Section \ref{sec hadamard}:} In this case, we also have \eqref{simpleappd} and hence we only need to check that \eqref{inprobS} holds. We calculate that 
  $$\mathbb E| e^{\top} Sw_i|^2 =\mathbb E| e^{\top} B_rz_i|^2 =\frac{1}{n}\mathbb E\Big|\sum_{l=1}^n (B_r)_{ll}z_i(l)\Big|^2 .$$
 For $l\ne l'$, we have
  \begin{align*}
  \mathbb E (z_i(l)z_i(l') )= \sum_{j,j'} \mathbb E[a^{(l)}_ja^{(l')}_j] w_i(j)w_i(j') = 0,
  \end{align*}
which gives that
  $$\mathbb E| e^{\top} Sw_i|^2 =\frac{1}{n}\frac{r}{n}\sum_{l=1}^n  \mathbb E |z_i(l)|^2 \le  \frac1n.$$
   Hence we have $ e^{\top} Sw_i\to 0$ in probability.

  \item[(5)] {\bf    CountSketch in Section \ref{sec count}:}  In this case, we also have 
  $$ \frac{1}{1+m_{1c}(z)\wh S\wh S^{\top}}\wh S = \frac{1}{1+m_{1c}(z) }\wh S$$
  and hence we only need to check that \eqref{inprobS} holds. Again we calculate the second moment of
  $$ e^{\top} \wh Sw_i  {=} \frac{1}{\sqrt{n}}\sum_{i=1}^n \sum_{\mu: h(\mu)=i} \wh S_{i\mu} w_i(\mu).$$
For $i\ne i'$, we have that 
\begin{align*}
\mathbb E \sum_{\mu,\mu': h(\mu)=i,h(\mu')=i'} \wh S_{i\mu}\wh S_{i'\mu'} w_{i}(\mu) w_{i'}(\mu')=0,
\end{align*}
which leads to
\begin{align*} 
\mathbb E|e^{\top} \wh Sw_i |^2 & = \frac{1}{n}\sum_{i=1}^n \mathbb E\sum_{\mu,\mu': h(\mu)=h(\mu')=i} \wh S_{i\mu} \wh S_{i\mu'} w_i(\mu)w_i(\mu') = \frac{1}{n}\sum_{i=1}^n  \frac{1(c_i\ne 0 )}{c_i} \mathbb E\sum_{\mu : h(\mu)= i} w_i(\mu)^2 \\
 & =\frac{1}{n}\sum_{i=1}^n  \frac{1(c_i\ne 0 )}{c_i}\frac{c_i}{n} \sum_{\mu } w_i(\mu)^2 \le \frac1n,
  \end{align*}
  where we used the exchangeability in the third step. Hence we have $ e^{\top} Sw_i\to 0$ in probability.
 \end{itemize}




{\small 
\setlength{\bibsep}{0.2pt plus 0.3ex}

}

\end{document}